%% file: paper_1.tex
\documentclass[11pt]{article}

\usepackage[utf8]{inputenc}
\usepackage[english]{babel}
\usepackage{tabularx}

\usepackage{graphicx}
\usepackage{amsmath}
\usepackage{amsthm}
\usepackage{amssymb}
\usepackage{amsfonts}
\usepackage{amstext}
\usepackage{rotating}
\usepackage{latexsym}
\usepackage{epsfig}
\usepackage{pdfpages}%pour pouvoir utiliser \includepdf
\usepackage[margin=2.5cm]{geometry}
\usepackage{dsfont}
\usepackage{xcolor}
\usepackage{stmaryrd}
\usepackage{tikz}
\usepackage{comment}
\usepackage{hyperref}
\usepackage{mathtools}
\usepackage[autostyle]{csquotes}
\usepackage{lineno}

\newtheorem{them}{Theorem}
\newtheorem{pro}[them]{Proposition}

\newtheorem{lem}[them]{Lemma}

\theoremstyle{definition}
\newtheorem{defi}[them]{Definition}

\newtheorem{remq}[them]{Remark}
\newtheorem{ex}[them]{Example}
\newtheorem{nott}[them]{Notation}

\numberwithin{them}{section}

\newcommand{\dcv}[1]{\xrightarrow[#1]{}}

\newcommand{\At}{\textrm{Atom}}
\newcommand{\dd}{\ \mathrm{d}}
\newcommand{\ens}[2]{\left\{ #1 ~;~#2 \right\} }

\newcommand{\g}{\gamma}

\renewcommand{\d}{\delta}
\newcommand{\s}{\sigma}

\newcommand{\ddd}{\mathrm{d}}

\newcommand{\F}{\mathrm{F}}
\newcommand{\D}{\mathrm{D}}
\renewcommand{\H}{\mathrm{H}}
\newcommand{\G}{\mathrm{G}}
\newcommand{\Cd}{\mathrm{C}}

\newcommand{\Q}{\mathbb{Q}}
\newcommand{\R}{\mathbb{R}}
\newcommand{\N}{\mathbb{N}}

\newcommand{\Z}{\mathbb{Z}}

\newcommand{\II}{\mathcal{I}}
\newcommand{\FF}{\mathcal{F}}
\newcommand{\CC}{\mathcal{C}}
\newcommand{\BB}{\mathcal{B}}
\newcommand{\KK}{\mathcal{K}}
\newcommand{\DD}{\mathcal{D}}

\newcommand{\TT}{\mathcal{T}}
\newcommand{\PP}{\mathcal{P}}

\newcommand{\LL}{\mathcal{L}}
\newcommand{\OO}{\mathcal{O}}

\newcommand{\MM}{\mathcal{M}}

\newcommand{\CCC}{\mathfrak{C}}

\newcommand{\OOO}{\mathfrak{O}}
\newcommand{\PPP}{\mathfrak{P}}

\newcommand{\RRR}{\mathfrak{R}}

\newcommand{\ee}{\mathbb{\varepsilon}}
\newcommand{\1}{\mathds{1}}

\newcommand{\ie}{\textnormal{i.e.}}
\newcommand{\id}{\textnormal{id}}

\newcommand{\Marg}{\textnormal{Marg}}
\newcommand{\Mart}{\textnormal{Mart}}
\newcommand{\argmin}{\textnormal{argmin}}
\newcommand{\Ent}{\textnormal{Ent}}
\newcounter{numeroexo}
\definecolor{darkgreen}{RGB}{0,100,0} % Vert foncé type "forest green"

\newcommand{\ti}[1]{\tilde{#1}}

\newcommand{\st}{\textnormal{st}}

\newcommand{\n}[2]{\left\|#1\right\|_{#2}}
\newcommand{\spt}{\textnormal{spt}}
\newcommand{\mres}{\mathbin{\vrule height 1.6ex depth 0pt width
		0.13ex\vrule height 0.13ex depth 0pt width 1.3ex}}
\newcommand{\res}{\hspace{0.072cm}\mres}
\DeclareMathSymbol{\Eta}{\mathalpha}{letters}{"48}

\newcommand{\accol}[1]{\left\{\begin{array}{l}#1\end{array}\right .}
\newcommand{\ent}[2]{\llbracket #1,#2\rrbracket}

\usepackage{scalerel} % Pour ajuster la taille
\usepackage{stackengine} % Pour manipuler les indices

\newtheorem{introthm}{Theorem} % Pas de lien avec section
\newtheorem*{introthm*}{Theorem}
\title{Entropic selection for optimal transport on the line with distance cost}
\author{
	Armand Ley}
\begin{document}
	\date{}
	\maketitle

	{
		\makeatletter
		\renewcommand{\@makefntext}[1]{\noindent #1}
		\makeatother
		
		\renewcommand{\thefootnote}{}
		\renewcommand{\thefootnote}{\arabic{footnote}}
	}
	
	\pagestyle{plain}
	\textbf{Abstract:} {\small
		 We study the small-regularisation limit of the entropic optimal transport problem on the line with distance cost. While convergence of entropic minimizers is well understood in the discrete setting and in the case where the cost is continuous and there is a unique optimal transport plan, the question of existence and characterization outside these settings remains largely open. We propose a natural candidate for the limiting object and establish its convergence under mutual singularity of the marginals. For arbitrary marginals, we moreover prove that every limit point of entropic minimizers obeys a structural condition known as weak multiplicativity. The construction of our candidate relies on a decomposition theorem for optimal transport plan that we believe is of independent interest. This article complements the previous work of Di Marino and Louet \cite{di_marino_entropic_2018}.}

	\section{Introduction}
	
	\subsection{On the Monge-Kantorovitch problem and its entropic regularization}
	In recent decades, optimal transport theory has attracted considerable attention, due to its many connections to probability, analysis, geometry and other areas of mathematics (see, e.g., \cite{villani_optimal_2009,villani_topics_2003,santambrogio_optimal_2015,ambrosio_gradient_2008,rachev_mass_1998}). In this paper, we investigate one of the most elementary optimal transport problems, namely the $L^1$ Monge-Kantorovich problem on the real line. Given two probability measures $\mu$ and $\nu$ on $\R$, this problem consists of minimizing the global transport cost function
	\begin{equation*}
		J: \pi \in \Marg(\mu,\nu) \mapsto \int_{\R^2} |y-x| \dd \pi (x,y),
	\end{equation*}
	over the set $\Marg(\mu,\nu)$ of measures that satisfy the marginal constraints $\pi(A \times \R) = \mu(A)$ and $\pi(\R \times B) = \nu(B)$ for every pair $(A,B)$ of Borel subsets of $\R$. If $|y-x|$ is replaced by $|y-x|^p$ for some $p>1$ in the definition of the cost function $J$, then the set of minimizers of $J$, known as optimal transport plans, reduces to a singleton which consists of the well-known (co)monotone transport. This transport plan, also known as quantile transport plan, is defined as $(G_\mu,G_\nu)_\# \left(\LL^1_{\res ]0,1[} \right)$, where $\LL^1_{\res ]0,1[}$ stands for the restriction of the Lebesgue measure to $]0,1[$, while $G_\mu$ and $G_\nu$ denote the quantile functions of $\mu$ and $\nu$, respectively. For $p=1$ uniqueness of an optimal transport plan does not hold, and describing the set $\OO(\mu,\nu)$ of optimal transport plans becomes a delicate problem which will address later through a decomposition result. The lack of uniqueness of optimal transport plans for the distance cost raises the broader question of selection under perturbation. Entropic regularisation, the most prominent perturbation of the Monge–Kantorovich problem, provides a natural framework in which to address this issue. Given a regularization parameter $\ee > 0$, the entropic optimal transport problem consists in minimizing the regularized cost function 
	$$J_\ee : \pi \in \Marg(\mu,\nu) \mapsto \int_{\R^2} |y-x| \dd \pi(x,y) + \ee \Ent(\pi|\mu \otimes \nu),$$
	where the relative entropy $\Ent(\cdot|\mu \otimes \nu)$ is defined as follows:
	\begin{equation*}
		\Ent(\pi|\mu \otimes \nu) =
		\begin{cases}
			\int_{\R^2} \log\left(\frac{\ddd \pi}{\ddd\mu \otimes \nu}\right) \frac{\ddd \pi}{\ddd\mu \otimes \nu} \dd\mu \otimes \nu &\text{if } \pi \ll \mu \otimes \nu \\+\infty &\text{otherwise }
		\end{cases}.
	\end{equation*}
	When $\ee = 0$, we recover the usual optimal transport problem. Heuristically, since $\Ent(\cdot|\mu \otimes \nu)$ can be interpreted as a measure of divergence from $\mu \otimes \nu$, the entropic term in $J_\ee$ encourages the minimizer $\pi_\ee := \argmin J_\ee$ to balance two competing effects: minimizing the transport cost \emph{and} being as close as possible to the product measure. This regularization, which applies to more general cost functions, was popularized by Cuturi \cite{cuturi_sinkhorn_2013} and has since become a highly active area of research. Its main appeal is likely due to the fact that it transforms the optimal transport problem ---which solutions are hard to compute in high dimensions--- into a computationally tractable problem. Indeed, adding this regularization term allows the use of robust and efficient algorithms, such as the celebrated Sinkhorn algorithm \cite{sinkhorn_relationship_1964}. We refer to \cite{peyre_computational_2019} and the numerous references therein for a comprehensive overview of the computational aspects related to optimal transport and to \cite{nutz_introduction_nodate} for a general introduction to entropic optimal transport theory. 
	
	Since entropic optimal transport is viewed as an approximation of classical optimal transport, understanding its behaviour as the regularization parameter $\ee$ tends to $0^+$ is of critical importance. Beyond its computational applications, several aspects of this convergence have been studied, including the convergence of minimizers, convergence of solutions to the associated dual problem, stability with respect to the data (marginals and cost function), and convergence rate estimates. In this paper, we are interested in the question of the convergence of the minimizers of $J_\ee$. We refer to the works \cite{carlier_convergence_2017,leonard_schrodinger_2012} in the context of the squared Euclidean distance, where $\Gamma$-convergence methods are used to prove the convergence of the entropic minimizers to the unique optimal transport plan. This result has been extended in \cite{bernton_entropic_2022}, where the authors show that for continuous cost functions, every cluster point of entropic minimizers is an optimal transport plan: in particular, if there exists a unique optimal transport plan, we obtain the convergence toward this transport plan. In the case of measures with finite support, it is known \cite[Proposition 4.1]{peyre_computational_2019} that entropic minimizers converge to $\argmin_{\OO(\mu,\nu)} \Ent(\cdot|\mu \otimes \nu).$ Outside these cases, proving the convergence and finding a characterization of the potential limit remains an open question and an active area of research. This selection problem can be reformulated as follows: is an optimal transport plan selected by entropic regularization, and if so, how can it be characterized? The most elementary setting in which this selection problem arises is the distance cost on the real line: $\OO(\mu,\nu)$ is generally not a singleton, and optimal transport plans are not even absolutely continuous with respect to the product measure.\footnote{E.g., if $\mu = \d_0 + \d_1 + \LL^1_{\res ]2,3[}$ and $\nu = \d_1 + \d_2 + \LL^1_{\res ]2,3[}$.} To state our result concerning the entropic selection problem in this setting, we first need to provide a more tractable description of the set $\OO(\mu,\nu)$.
	\subsection{A decomposition result for optimal transport plans}
	By a {decomposition result} for $\OO(\mu,\nu)$, we mean a representation of any optimal transport plan as a sum of simpler components that can be analysed independently. More specifically, we construct a set $\DD = \ens{(\mu_i,\nu_i)}{i \in \II}$ of pairs of measures such that, for all $i \in \II$, the set $\OO(\mu_i,\nu_i)$ admits a tractable description and, for every  $\pi \in \OO(\mu,\nu)$, there exists a unique family $(\pi_i)_{i \in \II} \in \prod_{i \in \II} \OO(\mu_i,\nu_i)$ such that $\pi = \sum_{i \in \II} \pi_i$. This motivates the following notation.
	\begin{nott}[Minkowski sum and direct sum of set of measures]\label{nott:somme_directe}
		Given a countable family $(E_i)_{i \in \II}$ of subsets of $\MM_+(\R^2)$, we denote by $\sum_{i \in \II} E_i$ the image of the map $\phi : (\pi_i)_{i \in \II} \in \prod_{i \in \II} E_i \mapsto \sum_{i \in \II} \pi_i \in \MM_+(\R^2).$ If $\phi$ is injective, we write $\bigoplus_{i \in \II} E_i$ instead of $\sum_{i \in \II} E_i$.
	\end{nott}
	With this notation, our decomposition theorem becomes $\OO(\mu,\nu) = \bigoplus_{i \in \II}\OO(\mu_i,\nu_i).$	Strictly speaking, the decomposition is formulated in the more general setting of cyclical monotone transport plans (see Theorem \ref{pro:decomposition}). In the case of measures without atoms and with compact support, it coincides with a previous result by Di Marino and Louet \cite[Proposition $3.1$]{di_marino_entropic_2018}. To motivate our decomposition theorem and clarify the exposition of the decomposition procedure, we briefly present a standard decomposition result for the set of transport plans satisfying the martingale constraint.

	\paragraph{Decomposition theorem for the set of martingale transport plans}	A classical theorem (see Strassen \cite{strassen_existence_1965}) states that the set $\Mart(\mu,\nu)$ of martingale transport plans, \ie \ the transport plans $\pi$ that admit a disintegration of the form $\pi(\ddd x,\ddd y) = \mu(\ddd x)k_x(\ddd y)$ with $\int_\R y k_x (\ddd y) = x$, is non-empty if and only if the measures $\mu$ and $\nu$ are in the convex order, meaning that the potential functions $u_\mu : t \mapsto \int_{\R} |x-t| \dd\mu(x)$ and $u_\nu : t \mapsto \int_{\R} |x-t| \dd\nu(x)$ satisfy $u_\mu \leq u_\nu$.\footnote{The measure $\mu$ is smaller than $\nu$ in the convex order if $\int f \dd \mu \leq \int f \dd \nu$ for all convex function $f$. This characterization trough the ordering of potential functions holds only for measures on $\R$. In the following, we refer to a statement as a \emph{Strassen-type} result if it characterizes a stochastic order via the existence of a transport plan satisfying structural constraints.} Assuming $u_\mu \leq u_\nu$, let $(I_k)_{k \in \KK}$ be the family of connected components of  $\{u_\mu< u_\nu\} := \ens{x \in \R}{u_\mu(x) < u_\nu(x)}$, and define the components of $\mu$ by setting  $\mu^= = \mu_{\res \{u_\mu = u_\nu\}}$ and $\mu_k = \mu_{\res I_k}$ for each $k \in \KK$. Beiglböck and Juillet \cite[Theorem A.4]{beiglbock_problem_2016} proved that there exists a unique family   $((\nu_k)_{k \in \KK},\nu^=)$ such that $\nu = \sum_{k \in \KK} \nu_k + \nu^=$, $u_{\mu^=} \leq u_{\nu^=}$ and $u_{\mu_k} \leq u_{\nu_k}$ for each $k \in \KK$. Moreover, the inequality $u_{\mu_k} < u_{\nu_k}$ is satisfied in the interior of the support of $\mu$ and $\nu$, and they established that $\Mart(\mu,\nu)$ admits the following decomposition:
	\begin{equation*}
		\Mart(\mu,\nu) = \left(\bigoplus_{k \in \KK} \Mart(\mu_k,\nu_k) \right) \oplus \Mart(\mu^=,\nu^=).
	\end{equation*}
	The proof of this result mainly relies on the fact that the set
	$$\ens{x \in \R}{\forall \pi \in \Mart(\mu,\nu), \pi(\left(]-\infty,x[ \times [x,+ \infty[\right) \uplus \left( ]x,+\infty[ \times ]-\infty,x] \right))=0}$$
	coincides precisely with the set where the potential functions are equal. In higher dimensions, although the situation is notably more intricate, similar results have recently been established \cite{ghoussoub_structure_2019,march_irreducible_2019,obloj_structure_2017}.

	\paragraph{A decomposition result for cyclically monotone transport plan} Our decomposition method for $\OO(\mu,\nu)$ closely mirrors that of $\Mart(\mu,\nu)$: the optimality constraint replaces the martingale constraint, and the cumulative distribution functions take the role of the potential functions. In this introduction, we present the decomposition result of $\OO(\mu,\nu)$ for atomless measures: in this case, the analogy with the martingale decomposition becomes particularly clear, and our decomposition theorem can be presented more intuitively. In this context, the cumulative distribution functions $F_\mu^+$ and $F_\nu^+$ of $\mu$ and $\nu$ are continuous, so that the sets $\{F_\mu^+ > F_\nu^+\}$ and $\{F_\mu^+ < F_\nu^+\}$ are open, and their connected components, denoted by $(I_k^+)_{k \in \KK^+}$ and $(I_k^-)_{k \in \KK^-}$, form countable families of open intervals. The components of $\mu$ and $\nu$ are then defined by restricting them to these connected components. More precisely, we set
	\begin{equation}\label{def:comp_marg_intro}
		\begin{cases}
			\mu^= = \mu_{\res \{F_\mu^+ = F_\nu^+\}}\\
			\nu^= = \nu_{\res \{F_\mu^+ = F_\nu^+\}}
		\end{cases},
		\begin{cases}
			\mu_k^+ = \mu_{\res I_k^+}\\
			\nu_k^+ = \nu_{\res I_k^+}
		\end{cases}
		\text{and }
		\begin{cases}
			\mu_k^- = \mu_{\res I_k^-}\\
			\nu_k^- = \nu_{\res I_k^-}
		\end{cases}.
	\end{equation}
	We write $\mu \leq_\F \nu$ when the inequality $F_\mu^+ \geq F_\nu^+$ holds, with strict inequality $F_\mu^+ > F_\nu^+$ on the interior of the union of the supports of $\mu$ and $\nu$.\footnote{For  measures that may have atoms, the definition $\leq_\F$ requires both right continuous \emph{and} left-continuous distribution functions, and strict inequality may also be required at extremal points of the support: we refer the reader to Definition \ref{def:Kellerer_order_large}.} We can now state our decomposition result in the case of atomless measures.
	\begin{introthm}\label{them:deco_intro}
		Let $\mu$ and $\nu$ be two atomless measures on $\R$ such that $\min J < +\infty$, and define the marginal components  $((\mu_k^+)_{k \in \KK^+},(\mu_k^-)_{k \in \KK^-},$$\mu^=)$ and $((\nu_k^+)_{k \in \KK^+},(\nu_k^-)_{k \in \KK^-},\nu^=)$  as in Equation  \eqref{def:comp_marg_intro}.
		\begin{enumerate}
			\item Then $\mu^= = \nu^=$, and for all $k \in \KK^+$ (resp. $k \in \KK^-$), we have $\mu_k^+ \leq_{\F} \nu_k^+$ (resp. $\nu_k^- \leq_{\F} \mu_k^-$).
			\item The set $\OO(\mu,\nu)$ admits the following decomposition: 
			\begin{equation*}
				\OO(\mu,\nu) = \left(\bigoplus_{k \in \KK^+} \OO(\mu_k^+,\nu_k^+) \right) \oplus \left(\bigoplus_{k \in \KK^-} \OO(\mu_k^-,\nu_k^-)\right) \oplus \OO(\mu^=,\nu^=).
			\end{equation*}
		\end{enumerate}
	\end{introthm}
	As for the decomposition of $\Mart(\mu,\nu)$, the proof of this result relies on the characterization of a specific set, whose elements will be called barrier points. This set, denoted by $\BB(\mu,\nu)$ is defined by $$\BB(\mu,\nu) =\ens{x \in \R}{\forall \pi \in \OO(\mu,\nu), \pi(]-\infty,x[ \times ]x,+\infty[ \cup ]x,+\infty[ \times ]-\infty,x)=0},$$ and coincides with the set of points where $F_\mu^+$ and $F_\nu^+$ are equal.\footnote{This is not true anymore if $\mu$ and $\nu$ may have atoms: we refer the reader to Equation \eqref{eq:E_==_B_modifie} for the general characterization of $\BB(\mu,\nu)$ in terms of cumulative distribution functions.} In case $\mu$ and $\nu$ also have compact support, this result coincides with that of Di Marino and Louet: we mention that their proof relies on the construction of an explicit solution to the dual problem and on the use of complementary slackness to constrain mass displacement for elements of $\OO(\mu,\nu)$.\footnote{Di Marino and Louet proved a slightly different version of this result (see Proposition \ref{pro:structure_result_DL} for the exact statement).} {In contrast, our approach is based on cyclical monotonicity, a geometric property of the support of optimal transport plans (see Definition \ref{defi:cycl_mon}), which we use to establish our earlier characterization of barrier points.} In Section \ref{sec:decomposition}, after proving Proposition \ref{them:deco_intro}, we extend this decomposition result to the case where the measures may have atoms, which, as we will see in detail, requires a more refined analysis. Without delving into technicalities at this stage, we note that the general case requires working with both right-continuous \emph{and} left-continuous cumulative distribution functions (see Definition \ref{defi:compo_real_line} and the preceding example), that components may not be mutually singular (see Definition \ref{def:component_marginals} and Point \ref{Point:not_singular} of Remark \ref{remq:on_marginals}), and that a refinement of the notion of barrier points is required to control mass repartition at the boundaries of the components (see Proposition \ref{pro:caracterization_barrier}). In Theorem \ref{pro:decomposition}, after adapting the definition of marginal components, we extend Theorem \ref{them:deco_intro} for any pair of measures $(\mu,\nu)$, including measures that may have atoms. This result, stated in the broader context of cyclically monotone transport plans, coincides with Theorem \ref{them:deco_intro} when $\mu$ and $\nu$ are atomless measures such that $\min J < + \infty$.	

	We now highlight a interesting consequence of our approach. For $L^p$ transport with $p>1$, it is well known that a transport plan is optimal if and only if there exists a set $\Gamma$ such that, for every $(x,y), (x',y') \in \Gamma$, $|y-x|^p+ |y'-x'|^p \leq |y-x'|^p+ |y'-x|^p$.\footnote{This condition corresponds to cyclical monotonicity (see Definition \ref{defi:cycl_mon}) for cycles of length two, and is equivalent to $(y'-x')(y-x) \geq 0$. This is the key argument to prove that $\OO(\mu,\nu)$ reduces to the monotone transport plan.} We establish that this characterization of optimality also holds for $p=1$, and provide an interpretation in terms of ``crossings" (see Proposition \ref{pro:swapping_lemma}).
	
	\subsection{Back to the selection problem}
	\paragraph{Convergence of minimizers for the $L^1$ entropic regularization on the real line}
	 Our first contribution to the study of the selection problem for the distance cost on the line is to apply our decomposition result for $\OO(\mu,\nu)$ and a Strassen-type theorem due to Kellerer \cite{kellerer_order_1986} to  build a natural candidate $\KK(\mu,\nu)$ to be the limit of the entropic minimizers $(\pi_\ee)_{\ee>0}$. This optimal transport plan $\KK(\mu,\nu)$ is characterized by the fact that it is as close from the product measure as possible in the following sense: each positive (respectively negative) component of $\KK(\mu,\nu)$ is strongly multiplicative, meaning that it coincides with the restriction of a product measure on $\F := \ens{(x,y) \in \R^2}{x \leq y}$ (respectively on $\ti{\F} := \ens{(x,y) \in \R^2}{x \geq y}$). We conjecture that $(\pi_\ee)_{\ee>0}$ converges to $\KK(\mu,\nu)$ when $\ee \to 0^+$. In \cite[Theorem $4.1$]{di_marino_entropic_2018}, the authors proved that $(\pi_\ee)_{\ee >0}$ converges to an optimal transport whose components are strongly multiplicative, under the assumption that $\mu$ and $\nu$ are compactly supported, have finite entropy with respect to the Lebesgue measure on $\R$, and there exists $\pi \in \OO(\mu-\mu^=,\nu-\nu^=)$ with finite entropy with respect to $\mu \otimes \nu$ (see Sub-subsection \ref{subsubsection:Dl_results}).\footnote{Along with some additional technical assumption, see Theorem \ref{them:conv_Di_Marino_Louet} for the precise statement.} Our main convergence result is stated as follows.
	\begin{introthm}\label{them:convergence_semi_discret_intro}
		Consider a pair $(\mu,\nu)$ of probability measures satisfying one of the two following hypothesis:
		\begin{enumerate}
			\item $\mu$ and $\nu$ both have finite support;
			\item For all $k \in \KK^+$ (respectively $k \in \KK^-$), $(\mu_k^+,\nu_k^+)$ (resp. $(\mu_k^-,\nu_k^-)$) forms a pair of mutually singular measures.
		\end{enumerate}
		Then $\pi_\ee \dcv{\ee \to 0^+} \KK(\mu,\nu).$ 
	\end{introthm}
	This covers the case where one measure is atomic and the other has no atoms. In the general case the conjecture is still open, but we prove the following result.
	
	\begin{introthm}\label{them:valeur_adherence_solutions_penalisees_generales_SWFM_intro}
		Let $(\mu,\nu)$ be a pair of probability measures. If $\pi^*$ is a cluster point of $(\pi_\ee)_{\ee>0}$, then every restriction of $\pi^*$ to a product set contained in $\ens{(x,y) \in \R^2}{x \leq y}$ or $\ens{(x,y) \in \R^2}{x \geq y}$ coincides with a product measure.
	\end{introthm}
	
	This property, known as weak multiplicativity was introduced by Kellerer \cite{kellerer_order_1986}, and its ``strict version" has proven useful in studying the Markov property of real-valued point processes \cite{kellerer_markov_1987}.
	
	\paragraph{Decomposition and regularized problem in higher dimension}
	
	In the Monge–Kantorovich problem on $\R^n$ with cost given by the Euclidean norm, a classical decomposition strategy due to Sudakov \cite{sudakov_geometric_1979} has proven instrumental to establish the existence of an optimal transport map. The core idea of this decomposition is to disintegrate the measures $\mu$ and $\nu$ into families $(\mu_\s)_\s$ and $(\nu_\s)_\s$ of measures concentrated on one-dimensional ``transport rays". This allows one to consider, on each transport ray $\s$, a solution $\pi_\s$ to the one-dimensional transport from $\mu_\s$ to $\nu_\s$. The delicate step consists in gluing the $\pi_\s$ together into an optimal transport plan from $\mu$ to $\nu$. Assuming $\mu \ll \LL^n$, if each $\pi_\s$ is chosen as the monotone transport between $\mu_\s$ and $\nu_\s$, this gluing procedure is well-defined and yields an optimal transport \emph{map}. For further details on this decomposition, we refer the reader to the lecture notes by Ambrosio \cite{ambrosio_lecture_2003} as well as to the textbooks \cite[Sections 3.1.3 and 3.1.4 ]{santambrogio_optimal_2015} and \cite[Chapter 18]{maggi_optimal_2023} (which contains an extensive bibliographical note on the historical development of this decomposition). 
	The idea of decomposing $\mu$ and $\nu$ along maximal transport rays has proven itself essential in addressing the selection problem arising in the study of  $L^1(\R^d)$ entropic optimal transport for the Euclidean distance cost $c(x,y) = \n{y-x}{}$ with $d \geq 2$. According to a recent preprint of Aryan and Ghosal \cite{aryan_entropic_2025}, if $X$ and $Y$ are compact subsets of $\R^d$ such that $d(X,Y) = \inf \left( \ens{\n{y-x}{}}{x \in X, y \in Y} \right) >0$, and  $\mu$ and $\nu$ are probability measures on $X$ and $Y$ with smooth densities with respect to the Lebesgue measure, then the family $(\pi_\ee)_{\ee>0}$ converges to an optimal transport plan $\pi^{\text{opt}}$. Moreover, $\R^d \times \R^d$ admits a decomposition into maximal transport rays $\s$, with respect to which the optimal coupling $\pi^{\text{opt}}$ admits the following variational characterization. For all $\s$, there exists $c_\sigma \in \R$ such that the components \( \pi^{\text{opt}}_{\sigma} \) of \( \pi^{\text{opt}} \) along $\sigma$ minimizes 
	\begin{equation}\label{eq:EOT_sup}
		\inf_{\pi_\s \in \Marg(\mu_\s,\nu_\s)} \int_{\R \times \R} \left(c_\sigma |y-x| - \frac{d-1}{2} \log(2\pi |y-x|)\right)\ddd \pi_\s(x,y) + \Ent(\pi_\s|\mu_\sigma \otimes \nu_\sigma).
	\end{equation}
	
	Observe that, contrary to what one might expect by looking at our one-dimensional result, \( \pi^{\text{opt}}_{\sigma} \) does \emph{not} minimize \( \Ent( \cdot | \mu_\sigma \otimes \nu_\sigma) \) over \( \OO(\mu_\sigma, \nu_\sigma) \). This surprising fact --- that \( \pi^{\text{opt}}_{\sigma} \) actually solves~\eqref{eq:EOT_sup} rather than minimizing \( \Ent(\cdot | \mu_\sigma \otimes \nu_\sigma) \) --- was first observed by Marcel Nutz and Chenyang Zhong, prior to~\cite{aryan_entropic_2025} (see the video~\cite{nutz2023entropic}). A paper is in preparation.

	\subsection{Organization of the paper}
	
	In Section \ref{sec:decomposition}, we focus on the decomposition theorem for optimal transport plans. We begin by presenting the result of Di Marino and Louet (Proposition~\ref{pro:structure_result_DL}) in the case of atomless marginals. This serves both to build intuition for the reader and to introduce the tools that will be used throughout the paper. After recalling the notions of cyclically monotone transport plans (Definition~\ref{defi:cycl_mon}) and barrier points (Definition~\ref{defi:barrier_points}), we provide a concise alternative proof of their result. The argument relies on the fact that the points where $F_\mu^+ = F_\nu^+$ are barrier points (Proposition~\ref{pro:barier_point_continuous_inclusion}), together with a characterization of optimal transport plans for stochastically ordered marginals (Proposition~\ref{pro:ordre_stochastique}). Building on these ingredients, we obtain our decomposition theorem for atomless marginals (Proposition~\ref{them:deco_cas_sans_atomes}). We then extend this result to arbitrary marginals, including those with atoms. The first step is a technical lemma showing that the relative position of the cumulative distribution functions constrains the regions where optimal transport plans may concentrate (Proposition~\ref{pro:barrier_points}). We next introduce the components of the real line (Definition~\ref{defi:compo_real_line}), of the marginals (Definition~\ref{def:component_marginals}), and of an optimal transport plan (Definition~\ref{defi:composantes}). We subsequently prove that the components of an optimal transport plan have the appropriate marginals (Proposition~\ref{pro:transport_fitting_marg}). After introducing the reinforced stochastic order (Definition~\ref{def:Kellerer_order_large}), we show that any two distinct marginal components are ordered w.r.t. this order (Proposition~\ref{pro:prop_deco}). These results culminate in our decomposition theorem for $\OO(\mu,\nu)$ (Theorem~\ref{pro:decomposition}). Finally, we discuss several applications of the decomposition. After recalling the Strassen-type theorem for $\leq_\F$, we define candidate $\KK(\mu,\nu)$ (Definition~\ref{defi:gen_kell_transport_plan}). We then provide a characterization of barrier points in terms of cumulative distribution functions (Proposition~\ref{pro:caracterization_barrier}), and conclude by proving that a transport plan is optimal if and only if it is supported on a family of transport paths that do not ``freely'' cross (Proposition~\ref{pro:swapping_lemma}).

	In Section~\ref{sec:convergence_entro}, we study the behaviour of $(\pi_\ee)_{\ee>0}$ as $\ee \to 0^+$. We first use an adaptation of a result of Di Marino and Louet (Lemma~\ref{lem:minimum_kell_ord}), together with the fact that the relative entropy of an optimal transport plan equals the sum of the entropies of its components (Proposition~\ref{pro:equi_reduction}), to establish that $\KK(\mu,\nu)$ minimizes the relative entropy among all optimal transport plans (Theorem~\ref{them:kellerer_minimizes}). This immediately implies that, whenever there exists an optimal transport plan with finite entropy, the family $(\pi_\ee)_{\ee>0}$ converges to $\KK(\mu,\nu).$ We then introduce the notion of (large) weak multiplicativity (Definition~\ref{defi:weak_multiplicativity}) and prove that this property is stable under weak convergence of transport plans (Proposition~\ref{cor:stabilite_SWFM}). As a consequence, every cluster point of $(\pi_\ee)_{\ee>0}$ is weakly multiplicative (Theorem~\ref{them:valeur_adherence_solutions_penalisees_generales_SWFM}). Finally, after introducing strict versions of the reinforced stochastic order, of weak multiplicativity, and of strong multiplicativity, we prove that if the non-equal pairs of components of $(\mu,\nu)$ are mutually singular, then $\pi_\ee$ converges to $\KK(\mu,\nu)$ as $\ee \to 0^+$. This result in particular covers the case of semi-discrete pairs of measures.

	\subsection{General notation}\label{subsection:not}

	\begin{enumerate}
		\item For all $(a,b) \in \Z^2$, we define $\ent{a}{b} = [a,b] \cap \Z$.
		\item Any Polish space \(E\) is equipped with its Borel \(\sigma\)-algebra \(\BB(E)\). We denote by $\PP(E)$, $\MM_+(E)$, and $\MM_+^\s(E)$ the set of probability measures, finite positive measures, and $\sigma$-finite positive measures on $(E,\BB(E))$, respectively. For all $\g \in \MM_+^\s(E)$ and $A \in \BB(E)$, we define the restriction of $\g$ to $A$ as $\g_{\res A} : B \in \BB(E) \mapsto \g(A \cap B)$. We will say that $\g$ is concentrated on $A \in \BB(E)$ if $\g(A^c) = 0$. We will often need to consider the support of the measure, denoted as $\spt(\g)$ of a measures $\g$. Recall that $\spt(\g)$ is a closed set, $\g$ is concentrated on $\spt(\g)$, and every neighbourhood of an element belonging to the support has positive measure. The pushforward notation will often be needed: if $E'$ is another Polish space, $f : E \to E'$ is measurable and $\g \in \MM_+^\s(E)$, we denote by $f_{\#} \g$ the pushforward measure of $\g$ by $f$, defined by $f_{\#} \g : B \in \BB(E') \mapsto \g(f^{-1}(B))$. 
		\item\label{point:atoms_not} In the case where $E=\R$, we introduce specific notation. We begin with extremal point of the support: define  $s_\g = \inf(\spt(\g))$ and $S_\g = \sup(\spt(\g)) \in [-\infty,+\infty]$ (the letter ``s'' stands for support). Next, let $F_\g^+ : t \in \R \mapsto \g(]-\infty,t]) \in \R^+ $ and $F_\g^- : t \in \R \mapsto \g(]-\infty,t[) \in \R^+$ denote the cumulative distribution functions of $\g$, respectively. Recall that $F_\g^-$ is left-continuous, $F_\g^+$ is right-continuous and $F_\g^- \leq F_\g^+$. Finally, we denote by $\At(\g) = \ens{x \in \R}{\g(\{x\}) > 0}$ the set of atoms of $\g$: we shall say that $\g$ is atomless when $\At(\g) = \emptyset.$
		\item The following notation concern specific spaces of measures. We define $\MM^2_+$ as $\MM^2_+ = \{(\mu,\nu) \in \MM_+(\R)^2~;~\mu(\R) = \nu(\R) > 0\}$: observe that $\MM_+^2$ is a proper subset of $\MM_+(\R)^2$. We also define $\MM_1(\R) = \ens{\mu \in \MM_+(\R)}{\int_{\R} |x| \dd\mu(x)<+\infty}$. Given two measures $\mu, \nu \in \MM_+(\R)$, we define the set of transport plans from $\mu$ to $\nu$ by  $\Marg(\mu, \nu) = \ens{\pi \in \MM_+(\R^2)}{{p_1}_{\#} \pi = \mu \text{ and } {p_2}_{\#} \pi = \nu}$, where $p_1 :(x,y) \in \R^2 \mapsto x \in \R$, $p_2 : (x,y) \in \R^2 \mapsto y\in \R$ stand for the projections from $\R^2$ to $\R$. Recall that the cost function for transport plans $J : \Marg(\mu,\nu) \to \R_+$ associated to the distance is defined by $J(\pi) = \int_{\R^2} |x-y| \dd \pi(x,y)$:  its minimum, also  known as the $1$-Wasserstein distance between $\mu$ and $\nu$, is denoted by $W_1(\mu,\nu) = \min(J)$. The set $\OO(\mu,\nu)$ of optimal transport plans from $\mu$ to $\nu$ is then defined as the set of minimizers of $J$, that is, $\OO(\mu,\nu) = \ens{\pi \in \Marg(\mu,\nu)}{J(\pi) = W_1(\mu,\nu)}$. Note that $(\mu,\nu) \in \MM_1(\R) \times \MM_1(\R)$ implies that $W_1(\mu,\nu) < + \infty$ and $W_1(\mu,\nu) < + \infty$ implies $(\mu,\nu) \in \MM_+^2$. We denote by $\CCC(\mu,\nu)$ the set of cyclically monotone transport plan for $\mu$ to $\nu$ (see Definition \ref{defi:cycl_mon}).
		\item The following notation for half-planes, adopted from Kellerer, will be useful: we define $\F = \ens{(x,y) \in \R^2}{x \leq y}$, $\ti{\F} = \ens{(x,y) \in \R^2}{y \leq x}$, $\G = \{(x,y) \in \R^2~;~x<y\}$,  $\ti{\G} = \{(x,y) \in \R^2~;~x>y\}$ and $\D = \ens{(x,x)}{x \in \R}$. Given $\H \in \BB(\R^2)$, we define  $\Marg_\H(\mu,\nu) = \{\pi \in \Marg(\mu,\nu)~;~\pi(\H^c) = 0\}.$
	\end{enumerate}

	\section{The decomposition result.}\label{sec:decomposition}
	
	\subsection{The structure result of Di Marino--Louet}\label{subsection:continuous_decomposition}
	
	Before presenting our decomposition, we examine the result by Simone Di Marino and Jean Louet regarding the structure of optimal transport plans, when $(\mu,\nu) \in \PP(\R)^2$ is a pair of atomless marginals \cite[Proposition $3.1$]{di_marino_entropic_2018}. Although their result is not originally presented as a decomposition of the space $\OO(\mu,\nu)$ in direct sums of subspaces, we shall see that the proof of this result leads to our decomposition result (for atomless measures). When $\mu$ and $\nu$ are atomless, the functions $F_\mu^+$ and $F_\nu^+$ are continuous; therefore, the sets $\{F_\mu^+ > F_\nu^+\} := \ens{x \in \R}{F_\mu^+(x) > F_\nu^+(x)}$ and $\{F_\mu^+ < F_\nu^+\} := \ens{x \in \R}{F_\mu^+(x) < F_\nu^+(x)}$ are open and consists of a countable family of open, connected components. Let $(]a_k^+,b_k^+[)_{k \in \KK^+}$ and $(]a_k^-,b_k^-[)_{k \in \KK^-}$ denote the connected components of $\{F_\mu^+ > F_\nu^+\}$ and $\{F_\mu^+ < F_\nu^+\}$, respectively. Using our notation, the result by Di Marino and Louet can be stated as follows.
	\begin{pro}[Structure result by Di Marino--Louet]\label{pro:structure_result_DL}
		Assume that $\mu$ and $\nu$ in $\PP(\R)$ are atomless and have compact support. For all $\pi \in \OO(\mu,\nu)$, there exists a set $S_\pi \in \BB(\R^2)$ such that $\pi(S_\pi) = 1$ and:
		\begin{enumerate}
			\item\label{Point:structure_fixe} $[S_\pi \cap (\{F_\mu^+ = F_\nu^+\} \times \R)] \subset \D$;
			\item\label{Point:structure_avant} For all $k \in \KK^+$, $[S_\pi \cap ( ]a_k^+,b_k^+[ \times \R)] \subset \F \cap (\R \times ]a_k^+,b_k^+[)$;
			\item\label{Point:structure_arriere} For all $k \in \KK^-$, $[S_\pi \cap (]a_k^-,b_k^-[ \times \R)] \subset \ti{\F} \cap (\R \times ]a_k^-,b_k^-[)$.
		\end{enumerate}
	\end{pro}
	
	Interpreting a transport plan $\pi \in \Marg(\mu,\nu)$ as a mass displacement from the distribution $\mu$ to the distribution $\nu$, a point $(x,y)$ represents a transport path and $\pi(x,y)$ the amount of mass displaced from $x$ to $y$. Observe that $\F$ represents forward transports, $\ti{\F}$ represents backward transports, and $\D$ corresponds to fixed mass. From this perspective, Point \ref{Point:structure_fixe} of Proposition \ref{pro:structure_result_DL} states that the mass on the set $\{F_\mu^+ = F_\nu^+\}$ remains fixed under any optimal transport for the distance cost. Point \ref{Point:structure_avant} (resp. Point \ref{Point:structure_arriere}) of Proposition \ref{pro:structure_result_DL} states that for an optimal displacement, the mass of $\mu(dx)$ on a connected component  of $\{F_\mu^+ > F_\nu^+ \}$ (resp. $\{F_\mu^+ < F_\nu^+ \}$) remains within the same component and moves forward (resp. backward). Interpreting a transport as a measure on the plan $\R^2$, Proposition \ref{pro:structure_result_DL} can be viewed as an assertion about the concentration region of elements of $\OO(\mu,\nu).$ Specifically, every element of $\OO(\mu,\nu)$ is concentrated on the following disjoint union of triangles and diagonal parts: $$\biguplus_{k \in \KK^+} (\F \cap ]a_k^+,b_k^+[^2) \biguplus \biguplus_{k \in \KK^-} (\ti{\F} \cap ]a_k^-,b_k^-[^2) \biguplus \left(\D \cap \{F_\mu^+ = F_\nu^+\}^2 \right).$$ 
	\begin{ex}\label{ex:unif_DL}
		Define $\mu = \1_{]0,1[} \cdot \LL^1$ and $\nu = \left(2 \1_{]1/8,1/4[} + 2 \1_{]3/8,1/2[} +  \1_{]1/2,3/4[} + 2 \1_{]3/4,7/8[}\right) \cdot \LL^1.$ We have $\{F_\mu^+ > F_\nu^+\} = ]0,1/4[ \cup ]1/4,1/2[$, $\{F_\mu^+ < F_\nu^+\} = ]3/4,1[$ and $\{F_\mu^+ = F_\nu^+\} = [1/2, 3/4].$ According to Proposition \ref{pro:structure_result_DL}, for every optimal displacement from $\mu$ to $\nu$ the mass moves as illustrated in Figure \ref{fig:arrow_cdf}  and is concentrated in the orange region shown in Figure \ref{fig:area_of_con}. 
	\end{ex}
		
	\begin{figure}[h!]
		\begin{center}
			\def\svgwidth{10cm}
			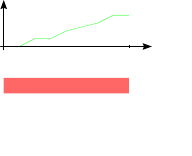
			\caption{Representation of $F_\mu^+$ (in red) and $ F_\nu^+$ (in green) and mass displacement of elements of $\OO(\mu,\nu)$  }
			\label{fig:arrow_cdf}
		\end{center}
	\end{figure}
	
	\begin{figure}[h]
		\begin{center}
			\def\svgwidth{10cm}
			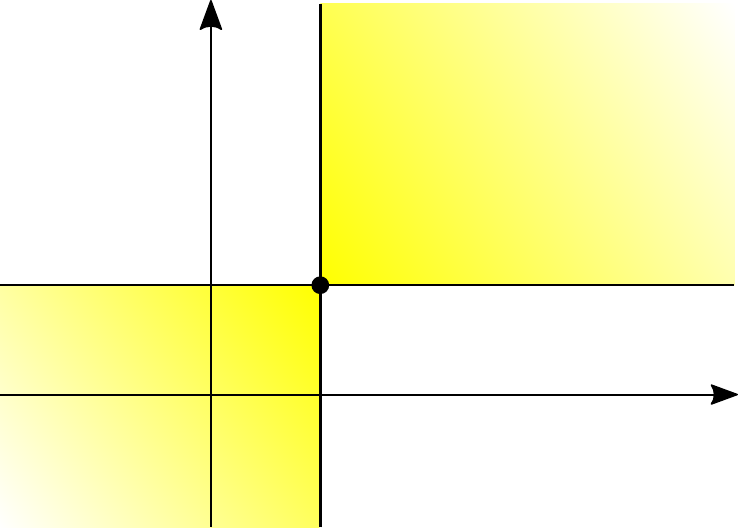
			\caption{Concentration region of optimal transport plans when $x \in \R$ is a barrier point.}
			\label{fig:area_of_con_bar}
		\end{center}
	\end{figure}

	We propose a brief alternative proof of Proposition \ref{pro:structure_result_DL}, which will provide the opportunity to introduce notions used throughout the article. Our proof relies on cyclical monotonicity, and does not use the construction and manipulation of an explicit Kantorovich potential that appears in the proof of Proposition $3.1$ in \cite{di_marino_entropic_2018}. Cyclical monotonicity is a standard tool in optimal transport theory: for further details, we refer to the monograph \cite[Chapter 5]{villani_optimal_2009}.
	
	\begin{defi}\label{defi:cycl_mon}
		A set $\Gamma \subset \R^2$ is said to be cyclically monotone if, for every $n \in \N^*$ and $((x_i,y_i))_{i \in \ent{1}{n}} \in \Gamma^n$, $\sum_{i =1}^n |y_i - x_i| \leq \sum_{i=1}^n | y_{i+1}-x_i|$ (with the convention $y_{n+1} = y_1$). A measure $\pi \in \MM_+(\R^2)$ is said to be cyclically monotone if it is concentrated on a  cyclically monotone set.
	\end{defi}
	
	\begin{nott}
		For any pair $(\g_1,\g_2) \in \MM_+(\R^2)$, define $\CCC(\g_1,\g_2)$ as the set of cyclically monotone measures in $\Marg(\g_1,\g_2).$
	\end{nott}
	
	\begin{remq}
		For any pair $(\g_1,\g_2) \in \MM_+^2$, $\CCC(\g_1,\g_2)$ is non-empty. Indeed, denoting by $G_{\g_1}$ and $G_{\g_2}$ the quantile functions of $\g_1$ and $\g_2$, respectively, the reader may verify that the monotone transport plan $(G_{\g_1},G_{\g_2})_\# \left( \1_{[0,1]} \cdot \LL^1 \right)$ belongs to $\CCC(\g_1,\g_2).$
	\end{remq}
	
	The following result is classical; see \cite[Theorem $5.10$]{villani_optimal_2009} for a proof with more general cost functions.
	
	\begin{them}\label{them:cyclicality_optimality}
		Assume that $(\g_1,\g_2) \in \MM_+(\R)^2$ and $W_1(\g_1,\g_2) < + \infty$. Then $\OO(\g_1,\g_2) = \CCC(\g_1,\g_2).$
	\end{them}
	
	The following remark states that a sub-measure of a cyclically monotone transport plan remains cyclically monotone for its own marginals. We next derive a corresponding result for optimal transport plans. For any measurable space $(E,\TT)$ and any pair $(\g_1,\g_2) \in \MM_+(E)^2$, we write $\g_1 \leq \g_2$ if, for all $A \in \TT$, $\g_1(A) \leq \g_2(A)$. 
	
	\begin{remq}[Cyclicity and optimality of sub-measures]\label{remq:cycl_res}
		\begin{enumerate}
			\item Consider $(\g_1,\g_2) \in \MM_+(\R)^2$, $\pi \in \CCC(\g_1,\g_2)$ and $\pi^* \in \MM(\R^2).$ If $\pi^* \leq \pi$, then $\pi^* \in \CCC({p_1}_\# \pi^*, {p_2}_\# \pi^*).$ Indeed, if $\Gamma$ is a cyclically monotone set on which $\pi$ is concentrated, then $\pi^*$ is also concentrated on $\Gamma.$
			\item Consider $(\g_1,\g_2) \in \MM_+(\R)^2$ such that $W_1(\g_1,\g_2) < + \infty$, $\pi \in \OO(\g_1,\g_2)$, and $\pi^* \in \MM_+(\R^2)$. If $\pi^* \leq \pi,$ then $\pi^* \in \OO({p_1}_\#\pi^*,{p_2}_\# \pi^*).$ Indeed, since 
			\begin{equation*}
				W_1\left({p_1}_\# \pi^*, {p_2}_\#\pi^*\right) \leq \int_{\R^2} |y-x|\dd \pi^* (x,y) \leq \int_{\R^2} |y-x|\dd \pi (x,y) = W_1(\g_1,\g_2)  < + \infty,
			\end{equation*}
			by Theorem \ref{them:cyclicality_optimality}, $\CCC(\g_1,\g_2) = \OO(\g_1,\g_2)$ and $\CCC({p_1}_\#\pi^*,{p_2}_\# \pi^*) = \OO({p_1}_\#\pi^*,{p_2}_\# \pi^*) $. Thus, $\pi^* \in \OO({p_1}_\#\pi^*,{p_2}_\# \pi^*)$ directly follows from the previous point.
		\end{enumerate}
	\end{remq}

	We now introduce the key concept of barrier point, which plays a central role in proving Proposition \ref{pro:structure_result_DL} and is used throughout the remainder of the article.
	
	\begin{defi}[Barrier points]\label{defi:barrier_points}
		For all $x \in \R$, define $\Cd(x) = (]-\infty,x[ \times ]x,+\infty[) \cup (]x,+\infty[ \times ]-\infty,x[).$
		We say that $x \in \R$ is a barrier point for $(\mu,\nu) \in \MM_+^2$ if, for all $\pi \in \CCC(\mu,\nu),$ we have $\pi(\Cd(x)) = 0.$ We denote by $\BB(\mu,\nu)$ the set of barrier points for $(\mu,\nu)$.
	\end{defi}

	Intuitively, a barrier point of $(\mu,\nu)$ is a point through which no optimal transport from $\mu$ to $\nu$ can move mass. Visually, $x \in \BB(\mu,\nu)$ if every $\pi \in \CCC(\mu,\nu)$ is concentrated in the yellow region shown in Figure \ref{fig:area_of_con_bar}. The next result provides a sufficient condition for a point to be a barrier point. Its proof relies on cyclical monotonicity and is postponed to the next subsection, following Proposition \ref{pro:barrier_points}, which presents a more general statement.

	\begin{figure}[h]
		\begin{center}
			\def\svgwidth{10cm}
			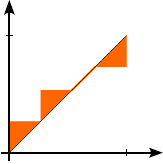
			\caption{Illustration of the concentration area of optimal transport plans.}
			\label{fig:area_of_con}
		\end{center}
	\end{figure}

	\begin{pro}\label{pro:barier_point_continuous_inclusion}
		Let $\mu,\nu \in \MM_+(\R)$ satisfy $W_1(\mu,\nu) < + \infty.$ The inclusion $\{F_\mu^+ = F_\nu^+\} \subset \BB(\mu,\nu)$ is satisfied.
	\end{pro}

	This proposition is an important step toward the proof of Proposition \ref{pro:structure_result_DL}. To see this, observe that in Example \ref{ex:unif_DL}, $1/4$ and $1/2$ lie in $\{F_\mu^+ = F_\nu^+\}$, and thus both belong to $\BB(\mu,\nu)$. Hence, the mass cannot exit the connected component $]1/4,1/2[$, which is precisely the content of Point \ref{Point:structure_avant} of Proposition \ref{pro:structure_result_DL}. To prove the ``moving forward'' part of Point \ref{Point:structure_avant}, we use the fact that, for measures in the stochastic order, a transport is optimal for the distance cost if and only if all the mass moves forward (see Proposition \ref{pro:ordre_stochastique}).

	\begin{defi}[Stochastic Order]
		We say that $\g_1 \in \MM_+(\R)$ is smaller than $\g_2 \in \MM_+(\R)$ in the stochastic order if, for all  bounded increasing function $f:\R \to \R$, $\int_{\R} f \dd\g_1 \leq \int_{\R} f \dd\g_2$. In this case, we write $\g_1 \leq_{\st} \g_2$.
	\end{defi}

	It is well known that $\g_1 \leq_{\st} \g_2$ if and only if both $\g_1(\R) = \g_2(\R)$ and $F_{\g_1}^+ \geq F_{\g_2}^+.$ We now recall the Strassen-type theorem for the order $\leq_{\st}$, along with the characterization of the set of transport plans for measures in the stochastic order. We refer the reader to \cite[Proposition 4.1]{kellerer_order_1986} for a proof of the first point. The second point will be proved in Subsection \ref{subsection:general_dec}, immediately following Proposition \ref{pro:barrier_points}.
	
	\begin{pro}\label{pro:ordre_stochastique}
		Consider $(\g_1,\g_2) \in \MM_+(\R)^2$.
		\begin{enumerate}
			\item \textbf{Strassen-type theorem for $\leq_{\st}$:}\label{point:Strasse_sto}  $\g_1 \leq_{\st} \g_2 \iff \ \Marg_{\F}(\g_1,\g_2) \neq \emptyset$. 
			\item \textbf{Characterization of optimality:}\label{point:caracterisation_sto} If $\g_1 \leq_{\st} \g_2$, then we have $\CCC(\g_1,\g_2) = \Marg_{\F}(\g_1,\g_2)$. When $W_1(\g_1,\g_2) < + \infty$, we obtain $\OO(\g_1,\g_2) = \Marg_{\F}(\g_1,\g_2)$.
		\end{enumerate}
	\end{pro}

	We prove now the structure result of Di Marino--Louet. 
	
	\begin{proof}[\textbf{Proof of Proposition \ref{pro:structure_result_DL}}]
		Let $\Cd(x)$ be defined as in Definition \ref{defi:barrier_points}, and define $B = \left(\bigcup_{k \in \KK^+}\{a_k^+,b_k^+\}\right) \cup \left(\bigcup_{k \in \KK^-}\{a_k^-, b_k^-\} \right)$, $S_\pi = \spt(\pi) \cap (B^c \times B^c)$. Since $\mu$, $\nu$ are atomless and $B$ is countable, $\pi(S_\pi) = 1.$ To prove Point \ref{Point:structure_avant}, fix $k \in \KK_+$. Since $a_k^+$ and $b_k^+$ belong to $\{F_\mu^+ = F_\nu^+\}$, by Proposition \ref{pro:barier_point_continuous_inclusion}, it follows that they also belong to $\BB(\mu,\nu)$. Thus $\pi \left(\Cd({a_k^+}) \cup \Cd(b_k^+)\right) =0$, which implies $\spt(\pi) \subset [\Cd({a_k^+}) \cup \Cd(b_k^+)]^c = ]-\infty,a_k^+]^2 \cup [a_k^+,b_k^+]^2 \cup [b_k^+,+\infty[^2,$ and therefore:
		\begin{equation}\label{eq:inclusion_F}
			\begin{cases*}
				S_\pi \cap (]a_k^+,b_k^+[\times \R) \subset ]a_k^+,b_k^+[^2 \subset \R \times ]a_k^+,b_k^+[\\
				S_\pi \cap (\R \times ]a_k^+,b_k^+[) \subset ]a_k^+,b_k^+[^2 \subset  ]a_k^+,b_k^+[ \times \R
			\end{cases*}.
		\end{equation}
		To complete the proof of Point \ref{Point:structure_avant}, it remains to show that $S_\pi \cap (]a_k^+,b_k^+[\times \R)$ is contained in $\F$. Let us define  $\pi_k^+  = \pi_{\res ]a_k^+,b_k^+[^2}$, $\mu_k^+ = \mu_{\res ]a_k^+,b_k^+[}$ and $\nu_k^+ = \nu_{\res ]a_k^+,b_k^+[}.$ As $\pi$ is concentrated on $\spt(\pi)$, the first inclusion of the first line and the second inclusion of the second line of Equation \eqref{eq:inclusion_F} imply that $\pi_k^+ = \pi_{\res  ]a_k^+,b_k^+[\times \R}$. Similarly, the second inclusion of the first line and the first inclusion of the second line of Equation \ref{eq:inclusion_F} yield $\pi_k^+ = \pi_{\res \R \times  ]a_k^+,b_k^+[}.$ Thus, $\pi_k^+ \in \Marg(\mu_k^+,\nu_k^+)$. By Remark \ref{remq:cycl_res}, it follows that $\pi_k^+ \in \OO(\mu_k^+,\nu^+_k)$. Since  $\mu_k^+(\R) = F_{\mu_k^+}(b_k^+) - F_{\mu_k^+}(a_k^+) = F_{\nu_k^+}(b_k^+) - F_{\nu_k^+}(a_k^+) = \nu_k^+(\R)$ and $F_{\mu_k^+}- F_{\nu^+_k} = \1_{]a_k^+,b_k^+[}(F_\mu^+ - F_\nu^+) \geq 0$, we conclude that $\mu_k^+ \leq_\st \nu_k^+$. By Proposition \ref{pro:ordre_stochastique}, we obtain $\pi_k^+(\F^c) = 0$, \ie, $\pi(\F^c \cap (]a_k^+,b_k^+[ \times \R))= 0.$ Thus, $\spt(\pi) \subset \F \cup (]a_k^+,b_k^+[ \times \R)^c$, which implies $S_\pi \cap (]a_k^+,b_k^+[ \times \R) \subset \F$, thereby completing the proof of Point \ref{Point:structure_avant} in Proposition \ref{pro:structure_result_DL}. Since the proof of Point \ref{Point:structure_arriere} is similar, it remains to prove Point \ref{Point:structure_fixe}. Consider $(x,y) \in [S_\pi \cap (\{F_\mu^+ = F_\nu^+\} \times \R)]$ and suppose  $x<y.$ As $x \notin B$, the situations $]x,y[ \subset \{F_\mu^+< F_\nu^+\}$ and $]x,y[ \subset \{F_\mu^+> F_\nu^+\}$ are excluded. Since $F_\mu^+$ and $F_\nu^+$ are continuous, there exists $z \in ]x,y[$ such that $F_\mu^+(z) = F_\nu^+(z).$ Proposition \ref{pro:barier_point_continuous_inclusion} implies $\pi(\Cd(z)) = 0$, whereas $(x,y) \in \Cd(z)$. This contradicts $(x,y) \in \spt(\pi)$, therefore $x \geq y$. Similarly $x \leq y$, which proves $x=y$ and completes the proof of Point \ref{Point:structure_fixe}.
	\end{proof}
	
	\begin{remq}
		Our alternative proof of the statement of Di Marino and Louet requires only $W_1(\mu,\nu) < +\infty$, rather than compact support of the measures. In Theorem \ref{them:deco_cas_sans_atomes}, we complete this result by a decomposition result. We will generalize Theorem \ref{them:deco_cas_sans_atomes} for marginals that may have atoms in Theorem \ref{pro:decomposition} (and $\CCC(\mu,\nu)$ instead of $\OO(\mu,\nu)$). 
	\end{remq}
	
	We fix the notation used in the previous proof concerning the components of the marginals.
	
	\begin{nott}[Components of the marginals]\label{nott:marg_atomless}  For all $k \in \KK^+$ (resp. $k \in \KK^-$), define $\mu_k^+ = \mu_{\res ]a_k^+,b_k^+[}$, $\nu_k^+ = \nu_{\res ]a_k^+,b_k^+[}$ (resp.  $\mu_k^- = \mu_{\res ]a_k^-,b_k^-[}$, $\nu_k^- = \nu_{\res ]a_k^-,b_k^-[}$). Then, define $\mu^= = \mu_{\res \{ F_\mu^+ = F_\nu^+\}}$ and $\nu^= = \nu_{\res \{F_\mu^+ = F_\nu^+\}}$.
	\end{nott}

	\begin{them}\label{them:deco_cas_sans_atomes}
		Consider a pair $(\mu,\nu) \in \PP(\R)^2$ of atomless measures with compact support. Using the notation introduced in Notation \ref{nott:marg_atomless}, we define $\PPP =  \prod_{k \in \KK^+} \OO({\mu}_k^+,{\nu}_k^+) \times \prod_{k \in \KK^-} \OO({\mu}_k^-, {\nu}_k^-) \times \OO({\mu}^=, {\nu}^=).$
		\begin{enumerate}
			\item For all $k \in \KK^+$ (resp. $\KK^-$), we have $\mu_k^+ \leq_{\st} \nu_k^+$ (resp. $\nu_k^- \leq_{\st} \mu_k^-$), and $\mu^= = \nu^=$.
			\item The map $\varphi :  \PPP \to \OO(\mu,\nu)$ defined by 
			$$\varphi((\pi_k^+)_{k \in \KK^+}, (\pi_k^+)_{k \in \KK^+},\pi^=) = \sum_{k \in \KK^+ } \pi_k^+ + \sum_{k \in \KK^- } \pi_k^- + \pi^=$$ is a bijection. Furthermore, its inverse $\varphi^{-1} : \OO(\mu ,\nu) \to \PPP $ is given by 
			\begin{equation}\label{eq:formul_inverse_con}
				\varphi^{-1}(\pi) = \left( \left(\pi_{\res ]a_k^+,b_k^+[^2} \right)_{k \in \KK^+},  \left(\pi_{\res ]a_k^-,b_k^-[^2} \right)_{k \in \KK^-},\pi_{\res \{F_\mu^+ = F_\nu^+\}^2} \right).
			\end{equation}
		\end{enumerate}
	\end{them}
	
	\begin{proof}
		\begin{enumerate}
			\item From the proof of Proposition \ref{pro:structure_result_DL}, for all $k \in \KK^+$, we have $\mu_k^+ \leq_{\st} \nu_k^+$ and similarly, for all $k \in \KK^-$, $\nu_k^- \leq_{\st} \mu_k^-.$ By Point \ref{Point:structure_fixe} of Proposition \ref{pro:structure_result_DL}, $\pi_{\res \{F_\mu^+ = F_\nu^+\} \times \R}$ is concentrated on $\D$ and its first marginal is $\mu^=$. Therefore, $\pi_{\res \{F_\mu^+ = F_\nu^+\} \times \R} =(\id,\id)_\# \mu^= = \pi_{ \res \{  F_\mu^+ = F_\nu^+\}^2}$. Similarly $\pi_{\res \R \times \{F_\mu^+ = F_\nu^+\}}=(\id,\id)_\# \nu^= = \pi_{ \res \{  F_\mu^+ = F_\nu^+\}^2}$. Hence $\mu^= = \nu^=$.
			\item To prove that $\varphi$ is surjective, fix $\pi \in \OO(\mu,\nu).$ We established in the proof of Proposition \ref{pro:structure_result_DL} that, for all $k \in \KK^+$, $\pi_k^+$, defined by $\pi_k^+ = \pi_{\res ]a_k^+,b_k^+[^2}$ belongs to $\OO(\mu_k^+,\nu_k^+)$. Similarly, for all $k \in \KK^-$, $\pi_k^-$, defined by $\pi_k^- = \pi_{\res ]a_k^-,b_k^-[^2}$ belongs to $\OO(\mu_k^-,\nu_k^-)$. By the first point, $\pi^= := (\id,\id)_\# \mu^=$ belongs to $\OO(\mu^=,\nu^=).$ Hence $\FF$, defined by $\FF =  ((\pi_k^+)_{k \in \KK^+}, (\pi_k^-)_{k \in \KK^-}, \pi^=)$ belongs to $\PPP.$ Moreover, the measure $\ti{\pi}$, defined as $\ti{\pi} =\sum_{k \in \KK^+} \pi_k^+ + \sum_{k \in \KK^-} \pi_k^- + \pi^=$ has first marginal $\sum_{k \in \KK^+} \mu_k^+ + \sum_{k \in \KK^-}\mu_k^- + \mu^= = \mu$ and corresponds to the restriction of $\pi$ to $\biguplus_{k \in \KK^+} (]a_k^+,b_k^+[^2) \biguplus \biguplus_{k \in \KK^-}(]a_k^-,b_k^-[^2) \biguplus (\{F_\mu^+ = F_\nu^+\}^2 )$. Thus, $\pi(\R^2) = \mu(\R) = \ti{\pi}(\R)$ and $\ti{\pi} \leq \pi$, which implies $\ti{\pi} = \pi.$ Finally, we have shown that $\FF \in \PPP$ and $\varphi(\FF) = \pi$. Therefore, $\varphi$ is surjective. Note that, as $\pi \in \OO(\mu,\nu)$ and $\FF \in \PPP$, integrating $(x,y) \mapsto |y-x|$ along the equality $\varphi(\FF) = \pi$ yields $W_1(\mu,\nu) = \sum_{k \in \KK^+} W_1(\mu_k^+,\nu_k^+) + \sum_{k \in \KK^-}W_1(\mu_k^-,\nu_k^-).$ This ensures that $\varphi$ takes values in $\OO(\mu,\nu)$. Indeed, for every $\FF = ((\pi_k^+)_{k \in \KK^+},(\pi_k^-)_{k \in \KK^-},\pi^=)$, we have $\varphi(\FF) \in \Marg(\mu,\nu)$ and
			\begin{align*}
				J(\varphi(\FF)) &= \sum_{k \in \KK^+} \int_{\R^2} |y-x| \dd \pi_k^+(x,y) + \sum_{k \in \KK^-} \int_{\R^2} |y-x| \dd \pi_k^-(x,y) + \int_{\R^2} |y-x| \dd \pi^=(x,y)\\
				&= \sum_{k \in \KK^+} W_1(\mu_k^+,\nu_k^+) + \sum_{k \in \KK^-} W_1(\mu_k^-,\nu_k^-) = W_1(\mu,\nu),
			\end{align*}
			which implies $\varphi(\FF) \in \OO(\mu,\nu).$ The injectivity of $\varphi$ and the validity of Equation \eqref{eq:formul_inverse_con} follow directly from the fact that  $\{]a_k^+,b_k^+[^2\}_{k \in \KK^+} \cup \{]a_k^-,b_k^-[^2\}_{k \in \KK^-} \cup \{\{F_\mu^+ = F_\nu^+\}^2\}$ forms a collection of disjoint sets.\qedhere
		\end{enumerate}
	\end{proof}

	\subsection{Generalized decomposition result}\label{subsection:general_dec}
	
	In this subsection, we aim to extend Theorem \ref{them:deco_cas_sans_atomes} to any pair of measures, including those with atoms. More precisely, for any pair $(\mu,\nu) \in \MM^2_+$, we will prove a decomposition for $\CCC(\mu,\nu)$. Note that we neither require the measures to be atomless nor that \( W_1(\mu, \nu) < +\infty \), and that we decompose \( \CCC(\mu, \nu) \) instead of \( \OO(\mu, \nu) \). By Theorem \ref{them:cyclicality_optimality}, this result also covers the decomposition of $\OO(\mu,\nu)$ when $W_1(\mu,\nu) < +\infty.$ Note that $\mu(\R) = \nu(\R)$ is the minimal assumption possible: otherwise $\Marg(\mu,\nu)$ is empty and there is nothing to investigate: this is implicitly assumed from now on, as we consider pair of marginals in $\MM_+^2$ (see Notation \ref{subsection:not}, Point 4). The following result extends Proposition
	\ref{pro:barier_point_continuous_inclusion} and serves as a fundamental step toward our decomposition.
	
	\begin{pro}\label{pro:barrier_points}
		Consider $(\mu,\nu) \in \MM^2_+$, $x \in \R$ and $ \pi \in \CCC(\mu,\nu).$
		\begin{enumerate}
			\item\label{point:bar_geq_pos} If $F_\mu^+(x) \geq F_\nu^+(x)$, then $\pi(]x,+\infty[\times ]-\infty,x]) = 0.$
			\item\label{point:bar_geq_neg} If $F_\mu^-(x) \geq F_\nu^-(x)$, then $\pi([x,+\infty[\times ]-\infty,x[) = 0.$
			\item\label{point:bar_leq_pos} If $F_\mu^+(x) \leq F_\nu^+(x)$, then $\pi(]-\infty,x]\times ]x,+\infty[) = 0.$
			\item\label{point:bar_leq_neg} If $F_\mu^-(x) \leq F_\nu^-(x)$, then $\pi(]-\infty,x[\times[x,+\infty[) = 0.$	
		\end{enumerate}
	\end{pro}
	
	\begin{proof}
		Assume that $x$ satisfies $F_\mu^+(x) \geq F_\nu^+(x)$. To reach a contradiction, suppose that $\pi(]x,+\infty[\times ]-\infty,x]) > 0$. Observe that $\pi(]-\infty,x] \times ]-\infty,x]) + \pi(]-\infty,x] \times ]x,+ \infty[) = F_\mu^+(x) \geq F_\nu^+(x) = \pi(]-\infty,x] \times ]-\infty,x]) + \pi(]x,+\infty[\times ]-\infty,x]) $, which implies $\pi(]-\infty,x] \times ]x,+\infty[) \geq \pi(]x,+\infty [ \times]-\infty,x]) >0.$ Let  $\Gamma \in \BB(\R^2)$ be a cyclically monotone set on which $\pi$ is concentrated. Then, $\pi(\Gamma \cap (]x,+\infty[ \times]-\infty,x])) = \pi(]x,+\infty[ \times]-\infty,x]) >0$ and $\pi(\Gamma \cap (]-\infty,x] \times ]x,+\infty[) ) = \pi(]-\infty,x] \times ]x,+\infty[) >0$. Hence, there exists $(x_1,y_1) \in (]x,+\infty[ \times ]-\infty,x]) \cap \Gamma$ and $(x_2,y_2) \in (]-\infty,x] \times ]x, + \infty[)\cap \Gamma.$ Define $a = \min(x_2,y_1), b = \max(x_2,y_1), c = \min(x_1,y_2)$ and $d = \max(x_1,y_2).$ We have the inequalities $x_2 \leq x < y_2$, $y_1 \leq x <x_1$, and $a \leq b \leq x < c \leq d$. Thus, $ |y_2 - x_1| + |y_1 - x_2| = (d-c) + (b-a) < (d-c) + 2(c-b) + (b-a) = (d+c)-(a+b) = (y_2 + x_1) - (y_1+x_2) = (y_2-x_2) + (x_1-y_1) = |y_2 - x_2| + |y_1- x_1|,$ which contradicts the cyclical monotonicity of $\Gamma$. Therefore, $\pi(]x,+\infty[\times ]-\infty,x]) = 0$. The reader may verify that the arguments for the three other points are analogue.
	\end{proof}
	
	\begin{remq}\label{remq:just_two}
		Observe that the proof of Proposition \ref{pro:barrier_points} only relies on the cyclical monotonicity for two pairs of points. Indeed, we used only that the set $\Gamma$ on which $\pi$ is concentrated satisfies: $$\forall (x_1,y_1),(x_2,y_2) \in \Gamma, |y_1-x_1| + |y_2-x_2| \leq |y_1-x_2| + |y_2-x_1|.$$
	\end{remq}
	
	By Points \ref{point:bar_geq_neg} and \ref{point:bar_leq_pos}, we deduce $\{F_\mu^+ = F_\nu^+\} = \{F_\mu^+ \geq F_\nu^+\} \cap \{F_\mu^+ \leq F_\nu^+\} \subset \BB(\mu,\nu)$, thus establishing Proposition \ref{pro:barier_point_continuous_inclusion}. For a more general result, see Proposition \ref{pro:caracterization_barrier} and Remark \ref{remq:computation_E^=}, where the following equality is shown:
	\begin{equation}\label{eq:E_==_B_modifie}
		\BB(\mu,\nu) = \{F_\mu^+ = F_\nu^+\} \cup  \{F_\mu^- = F_\nu^-\} \cup \{F_\nu^- < F_\mu^- \leq F_\mu^+ < F_\nu^+\} \cup  \{F_\mu^- < F_\nu^- \leq F_\nu^+ < F_\mu^+\}.
	\end{equation}
	
	We now apply Proposition \ref{pro:barrier_points} to establish Point \ref{point:caracterisation_sto} of Proposition \ref{pro:ordre_stochastique}.
	
	\begin{proof}[\textbf{Proof of Point \ref{point:caracterisation_sto} of Proposition \ref{pro:ordre_stochastique}}]
		Assume $\pi \in \Marg_{\F}(\mu,\nu)$. For all $n \geq 1$ and $((x_i,y_i))_{i \in \ent{1}{n}} \in \F^n$, we have $\sum_{i=1}^n |y_i -x_i| = \sum_{i=1}^n y_i - x_i = \sum_{i=1}^n y_{i+1} - x_i \leq \sum_{i=1}^n |y_{i+1} - x_i|$. Hence $\F$ is cyclically monotone. Since $\pi(\F^c) = 0$, $\pi \in \CCC(\mu,\nu)$. Assume now $\pi \in \CCC(\mu,\nu)$. Since $F_\mu^+ \geq F_\nu^+$, by Point \ref{point:bar_geq_pos} of Proposition \ref{pro:barrier_points}, for all $x \in \R$, $\pi(]x,+\infty[ \times ]-\infty,x]) = 0.$ Since $\F^c = \cup_{r \in \Q} ]r,+\infty[ \times ]-\infty,r]$, we obtain $\pi(\F^c) = 0.$ The last line follows directly from Theorem \ref{them:cyclicality_optimality}.
	\end{proof}
	We have shown that $\F$ is a cyclically monotone set, and by a similar argument, $\ti{\F}$ is also cyclically monotone.	
	\begin{remq}
		When $(\mu,\nu) \in \MM_1(\R)^2$, Point \ref{point:caracterisation_sto} of Proposition \ref{pro:ordre_stochastique} seems to be well known. In this case $\CCC(\mu,\nu) = \OO(\mu,\nu)$ and the proof of $\Marg_\F(\mu,\nu) = \OO(\mu,\nu)$ relies on the equality of the Jensen inequality. For all $\pi \in \Marg(\mu,\nu)$, $\int |y-x| \dd \pi \geq \left| \int y -x \dd \pi(x,y) \right| = \int y \dd\nu(y) - \int x \dd\mu(x)$, and the equality holds if and only if $\pi \in \Marg_\F(\mu,\nu)$.
	\end{remq}
	As in the decomposition of Di Marino and Louet, we begin by partitioning the real line according the relative position of the cumulative distribution functions. Next, we define a set $\{(\mu_k^+,\nu_k^+)\}_{k \in \KK^+} \cup \{(\mu_k^-,\nu_k^-)\}_{k \in \KK^-} \cup \{(\mu^=,\nu^=)\}$ of measure components of $(\mu,\nu)$, and finally, we prove that $$\CCC(\mu,\nu) = \left(\bigoplus_{k \in \KK^+} \CCC(\mu_k^+,\nu_k^+) \right) \oplus \left(\bigoplus_{k \in \KK^-} \CCC(\mu_k^-,\nu_k^-) \right) \oplus \CCC(\mu^=,\nu^=).$$

		Naturally, when dealing with measures that have atoms, the approach of Di Marino and Louet must be adapted. First, for the decomposition of the real line, observe that $\{F_\mu^+ > F_\nu^+\}$ need not be open in general, making it difficult to consider the connected components. More importantly, if we consider the connected components of $\{F_\mu^+> F_\nu^+\}$ and define the components of the marginal by restriction of $\mu$ and $\nu$ to the connected components, our decomposition would not be sufficiently fine. For instance, consider $\mu = \1_{]0,1[} \cdot \LL^1 + 2 \d_1$ and $\nu = 2 \1_{]1/2,2[}\cdot \LL^1.$ Since $\{F_\mu^+ > F_\nu^+\} = ]0,2[$, the set of component would be $\{(\mu,\nu)\}$, which is associated to the equality $\CCC(\mu,\nu) = \CCC(\mu,\nu).$ However, noting that $1$ belongs to $\{F_\mu^- = F_\nu^-\} \subset \BB(\mu,\nu)$, the reader may verify that $\CCC(\mu,\nu)$ can be decomposed as
	$$\CCC(\mu,\nu) = \CCC(\mu_1, \nu_1) \oplus \CCC(\mu_2,\nu_2),$$
	where $\mu_1= \1_{]0,1[} \cdot \LL^1 $, $ \nu_1 = 2 \1_{[1/2,1]} \cdot \LL^1$, $\mu_2 = 2 \d_1$ and $\nu_2 = 2 \1_{]1,2[} \cdot \LL^1$. This shows that the decomposition of the real line obtained by taking the connected components of $\{F_\mu^+> F_\nu^+\}$ is not fine enough. The following definition addresses this issue by considering the signs of both $F_\mu^+-F_\nu^+$ \emph{and} $F_\mu^- - F_\nu^-$, rather than solely the sign of $F_\mu^+-F_\nu^+.$
	\begin{defi}[Components of the real line]\label{defi:compo_real_line}
		Consider $(\mu,\nu) \in \MM^2_+.$
		\begin{enumerate}
			\item Define $E^+ = \{F_\mu^+ > F_\nu^+\} \cap \{F_\mu^- > F_\nu^-\}$, $E^- = \{F_\mu^+ < F_\nu^+\} \cap \{F_\mu^- < F_\nu^-\}$  and $E^= = \R \setminus (E^+ \cap E^-)$. Since $F_\mu^+-F_\nu^+$ and $F_\mu^- - F_\nu^-$ are left-continuous and right-continuous, respectively, the reader may verify that $E^-$ and $E^+$ are open. Note that $E^=$ is closed. Since $E^+$ and $E^-$ are disjoint, $(E^+,E^-,E^=)$ forms a disjoint covering of $\R.$
			\item Since $E^+$ and $E^-$ are open sets, each admits a countable family of open connected component. Let $(]a_k^+,b_k^+[)_{k \in \KK^+}$ and $(]a_k^-,b_k^-[)_{k \in \KK^-}$ denote the family of connected component of $E^+$ and $E^-$, respectively.
			\item Let $B_l^+$  denote the set $\ens{a_k^+}{k \in \KK^+}$ of the left boundary points of connected components of $E^+$. Similarly, define $B_r^+ = \ens{b_k^+}{k \in \KK^+}$, $B_l^- = \ens{a_k^-}{k \in \KK^-}$ and $B_r^- = \ens{b_k^-}{k \in \KK^-}$. Then, we denote by $B_l = B_l^+ \cup B_l^-$,  $B_r = B_r^+ \cup B_r^-$ and $B = B_l \cup B_r$ the sets of left boundaries, right boundaries and overall boundaries of the connected components of $E^+$ and $E^-$, respectively. Since these connected components form a countable family, $B$ is countable. Moreover, note that $B \subset E^=$.
		\end{enumerate}
	\end{defi}
	\begin{ex}\label{ex:guiding_example}
		The following example (see Figure \ref{fig:point_splitting}) which will be referenced multiple times. Define 
		\begin{equation}
			\begin{cases}
				\mu = \1_{[0,1]} \cdot \LL^1 + \d_1 + \1_{[1,2]} \cdot \LL^1 + \d_2 +  \1_{[2,3]} \cdot \LL^1 + \d_3 +  \1_{[3,4]} \cdot \LL^1 + \d_4 +  \1_{[4,5]} \cdot \LL^1 + \d_5 +  \1_{[5,6]} \cdot \LL^1 \\
				\nu =  \1_{[\frac{1}{2},1]} \cdot \LL^1  + \d_1 + \1_{[1,2]} \cdot \LL^1 + 2\d_2 +  \frac{1}{2} \1_{[2,3]} \cdot \LL^1 + \frac{3}{2}\d_3 + \frac{1}{2} \1_{[3,4]} \cdot \LL^1 + \d_4 +  \1_{[4,5]} \cdot \LL^1 + \d_5 +  \1_{[5,6]} \cdot \LL^1
			\end{cases}.
		\end{equation}
		Then, we obtain $E^+ = ]0,2[$, $E^- = ]2,3[ \cup ]3,4[$, and $E^= = ]-\infty,0] \cup \{2,3\} \cup [4,+\infty[.$ Hence, the boundary sets are given by $B_l^+  = \{0\}, B_l^- = \{2,3\}, B_r^+ = \{2\}, B_r^- =\{3,4\} , B_l = \{0,2,3\}, B_r = \{2,4\}$, and $B = \{0,2,3,4\}.$
	\end{ex}
	\begin{remq}\label{remq:deco_real_line}
		\begin{enumerate}
			\item Observe that the dependence of $E^+$, $E^-$, $E^=$, $a_k^+$, $b_k^+$, $B_l^+$, $B_l^-$, $B_r^+$, $B_r^-$, $B_l$, and $B_r$ on $(\mu,\nu)$ has been omitted here for clarity. We also write $\BB$ instead of $\BB(\mu,\nu)$ (see Definition \ref{defi:barrier_points}). In cases of ambiguity, we make the dependence explicit --- for example by writing $E^+(\mu,\nu)$ instead of $E^+$.
			\item\label{point:sym_deco} Our decomposition is symmetric: reversing the transport from $\nu$ to $\mu$ conserves the components, up to a sign change. More precisely, $E^{\pm}(\mu,\nu) = E^\mp(\nu,\mu)$, $E^=(\mu,\nu) = E^=(\nu,\mu)$, $B_l^\pm(\mu,\nu) = B_l^{\mp}(\nu,\mu)$, $B_r^\pm(\mu,\nu) = B_r^{\mp}(\nu,\mu)$, $B_l(\mu,\nu) = B_l(\nu,\mu)$, $B_r(\mu,\nu) = B_r(\nu,\mu)$ and $B(\mu,\nu) = B(\nu,\mu)$.
			\item Note that $(B_l^+,B_l^-,B_l^c)$ and $(B_r^+,B_r^-,B_r^c)$ are two disjoint coverings of $\R$. Note that their nine pairwise intersection -- namely $B_l^+ \cap B_r^+, B_l^+ \cap B_r^-, B_l^- \cap B_r^+, B_l^- \cap B_r^- ,B_l^+ \setminus B_r, B_l^- \setminus B_r, B_r^+ \setminus B_l, B_r^- \setminus B_l$ and $B_l^c \cap B_r^c$ -- can be non-empty in general. For example, in Example \ref{ex:guiding_example}, we have $0 \in B_l^+ \cap B_r^c$, $ 2 \in B_l^- \cap B_r^+$, $3 \in B_r^- \cap B_l^-$, and $4 \in B_r^- \cap B_l^c.$ 
			\item Note that $B$ represents the set of boundaries points of connected components of $E^+ \cup E^-$, but does not coincide with the topological boundary of that set. For instance, consider $\mu = \sum_{n \geq 0} \frac{1}{2^{n+1}} \d_{1- \frac{1}{2^n}}$ and $\nu = \1_{[0,1]}\cdot \LL^1$, where the set $E^+(\mu,\nu)$ is given by $E^+(\mu,\nu) = \cup_{n \geq 0} \left]1-\frac{1}{2^n},1- \frac{1}{2^{n+1}}\right[$. Therefore, $1$ lies in the topological boundary of $E^+$, but is not an element of $B.$
			\item The relative position of $F_\mu^- $ and $F_\nu^+$ is not prescribed in $E^+(\mu,\nu)$. For instance, in Example \ref{ex:guiding_example}, $1/2 \in E^+(\mu,\nu) \cap \{F_\mu^- > F_\nu^+\}$ whereas $ 1 \in E^+(\mu,\nu) \cap \{ F_\mu^- < F_\nu^+ \}.$
		\end{enumerate}
	\end{remq}
	\begin{lem}\label{remq:preliminary_inclusion}
		Consider $(\mu,\nu) \in \MM^2_+.$ Then, $E^=(\mu,\nu) \subset \BB(\mu,\nu)$.  
	\end{lem}
	
	\begin{proof}
		We have $E^= =(E^- \cup E^+)^c = \left(E^-\right)^c \cap \left(E^+\right)^c = [\{ F_\mu^- \geq F_\nu^-\} \cup \{ F_\mu^+ \geq F_\nu^+\}] \cap [\{ F_\mu^- \leq F_\nu^-\} \cup \{ F_\mu^+ \leq F_\nu^+\}]$. By applying Proposition \ref{pro:barrier_points}, it follows that $E^= \subset \BB.$ 
	\end{proof}
	
	The inclusion $\BB(\mu,\nu) \subset E^=$ is also satisfied as we shall prove in Proposition \ref{pro:caracterization_barrier}. We now want to define the components $((\mu_k^+,\nu_k^+)_{k \in \KK^+}, (\mu_k^-,\nu_k^-)_{k \in \KK^-},(\mu^=,\nu^=))$ of $\mu$ and $\nu$ for the general case. Observe that the components of $\mu$ and $\nu$ cannot be defined solely by restricting $\mu$ and $\nu$ to the connected components of $E^+$, $E^-$ and $E^=.$ For instance, consider $\mu = \d_{-1} + 6\d_0 + \d_1$ and $\nu = \d_{-3}+2\d_{-2} + 3\d_0 + \d_{2} + \d_3$. We have $E^- = ]-3,0[$ and $E^+ = ]0,3[$.  The reader may verify that  the mass at point $0$ should split as follows: $(F_\nu^-(1)-F_\mu^-(1)) = 2$ units of mass should go from $0$ to to $[2,0[$, $(F_\mu^+(1)-F_\nu^+(1)) = 1$ unit go from $0$ to $]0,3]$ and $3$ units of $0$ does not move. The following definition of the marginal components accounts for the possibility that mass at boundary points may be shared among a fixed part and mass moving to the left or right; conversely boundary points can receive mass from the left and right. The amount of mass allocated to each part is expressed using differences of the cumulative distribution functions, consistently with our illustrative example. In the following definition and throughout the paper, when referring to intervals $[a,b],]a,b]$ or $[a,b[$, where $a$ and $b$ that may be infinite, we use the following convention:  $[-\infty,x[ = ]-\infty,x[$, $]x,+\infty] = ]x,+\infty[$  for all $x \in \R$, and $[-\infty,+\infty] = \R$. For every $\g \in \MM_+(\R)$, we also set $F_{\g}^\pm(-\infty) =0$ and $F_{\g}^\pm(+\infty) =1.$
	
	\begin{defi}[Components of the marginals]\label{def:component_marginals}
		Consider $(\mu,\nu) \in \MM^2_+.$
		\begin{enumerate}
			\item For all $k \in \KK^+$, define ${\mu}_k^+ = (F_\mu^+(a_k^+) - F_\nu^+(a_k^+)) \d_{a_k^+} + \mu_{\res ]a_k^+,b_k^+[}$ and ${\nu}_k^+ = \nu_{\res{]a_k^+,b_k^+[}} + (F_\mu^-(b_k^+)-F_\nu^-(b_k^+))\d_{b_k^+}$. Then, set ${\mu}^+ = \sum_{k \in \KK^+} {\mu}_k^+$ and ${\nu}^+ = \sum_{k \in \KK^+} {\nu}_k^+$.
			\item For all $k \in \KK^-$, define ${\mu}_k^- =  \mu_{\res ]a_k^-,b_k^-[} + \left(F_\nu^-(b_k^-)-F_\mu^-(b_k^-)\right) \d_{b_k^-}$ and ${\nu}_k^- =  (F_\nu^+(a_k^-)-F_\mu^+(a_k^-))\d_{a_k^-} + \nu_{\res{]a_k^-,b_k^-[}} $. Then, set ${\mu}^- = \sum_{k \in \KK^-} {\mu}_k^-$ and ${\nu}^- = \sum_{k \in \KK^-} {\nu}_k^-$.
			\item Define  ${\mu}_1^= = \mu_{\res B^c \cap E^=}$, ${\nu}_1^= = \nu_{\res B^c \cap E^=}$, ${\mu}_2^= = \sum_{x \in B} (\mu(x)-{\mu}^+(x)-{\mu}^-(x))\d_x$, ${\nu}_2^= = \sum_{x \in B} (\nu(x)-{\nu}^+(x)-{\nu}^-(x))\d_x$, ${\mu}^= = {\mu}_1^= + {\mu}_2^=$ and ${\nu}^= = {\nu}_1^= + {\nu}_2^=.$
			\item Define $\DD_\KK = \{(\mu_k^+,\nu_k^+)\}_{k \in \KK^+} \cup \{(\mu_k^-,\nu_k^-)\}_{k \in \KK^-} \cup \{(\mu^=,\mu^=)\}$. 
		\end{enumerate}
	\end{defi}
	
	\begin{ex}\label{ex:guiding_example_mesure}
		In Example \ref{ex:guiding_example}, we have ${\mu}_1^+ = \1_{[0,1]} \cdot \LL^1 + \d_1 + \1_{[1,2]}\cdot \LL^1$, ${\nu}_1^+ = \1_{[1/2,1]} \cdot \LL^1 + \d_1 + \1_{[1,2]}\cdot \LL^1 +\frac{1}{2} \d_2$, ${\mu}_1^- = \1_{[2,3]}\cdot \LL^1$, ${\nu}_1^- = \frac{1}{2} \d_2 + \frac{1}{2} \1_{[2,3]}\cdot \LL^1,$  ${\mu}_2^- = \1_{[3,4]}\cdot \LL^1$, ${\nu}_2^- =\frac{1}{2} \d_3 + \frac{1}{2} \1_{[3,4]},$ ${\mu}_1^= = \1_{[4,5]}\cdot \LL^1 + \d_5 + \1_{[5,6]}\cdot \LL^1={\nu}_1^=$, ${\mu}_2^= = \d_2 + \d_3 ={\nu}_2^=$, and $\mu^= = \d_2 + \d_3 + \d_4 + \1_{[4,5]}\cdot \LL^1 + \d_5 + \1_{[5,6]}\cdot \LL^1 = {\nu}^=$. Figure \ref{fig:point_splitting} illustrates how the mass of $\nu$ at point $2$ is allocated between $\nu_1^+$, $\nu^=$ and $\nu_1^-$.
	\end{ex}
	
	\begin{figure}[h]
		\begin{center}
			\def\svgwidth{10cm}
			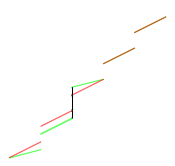
			\caption{Splitting of point 2 in Example \ref{ex:guiding_example}}
			\label{fig:point_splitting}
		\end{center}
	\end{figure}

	\begin{remq}\label{remq:on_marginals}
		\begin{enumerate}
			\item\label{point:marg_sum_compo} As expected, the measures $\mu$ and $\nu$ coincide with sum of their respective components. Indeed, for all $x \in B$,  ${\mu}^-(x) + {\mu}^+(x) + {\mu}_1^=(x) + {\mu}_2^=(x) = {\mu}^+(x) + {\mu}^-(x) + 0 + \big(\mu(x)- {\mu}^-(x) - {\mu}^+(x)\big) = \mu(x)$, which implies $\left({\mu}^- + {\mu}^+ + {\mu}_1^= + {\mu}_2^=\right)_{\res B} = \mu_{\res B}$. Moreover,
			\begin{align*}
				\mu_{\res B^c} &= \mu_{\res B^c \cap E^+} + \mu_{\res B^c \cap E^-} + \mu_{\res B^c \cap E^=} \\ &= \mu_{\res E^+} + \mu_{\res E^-} + \mu_{\res B^c \cap E^=} \\ &=  {\mu}^+_{\res B^c} + {\mu}^-_{\res B^c} + {\mu}_1^= + 0 = \left({\mu}^- + {\mu}^+ + {\mu}_1^= + {\mu}_2^=\right)_{\res B^c}.
			\end{align*}
			Therefore $\mu = {\mu}^+ + {\mu}^- + {\mu}_1^= + {\mu}_2^= = \sum_{k \in \KK^+} {\mu}_k^+ + \sum_{k \in \KK^-} {\mu}_k^- + \mu^=$. Similarly, $\nu = \sum_{k \in \KK^+} {\nu}_k^+ + \sum_{k \in \KK^-} {\nu}_k^- +{\nu}^=$.
			\item\label{point:sym_marg} The symmetry between transport from $\mu$ to $\nu$ and transport from $\nu$ to $\mu$ still holds. According to Point \ref{point:sym_deco} of Remark \ref{remq:deco_real_line}, $\left(]a_k^\mp,b_k^\mp[\right)_{k \in \KK^\mp}$  is the family of connected components of $E^\pm(\nu,\mu)$, where $\pm$ stands for a generic symbol of $+$ and $-$ and $\mp$ for the other one. It follows directly that, for each $k \in \KK^\mp$, the pair of marginal components associated to the component $\left]a_k^\mp,b_k^\mp\right[$ is $\left({\nu}_k^\mp,{\mu}_k^\mp\right).$  
			\item\label{Point:not_singular} Observe that $\{{\mu}_k^+\}_{k \in \KK^+} \cup \{\nu_k^-\}_{k \in \KK^-}$ forms a set of singular measures. However, ${\mu}_k^+$ and ${\mu}_j^-$ can both assign mass to $a_k^+$. This happens precisely when $b_j^- = a_k^+$, $F_\mu^+(a_k^+) > F_\nu^+(a_k^+)$ and $F_\nu^-(b_j^-) > F_\mu^-(b_j^-).$ Likewise, $({\mu}_k^+,{\nu}_j^+)$ and $({\mu}_k^+, {\mu}_2^=)$ may not be pairs of singular measures. 
			\item The fixed part ${\mu}^=$ of $\mu$ splits into two singular measures ${\mu}_1^=$ and ${\mu}_2^=$, that will each requiring a different mathematical treatment. The first part, ${\mu}_1^=$, represents the mass of $\mu$ that lies outside the closure of any positive of negative component of the decomposition: We will later show that this mass is fixed by every optimal transport plan. The second part, ${\mu}_2^=$, is atomic and corresponds to the part of the mass of $\mu$ located at boundary points that is not allocated to any positive or negative component of our decomposition. We explicitly compute ${\mu}_2^=(x)$ in the next lemma and later prove that ${\mu}_2^=$ measures the amount of mass at boundary points that is fixed by every optimal transport. In the atomless case, ${\mu}_2^= = 0$ and $E^=(\mu,\nu) = \{F_\mu^+ = F_\nu^+\}$. Thus ${\mu}^= = {\mu}_1^= = \mu_{ \res \{F_\mu^+ = F_\nu^+\}} = \mu_{\res E^=}$, which corresponds to the fixed part in Subsection \ref{subsection:continuous_decomposition}. In the general case, the equality ${\mu}^= = \mu_{\res E^=}$ does not hold.
		\end{enumerate}
	\end{remq}
	
	\begin{lem}\label{lem:eq_fixe_deux}
		\begin{enumerate}
			\item\label{point:partition_B}  The family $$\left(B_l^+ \cap B_r^+, B_l^+ \cap B_r^-,B_l^- \cap B_r^+, B_l^- \cap B_r^-,  B_l^+ \setminus B_r, B_l^- \setminus B_r, B_r^+ \setminus B_l, B_r^- \setminus B_l\right)$$ is a family of disjoint set  whose union is $B$.
			\item The following inclusions are satisfied:
			\begin{equation}\label{eq:inclusion_border}
				\accol{ B_l^+ \subset \{F_\mu^+ \geq F_\nu^+\} \\ B_r^+ \subset \{F_\mu^- \geq F_\nu^-\} \\ B_l^- \subset \{F_\mu^+ \leq F_\nu^+\} \\ B_r^- \subset \{F_\mu^- \leq F_\nu^-\}  \\ B_r \setminus B_l \subset \{F_\mu^+ = F_\nu^+\} \\ B_l \setminus B_r \subset \{F_\mu^- = F_\nu^-\}  }.
			\end{equation}
			\item For all $x \in B$, 
			\begin{equation}\label{eq:comp_mu+_mu-}
				\accol{{\mu}^+(x) = \1_{B_l^+}(x)(F_\mu^+(x) - F_\nu^+(x)) \\ {\nu}^+(x) = \1_{B_r^+}(x)(F_\mu^-(x) - F_\nu^-(x)) \\ {\mu}^-(x) = \1_{B_r^-}(x) (F_\nu^-(x) - F_\mu^-(x)) \\  {\nu}^-(x) = 	\1_{B_l^-}(x)(F_\nu^+(x) - F_\mu^+(x))}.
			\end{equation}
			\item We have ${\mu}_2^= = {\nu}_2^= = \sum_{x \in B} \left[ \min\left(F_\mu^+(x), F_\nu^+(x)\right) - \max\left(F_\mu^-(x),F_\nu^-(x)\right) \right]\d_x.$
		\end{enumerate}
	\end{lem}
	
	\begin{proof}
		\begin{enumerate}
			\item Since $(B_l^+,B_l^-,B_l^c)$ and $(B_r^+,B_r^-,B_r^c)$ are two families of disjoint sets covering $\R$, the family $\FF = (B_l^+ \cap B_r^+,B_l^+ \cap B_r^-, B_l^- \cap B_r^+, B_l^- \cap B_r^-, B_l^+ \setminus B_r, B_l^- \setminus B_r, B_r^+ \setminus B_l, B_r^- \setminus B_l, B_l^c \cap B_r^c)$ is a family of disjoint sets that cover $\R.$ Since the last element of $\FF$ is $B_l^c \cap B_r^c = (B_l \cup B_r)^c = B^c$ and all the other terms are subsets of $B$,
			$$B = B \cap \R = B \cap \biguplus_{C \in \FF} C = \biguplus_{C \in \FF \setminus \{B^c\}} B \cap C = \biguplus_{C \in \FF \setminus \{B^c\}} C ,$$
			which proves Point \ref{point:partition_B}. 
			\item For all $k \in \KK^+$, $]a_k^+,b_k^+[ \subset E^+ \subset \{F_\mu^+ \geq F_{\nu}^+\}$. Since $F_\mu^+$ and $F_\nu^+$ are right-continuous, we obtain $F_\mu^+(a_k^+) = \lim_{x \to a_k^+ ; x > a_k^+} F_\mu^+(x) \geq \lim_{x \to a_k^+ ; x > a_k^+} F_\nu^+(x) = F_\nu^+(a_k^+),$ establishing the first inclusion. The second, third and fourth inclusion are proved similarly. We now turn our attention to the fifth inclusion. Consider $x \in B_r \setminus B_l$. Since $x \in B_r \subset E^= \subset (E^+)^c$, if there exists $\ee> 0$ such that $]x,x+\ee[ \subset E^+$, we would have $x \in B_l^+ \subset B_l$. As $x \notin B_l$, for all $\ee > 0$, we get $]x,x+\ee[ \cap \left(E^+\right)^c \neq \emptyset.$ Thus, there exists a sequence $(x_n)_{n \geq 1} \in \prod_{n \geq 1} ]x,x+1/n]$ such that, for all $n \geq 1$, $F_\mu^-(x_n) \leq F_\nu^-(x_n)$ or $F_\mu^+(x_n) \leq F_\nu^+(x_n)$. Letting $n$ go to $+\infty$, we obtain $F_\mu^+(x) \leq F_\nu^+(x).$ Hence $B_r \setminus B_l \subset \{F_\mu^+ \leq F_\nu^+\}$ and the reader may verify that the proof of $B_r \setminus B_l \subset \{F_\mu^+ \geq F_\nu^+\}$ is similar, thus establishing the fifth inclusion. The proof of the sixth inclusion is identical.
			\item For all $x \in B$, ${\mu}^+(x) = \sum_{k \in \KK^+} {\mu}_k^+(x) = \sum_{k \in \KK^+} \1_{a_k^+ = x} (F_\mu^+(x) -F_\nu^+(x)) = \1_{B_l^+}(x)(F_\mu^+(x) -F_\nu^+(x))$. The proofs of the three other inequalities are similar.
			\item For all $x \in B$, the values of ${\mu}^+(x)$ and ${\mu}^-(x)$ can be determined  using the partition described in Point \ref{point:partition_B} and the expressions given in Equation \eqref{eq:comp_mu+_mu-}. These computations are summarized in the two first columns of Table \ref{tab:mu_sub}.  The computation of $\mu_2^=(x) = \mu(x) - {\mu}^+(x) - {\mu}^-(x)$ is then obtained by subtracting the two first columns to $\mu(x) = F_\mu^+(x)- F_\mu^-(x)$. The corresponding result is presented in the third column of Table \ref{tab:mu_sub}. Based on the partition introduced in Point  \eqref{point:partition_B}, and the expressions in Equation \eqref{eq:inclusion_border}, the values of $\min(F_\mu^+(x), F_\nu^+(x))$ and $\max(F_\mu^-(x),F_\nu^-(x))$ are computed and shown in Table \ref{tab_max_moins_min}. By subtracting the first column to the second, we recover the third column of Table \ref{tab:mu_sub}, which proves that $\mu_2^=(x) = \mu(x)-{\mu}^+(x)-{\mu}^-(x) = \min(F_\mu^+(x), F_\nu^+(x)) - \max(F_\mu^-(x),F_\nu^-(x))$.   The proof of $\nu_2^=(x) = {\nu}(x)-{\nu}^+(x)-{\nu}^-(x) = \min(F_\mu^+(x), F_\nu^+(x)) - \max(F_\mu^-(x),F_\nu^-(x))$ is similar.\qedhere
		\end{enumerate}
	\end{proof}

	\begin{table}[h!]
		\begin{center}
			\begin{tabular}{|c||c|c|c|}
				\hline $x \in \cdot$ & ${\mu}^-(x) $ & ${\mu}^+(x) $ & $\mu(x) - {\mu}^-(x) - {\mu}^+(x)$  \\
				\hline $B_l^+ \cap B_r^+$ & $0$ & $F_\mu^+(x) - F_\nu^+(x)$ &  $F_\nu^+(x) - F_\mu^-(x)$  \\
				\hline $ B_l^+ \cap B_r^-$ & $F_\nu^-(x) - F_\mu^-(x)$ & $F_\mu^+(x) - F_\nu^+(x)$ & $\nu(x)$  \\
				\hline $B_l^- \cap B_r^+$ & $0$ & $0$ & $\mu(x)$\\
				\hline $B_l^- \cap B_r^- $ & $F_\nu^-(x) - F_\mu^-(x)$ & $0$ & $F_\mu^+(x) - F_\nu^-(x)$\\
				\hline $B_l^+ \setminus B_r$ & $0$ & $F_\mu^+(x) - F_\nu^+(x)$  & $F_\nu^+(x) - F_\mu^-(x)$ \\
				\hline $B_l^- \setminus B_r$ & $0$ & $0$ & $\mu(x)$ \\
				\hline $B_r^+ \setminus B_l$ & $0$ & $0$ & $\mu(x)$  \\
				\hline $B_r^- \setminus B_l$ & $F_\nu^-(x) - F_\mu^-(x)$ & $0$ & $F_\mu^+(x) - F_\nu^-(x)$ \\
				\hline
			\end{tabular}
			\caption{Computations of $\mu(x) - {\mu}^+(x) - {\mu}^-(x)$}
			\label{tab:mu_sub}
		\end{center}
	\end{table}

	\begin{table}[h!]
		\begin{center}
			\begin{tabular}{|c||c|c|}
				\hline $x \in \cdot$ & $\min(F_\mu^+(x),F_\nu^+(x))$ & $\max(F_\mu^-(x),F_\nu^-(x))$ \\
				\hline $B_l^+ \cap B_r^+$ & $F_\nu^+(x)$ & $F_\mu^-(x)$ \\
				\hline $B_l^+ \cap B_r^-$ & $F_\nu^+(x)$ & $F_\nu^-(x)$ \\
				\hline $B_l^- \cap B_r^+$ & $F_\mu^+(x)$ & $F_\mu^-(x)$ \\
				\hline $B_l^- \cap B_r^-$ & $F_\mu^+(x)$ & $F_\nu^-(x)$ \\
				\hline $B_l^+ \setminus B_r$ & $F_\nu^+(x)$ & $F_\mu^-(x)$ \\
				\hline $B_l^- \setminus B_r$ & $F_\mu^+(x)$ & $F_\mu^-(x)$ \\
				\hline $B_r^+ \setminus B_l$ & $F_\mu^+(x)$ & $F_\mu^-(x)$ \\
				\hline $B_r^- \setminus B_l$ & $F_\mu^+(x)$ & $F_\nu^-(x)$ \\
				\hline
			\end{tabular}
			\caption{Computations of $\min(F_\mu^+(x),F_\nu^+(x))-\max(F_\mu^-(x),F_\nu^-(x)) $}
			\label{tab_max_moins_min}
		\end{center}
	\end{table}

	\begin{defi}[Components of elements of $\CCC(\mu,\nu)$]\label{defi:composantes}
		Consider $\pi \in \CCC(\mu,\nu).$
		\begin{enumerate}
			\item For all $k \in \KK^+$, define $A_k^+ = [a_k^+,b_k^+[ \times ]a_k^+,b_k^+]$ and $\pi_k^+ = \pi_{\res A_k^+}.$
			\item For all $k \in \KK^-$, define $A_k^- = ]a_k^-,b_k^-] \times [a_k^-,b_k^-[$ and $\pi_k^- = \pi_{\res A_k^-}.$
			\item Define $\pi_1^= = \pi_{\res (B^c \cap E^=) \times \R}$, $\pi_2^= = \pi_{\res \D \cap (B \times \R)}$, $A^= = [(B^c \cap E^=) \times \R] \biguplus [\D \cap (B \times \R)]$   and $\pi^= = \pi_{\res A^=} = \pi_1^= + \pi_2^=.$
		\end{enumerate}
	\end{defi}
	
	\begin{remq}
		\begin{enumerate}
			\item Observe that the sets $\{A_k^+\}_{k \in \KK^+} \cup \{A_k^-\}_{k \in \KK^-} \cup \{A^=\}$ are pairwise disjoint, which implies that the corresponding components $\{\pi_k^+\}_{k \in \KK^+}\cup \{\pi_k^-\}_{k \in \KK^-} \cup \{\pi^=\}$ forms a family of mutually singular measures.
			\item Observe that the inclusion or exclusion of boundary points when restricting $\pi$ is crucial for ensuring $\pi_k^+ \in \Marg({\mu}_k^+,{\nu}_k^+)$. Since ${\mu}_k^+$ is concentrated on $[a_k^+,b_k^+[$ and ${\nu}_k^+$ on $]a_k^+,b_k^+]$, we must restrict $\pi$ to $[a_k^+,b_k^+[ \times ]a_k^+,b_k^+]$. For instance, if $\mu = 2\d_1 + \1_{[1,2]} \cdot \LL^1 + \d_2$ and $\nu = \d_1 + \1_{[1,2]} \cdot \LL^1 + 2\d_2$, $\mu$ admits two components $\mu_1^+ = \d_1 + \1_{[1,2]} \cdot \LL^1$, $\mu^= = \d_1 + \d_2$, while $\nu$ admits two of them $\nu_1^+ = \1_{[1,2]} \cdot \LL^1 + \d_2$ and $\nu^= = \d_1 + \d_2$. One may verify that, among the 16 possible combinations of boundary inclusion, only the chosen convention yields a restriction of optimal transport plans that lies in $\Marg({\mu}_k^+,{\nu}_k^+).$
			\item Symmetry is also preserved at the level of cyclically monotone transport plans. Namely, as $i:(x,y) \mapsto (y,x)$ satisfies $i \circ i = \id$, the map $i_\# : \pi \in \PP(\R^2) \mapsto i_\# \pi\in  \PP(\R^2)$ is a bijection from $\CCC(\mu,\nu)$ to $\CCC(\nu,\mu).$ Moreover, there is a correspondence between the components of a plan $\pi \in \CCC(\mu,\nu)$ and the components of $i_\# \pi \in \CCC(\nu,\mu).$ More precisely, for all $k \in \KK^+$,  it follows from Point \ref{point:sym_deco} of Remark \ref{remq:deco_real_line}, that $]a_k^+,b_k^+[$ is a component for $E^+(\mu,\nu)$ and a component for $E^-(\nu,\mu)$. Let  $\varphi_k : \pi \in \CCC(\mu,\nu) \mapsto \pi_{\res [a_k^+,b_k^+[ \times ]a_k^+,b_k^+]}$
			and $\phi_k : \pi \in \CCC(\nu,\mu) \mapsto \pi_{\res ]a_k^+,b_k^+] \times [a_k^+,b_k^+[}$ denote the functions	mapping every element of \( \CCC(\mu,\nu) \) and \( \CCC(\nu,\mu) \), respectively, to its component associated with \( ]a_k^+,b_k^+[ \). Then, it is straightforward that $\phi_k = i_\#\circ  \varphi_k \circ i_\#$. The same results holds for $k \in \KK^-$. For each ``$=$'' type components of the decomposition of $(\mu,\nu)$ and $(\nu,\mu)$, one can go from one to another by applying $i_\#$. This is a consequence of the equalities $E^=(\mu,\nu) = E^=(\nu,\mu)$ and $B(\mu,\nu) = B(\nu,\mu)$ established in Point \ref{point:sym_deco} of Remark \ref{remq:deco_real_line}.
		\end{enumerate}
	\end{remq}
	
	The following lemma quantifies how mass splits at boundary points, specifying the amount of mass that moves strictly to the left or strictly to the right. To quantify the amount of mass moving to the right, we use the disjoint covering $B = B_l^- \uplus B_l^+ \uplus (B_r \setminus B_l)$; for mass moving on the left we base our computations on the disjoint covering $B = B_r^- \uplus B_r^+ \uplus (B_l \setminus B_r).$
	
	\begin{lem}\label{lem:mass_transfert}
		Consider $(\mu,\nu) \in \MM^2_+$, $x \in B$ and $\pi \in \CCC(\mu,\nu).$
		\begin{enumerate}
			\item\label{point:controle_bord_gauche_neg} If $x \in B_l^- \cup (B_r \setminus B_l)  $, then $\pi(\{x\} \times ]x,+\infty[) = 0.$
			\item\label{point:controle_bord_gauche_pos} If  $x \in B_l^+$, $\pi(\{x\} \times ]x,+\infty[) = F_\mu^+(x) - F_\nu^+(x).$
			\item\label{point:controle_bord_droit_pos} If $x \in B_r^+ \cup (B_l \setminus B_r)$, $\pi(\{x\}\times ]-\infty,x[) = 0.$
			\item\label{point:controle_bord_droit_neg} If $x \in B_r^-$, $\pi(\{x\}\times ]-\infty,x[) = F_\nu^{-}(x) - F_\mu^-(x).$
		\end{enumerate}
	\end{lem}
	
	\begin{proof}
		The proofs of Points \ref{point:controle_bord_droit_pos} and \ref{point:controle_bord_droit_neg}  are analogous to those of Points \ref{point:controle_bord_gauche_neg} and \ref{point:controle_bord_gauche_pos}, and are therefore omitted.
		\begin{enumerate}
			\item Equation \eqref{eq:inclusion_border}, implies that $B_l^- \cup \left(B_r \setminus B_l \right) \subset \{F_\mu^+ \leq F_\nu^+\}$. By Point \ref{point:bar_leq_pos} of Proposition  \ref{pro:barrier_points}, it follows that $\pi(]-\infty,x]\times ]x,+\infty[) = 0$. In particular, $\pi(\{x\}\times ]x,+\infty[) = 0$.
			\item

			Since $B_l^+ \subset E^=$, Lemma \ref{remq:preliminary_inclusion} implies $B_l^+ \subset \BB$, so that $\pi(]-\infty,x[\times]x,+\infty[) = 0.$ Hence,
			
			\begin{align*}
				F_\mu^+(x) &= \pi(]-\infty,x] \times \R)\\  &= \pi(]-\infty,x]^2) + \pi(]-\infty,x] \times ]x,+\infty[)\\ &= \pi(]-\infty,x]^2) + \pi(]-\infty,x[ \times ]x+ \infty[) + \pi(\{x\} \times ]x,+\infty[) \\
				&= \pi(]-\infty,x]^2) + \pi(\{x\} \times ]x,+\infty[).
			\end{align*}
			
			Moreover, from Equation \eqref{eq:inclusion_border}, we have $B_l^+ \subset \{F_\mu^+ \geq F_\nu^+\}$. By  Point \ref{point:bar_geq_pos} of Proposition \ref{pro:barrier_points}, it follows that  $\pi(]x,+\infty[\times]-\infty,x]) = 0$. Hence,
			\begin{align*}
				F_\nu^+(x) &= \pi(\R \times ]-\infty,x]) \\
				&= \pi(]-\infty,x]^2) + \pi(]x,+\infty[ \times ]-\infty,x]) \\
				&= \pi(]-\infty,x]^2).
			\end{align*}
			Therefore, $F_\mu^+(x) - F_\nu^+(x) = \pi(\{x\} \times ]x,+\infty[)$.\qedhere
		\end{enumerate}
	\end{proof}

	We now establish that the marginals of the components of a cyclically monotone transport plan (see Definition \ref{defi:composantes}) are the components of the marginals (see Definition \ref{def:component_marginals}). We then prove that a cyclically monotone transport plan equals the sum of its components.
	
	\begin{pro}\label{pro:transport_fitting_marg}
		Consider $(\mu,\nu) \in \MM^2_+$ and $\pi \in \CCC(\mu,\nu)$. 
		\begin{enumerate}
			\item \begin{enumerate}
				\item\label{point:transport_fitting_marg_plus} For all $k \in \KK^+$, $\pi_k^+ \in \Marg({\mu}_k^+, {\nu}_k^+).$
				\item\label{point:transport_fitting_marg_minus} For all $k \in \KK^-$, $\pi_k^- \in \Marg({\mu}_k^-, {\nu}_k^-).$
				\item\label{point:transport_fitting_marg_eq_first} Equalities $\pi_1^= = (\id,\id)_\# {\mu}_1^= = (\id,\id)_\# {\nu}_1^=$ are satisfied.
				\item\label{point:transport_fitting_marg_eq_second} Equalities $\pi_2^= = (\id,\id)_\#  {\mu}_2^= = (\id,\id)_\#  {\nu}_2^= $ are satisfied.
			\end{enumerate}
			\item\label{point:surj_int} Equality $\pi = \sum_{k \in \KK^+} \pi_k^+ + \sum_{k \in \KK^-} \pi_k^- + \pi^=$ is satisfied.
		\end{enumerate}
	\end{pro}

	\begin{proof}
		\begin{enumerate}
			\item
			\begin{enumerate}
				 \item For all $t \in \R$, 
				 \begin{equation*}
				 	{\mu}_k^+(]-\infty,t])=
				 	\begin{cases}
				 		0 &\text{if } t< a_k^+ \\
				 		F_\mu^+(t)  - F_\nu^+(a_k^+) &\text{if } t \in [a_k^+,b_k^+[\\
				 		F_\mu^-(b_k^+) - F_\nu^+(a_k^+) &\text{if } t \geq b_k^+
				 	\end{cases}
				 \end{equation*}
				 and
				 \begin{equation*}
				 	\pi_k^+(]-\infty,t] \times \R) =
				 	\begin{cases}
				 		0 &\text{if } t < a_k^+\\
				 		\pi([a_k^+,t]\times ]a_k^+,b_k^+]) &\text{if } t \in [a_k^+,b_k^+[ \\
				 		\pi([a_k^+,b_k^+[ \times ]a_k^+,b_k^+]) &\text{if } t \geq b_k^+
				 	\end{cases}.
				 \end{equation*}
				 Assume for a while that the second line of both systems coincide. Then, 
				 \begin{equation*}
				 	F_\mu^-(b_k^+) - F_\nu^+(a_k^+) = \lim_{t \to b_k^+, t < b_k^+} F_\mu^+(t) - F_\nu^+(a_k^+) =\lim_{t \to b_k^+, t < b_k^+} \pi([a_k^+,t]\times ]a_k^+,b_k^+]) =\pi([a_k^+,b_k^+[\times ]a_k^+,b_k^+])
				 \end{equation*}
			 	and both system coincide. Thus, to prove that ${p_1}_\# \pi_k^+ = \mu_k^+$, we are left to show
				 that, for all $t \in [a_k^+,b_k^+[$, $\pi([a_k^+,t]\times ]a_k^+,b_k^+]) = F_\mu^+(t) - F_\nu^+(a_k^+)$ holds. Now, for all $t \in [a_k^+, b_k^+[$, 
				 \begin{align*}
				 	\pi([a_k^+,t] \times ]a_k^+,b_k^+]) &= \pi([a_k^+,t] \times ]a_k^+,+\infty[) \\
				 	&= \mu([a_k^+,t]) - \pi([a_k^+,t] \times ]-\infty,a_k^+]) \\
				 	&= F_\mu^+(t) - F_\mu^-(a_k^+) - \pi(\{a_k^+\} \times ]-\infty,a_k^+]) \\
				 	&= F_\mu^+(t) - F_\mu^-(a_k^+) - (\mu(a_k^+)- \pi(\{a_k^+\}\times ]a_k^+,+\infty[)) \\
				 	&= F_\mu^+(t) - F_\nu^+(a_k^+),
				 \end{align*}
				 where the first equality comes from $b_k^+ \in E^= \subset \BB$, the third equality comes from  Point \ref{point:bar_geq_pos} of Proposition \ref{pro:barrier_points} (with $x = a_k^+$) and Equation \eqref{eq:inclusion_border}, while the last equality comes from Point \ref{point:controle_bord_gauche_pos} of Lemma \ref{lem:mass_transfert}. Therefore, ${p_1}_\# \pi_k^+ = {\mu}_k^+$, and a similar argument shows that ${p_2}_\# \pi_k^+ = {\nu}_k^+$.
				 \item Same proof as for the previous point.
				 \item We first show the inclusion $\spt(\pi) \cap (E^=  \times \R) \cap (B^c \times \R) \subset \D.$ Assume for contradiction that there exists $(x,y) \in  \spt(\pi) \cap (E^=  \times \R) \cap (B^c \times \R)$ such that $x \neq y$. Since $x \in B^c$, $]\min(x,y),\max(x,y)[ \subset E^-$ or $]\min(x,y),\max(x,y)[ \subset E^+$ is prohibited, which implies there exists $z \in ]\min(x,y),\max(x,y)[ \cap E^=$. Thus, as $x<y$ implies $(x,y) \in ]-\infty,z[\times]z,+\infty[ \subset \Cd(z)$ and $y<x$ implies $(x,y) \in ]z,+\infty[\times ]-\infty,z[ \subset \Cd(z)$, $\Cd(z)$ is an open set containing $(x,y)$. Since $z \in E^= \subset B$ by \ref{remq:preliminary_inclusion} , $\pi(\Cd(z))=0$. This contradicts the fact that $(x,y) \in \spt(\pi)$, thereby proving the inclusion. By construction, the first marginal of $\pi_1^= = \pi_{\res [(B^c \cap E^=) \times \R} $ is ${\mu}_1^= = \mu_{\res [(B^c \cap E^=) \times \R}$ and $0 \leq \pi_1^=(\D^c) = \pi([(B^c \cap E^=) \times \R] \cap \D^c) \leq \pi(\D \cap \D^c) = 0,$ which shows the first identity. Observe that $\pi_1^= = \pi_{(E^=\cap B^c) \times (E^=\cap B^c)}$.  The proof of $\pi_{\res \R \times (E^= \cap B^c)} = (\id,\id)_\# \nu_1^=$ is similar. Therefore, $(\id,\id)_\# \mu_1^= =\pi_{\res (E^=\cap B^c) \times \R} = \pi_{\res (E^= \cap B^c)^2} = \pi_{\res \R \times (E^=\cap B^c)} =(\id,\id)_\# \nu_1^=$.
				 \item By definition, ${\pi}_2^= = \sum_{x \in B} \pi(x,x) \d_{(x,x)},$ so we have to show that, for all $x \in B$, $\pi(x,x) = {\mu}_2^=(x) = \mu(x) - {\mu}^-(x) - {\mu}^+(x)$. We compute $\pi(x,x)$ using the partition of Point \ref{point:partition_B} of Lemma \ref{lem:eq_fixe_deux}, together with Lemma \ref{lem:mass_transfert}. More precisely, the two first point of Lemma \ref{lem:mass_transfert} give the value of first column of Table \ref{tab_pi_x_x}, while the two other points give the value of the second column of Table \ref{tab_pi_x_x}. The final column follows from the identity $\pi(x,x) = \mu(x) - \pi(\{x\} \times ]x,+\infty[) - \pi(\{x\} \times ]-\infty,x[)$. Since the last columns of Table \ref{tab:mu_sub} and Table \ref{tab_pi_x_x} are the same, this finishes the proof.
			\end{enumerate}
			\item Define $\ti{\pi} = \sum_{k \in \KK^+} {\pi}_k^+ + \sum_{k \in \KK^-} {\pi}_k^- + {\pi}^=$. Since  $\ens{A_k^+}{k \in \KK^+} \cup \ens{A_k^-}{k \in \KK^-} \cup \{ A^=\}$ is a class of disjoint sets, we have $\ti{\pi} = \pi_{\res \left(\cup_{k \in \KK^+} A_k^+\right) \cup \left(\cup_{k \in \KK^-} A_k^-\right) \cup A^=} \leq \pi$. Furthermore, $\ti{\pi}(\R^2) =\sum_{k \in \KK^+} {\mu}_k^+(\R) + \sum_{k \in \KK^-} {\mu}_k^-(\R) + {\mu}^=(\R) = \mu(\R) = \pi(\R^2).$ Thus, $\pi = \ti{\pi}$.\qedhere
		\end{enumerate}
	\end{proof}

	\begin{table}[h!]
		\centering
		\begin{tabular}{|c||c|c|c|}
			\hline  $x \in \cdot$ & $\pi(\{x\} \times ]x,+\infty[)$ & $\pi(\{x\} \times ]-\infty,x[)$ & $\pi(\{x\} \times \{x\})$ \\
			\hline $B_l^+ \cap B_r^+$ & $F_\mu^+(x) - F_\nu^+(x)$ & $0$ & $F_\nu^+(x) - F_\mu^-(x)$ \\
			\hline $ B_l^+ \cap B_r^-$ & $F_\mu^+(x) - F_\nu^+(x)$ & $F_\nu^-(x) - F_\mu^-(x)$ &  $\nu(x) $\\
			\hline $B_l^- \cap B_r^+$ & $0$ & $0$ & $\mu(x)$ \\
			\hline $B_l^- \cap B_r^- $ & $0$ & $F_\nu^-(x) - F_\mu^-(x)$ & $F_\mu^+(x) - F_\nu^-(x)$  \\
			\hline $B_l^+ \setminus B_r$ & $F_\mu^+(x) - F_\nu^+(x)$ & $0$ & $F_\nu^+(x) - F_\mu^-(x)$ \\
			\hline $B_l^- \setminus B_r$ & $0$ & $0$ & $\mu(x)$\\
			\hline $B_r^+ \setminus B_l$ & $0$ & $0$ & $\mu(x)$ \\
			\hline $B_r^- \setminus B_l$ & $0$ & $F_\nu^-(x) - F_\mu^-(x)$ & $ F_\mu^+(x) - F_\nu^-(x)$ \\
			\hline
		\end{tabular}
		\caption{Values of $\pi(x,x)$}
		\label{tab_pi_x_x}
	\end{table}
	
	\begin{defi}\label{defi:order}
		We denote by $\leq_{\R^2}$ the usual partial order on $\R^2$: for all $(x_1,y_1), (x_2,y_2) \in \R^2$, $(x_1,y_1) \leq_{\R^2} (x_2,y_2)$ if $x_1 \leq x_2$ and $y_1 \leq y_2.$ This order naturally extends to subsets of $\R^2$ as follows: for all subsets $A,B$ of $\R^2$, we write $A \leq_{\R^2} B$ if for all $(a,b) \in A \times B$, $a \leq_{\R^2} b.$  Let $\OOO\RRR$ denote the set formed by the classes $\CC$ of subsets of $\R^2$ such that, $$\forall A,B \in \CC: A \neq B \implies A \leq_{\R^2} B \text{ or } B \leq_{\R^2} A.$$
	\end{defi}
	The reader may think of $\OOO\RRR$ as the set of totally ordered classes of subset of $\R^2$. However, we do not require that classes $\CC \in \OOO\RRR$ satisfy $A \leq_{\R^2} A$ for every $A \in \CC.$ 
	\begin{ex}
		The class $\CC = \ens{[n,n+1]^2}{n \in \N}$ is an element of $\OOO\RRR$.
	\end{ex}
	\begin{lem}\label{lem:cycl_union}
		\begin{enumerate}
			\item Let $A$ and $B$ denote two cyclically monotone subsets of $\R^2$. If $A \leq_{\R^2} B$, then $A \cup B$ is a cyclically monotone set.
			\item If $\Gamma$ is a cyclically monotone set, then $\D \cup \Gamma$ is also a cyclically monotone set
		\end{enumerate}	
	\end{lem}
	
	\begin{proof}
		\begin{enumerate}
			\item\label{point:cup_cm_ordered} Consider $n \geq 1$ and $((x_k,y_k))_{k \in \ent{1}{n}} \in \left(A \cup B \right)^n.$ Define $I_A = \ens{i \in \ent{1}{n}}{(x_i,y_i) \in A}$ and $I_B = \ens{i \in \ent{1}{n}}{(x_i,y_i) \in B \setminus A}$. Observe that $I_A\uplus I_B = \ent{1}{n}$ and define $k = \#I_A$. Let \( (x_{(1)}, \dots, x_{(k)}) \), \( (x_{(k+1)}, \dots, x_{(n)}) \), \( (y_{(1)}, \dots, y_{(k)}) \), and \( (y_{(k+1)}, \dots, y_{(n)}) \) denote the families \( (x_i)_{i \in I_A} \), \( (x_i)_{i \in I_B} \), \( (y_i)_{i \in I_A} \), and \( (y_i)_{i \in I_B} \), respectively, sorted in non-decreasing order.
			Since $A \leq_{\R^2} B$, there exists $(x,y) \in \R^2$ such that $A \subset ]-\infty,x] \times ]-\infty,y]$ and $B \subset [x,+\infty[ \times [y,+\infty[.$ Thus $x_{(1)} \leq \dots \leq x_{(k)} \leq x \leq x_{(k+1)} \leq \dots x_{(n)}$ and $y_{(1)} \leq \dots \leq y_{(k)} \leq y \leq y_{(k+1)} \leq \dots y_{(n)}$. Hence, \( (x_{(1)}, \dots, x_{(n)}) \) and \( (y_{(1)}, \dots, y_{(n)}) \) correspond to the sequences \( (x_1, \dots, x_n) \) and \( (y_1, \dots, y_n) \), respectively, sorted in non-decreasing order. Since the monotone transport plan $\sum_{i=1}^n \d_{(x_{(i)},y_{(i)})}$ from $\sum_{i=1}^n \d_{x_i}$ to $\sum_{i=1}^n \d_{y_i}$ is optimal, we have $\sum_{i=1}^n |y_{(i)}-x_{(i)}| \leq \sum_{i=1}^n |y_{i+1}-x_i|.$ Moreover, as $((x_i,y_i))_{i \in I_A} \in A^k$, $((x_i,y_i))_{i \in I_B} \in B^{n-k}$ and $(A,B)$ is a pair of cyclically monotone sets, we deduce that: $$\sum_{i=1}^n |y_i-x_i| = \sum_{i \in I_A} |y_i-x_i| + \sum_{i \in I_B} |y_i-x_i|\leq \sum_{i=1}^k |y_{(i)}-x_{(i)}| + \sum_{i=k+1}^n |y_{(i)}-x_{(i)}|.$$ Thus, $$\sum_{i=1}^n |y_i-x_i| \leq \sum_{i=1}^n |y_{(i)}-x_{(i)}| \leq  \sum_{i=1}^n |y_{i+1}-x_i|.$$ Therefore, $A \cup B$ is cyclically monotone.
			\item\label{point:cycl_D_cup_g} Consider $((x_i,y_i))_{i \in \ent{1}{n}} \in (\D \cup \Gamma)^n$ and define $I = \ens{i \in \ent{1}{n}}{(x_i,y_i) \in \Gamma}$. For all $i \in I$, set $\s(i) = \inf\ens{j > i}{(x_j,y_j) \in \Gamma}$.\footnote{Where we use the convention $(x_{n+j},y_{n+j}) = (x_j,y_j)$.} For every $i \in I$ and $j \in \ent{i+1}{\s(i)-1}$, $(x_j,y_j)$ belongs to $\D$. As $\Gamma$ is cyclically monotone, it follows that:
			\begin{equation*}
				\sum_{i = 1}^n |y_i-x_i| = \sum_{i \in I} |y_i-x_i| \\
				\leq \sum_{i \in I} |y_{\s(i)}-x_i| \\
				= \sum_{i\in I} \left|\sum_{j =i}^{\s(i)-1} y_{j+1}-x_j \right|\\
				\leq \sum_{i \in I} \sum_{j=i}^{\s(i)-1} |y_{j+1}-x_j| \\
				= \sum_{i=1}^n |y_{i+1}-x_i|.
			\end{equation*}
			Thus, $\Gamma \cup \D$ is cyclically monotone.\qedhere
		\end{enumerate}
	\end{proof}
	
	\begin{remq}\label{remq:cyclicall_monotonicity_triangle}
		Let $(C_i)_{i \in \II}$ be a family of ordered cyclically monotone sets and define $\Gamma  = \D \cup \left(\bigcup_{i \in \II} C_i\right)$. Then $\Gamma$ is a cyclically monotone set. Indeed, by induction on Point \ref{point:cup_cm_ordered} of Lemma \ref{lem:cycl_union}, for any finite subset $J$ of $I$, $\cup_{j \in J} C_j$ is cyclically monotone. As every finite sequence of element of $\cup_{i \in \II} C_i$ belongs to a such set, $\cup_{i \in \II} C_i$ is cyclically monotone. According to Point \ref{point:cycl_D_cup_g} of Lemma \ref{lem:cycl_union}, $\Gamma$ is cyclically monotone.
	\end{remq}

	To state the properties of our marginal decomposition, we introduce a reinforced version of the stochastic order, due to Kellerer \cite[Definition $1.17$]{kellerer_order_1986}. 
	\begin{defi}[Large\footnote{We added \emph{large} to make a distinction between this order and a similar order --- related to $\G$ instead of $\F$ --- that will be introduced later in the article (Definition \ref{def:Kellerer_order_strict}). For the sake of fluidity, we omit the word large in the following.} reinforced stochastic order of Kellerer]\label{def:Kellerer_order_large}
		Consider $(\g_1,\g_2) \in \MM_+(\R)^2$ and define 
		\begin{equation*}
				T_+(\g_1,\g_2) = \ens{t \in \R}{F_{\g_1}^+(t) >0 \text{ and } F_{\g_2}^+(t) < \g_2(\R)}
		\end{equation*}
		and
		\begin{equation*}
			T_-(\g_1,\g_2) = \ens{t \in \R}{F_{\g_1}^-(t) > 0 \text{ and } F_{\g_2}^-(t) < \g_2(\R)}.
		\end{equation*}
		We say that $\g_1$ is smaller than $\g_2$ in the (large) reinforced stochastic order if $\g_1 \leq_{\st} \g_2$, $T_+(\g_1,\g_2) \subset \{F_{\g_1}^+ > F_{\g_2}^+\}$, and $T_-(\g_1,\g_2) \subset \{F_{\g_1}^- > F_{\g_2}^-\}.$ In this case, we write $\g_1 \leq_{\F} \g_2$.\footnote{This notation, motivated by the associated Strassen-type theorem, is adopted from Kellerer \cite{kellerer_order_1986}.}
	\end{defi}

	\begin{remq}\label{remq:leq_F}
		\begin{enumerate}
			\item Intuitively, if ${\g_1} \leq_{\st} {\g_2}$, then ${\g_1} \leq_{\F} {\g_2}$ means that the relations $F_{\g_1}^+ \geq F_{\g_2}^+$ and $F_{\g_1}^- \geq F_{\g_2}^-$ given by the stochastic order are strict ``wherever possible''. Specifically, since $T_+({\g_1},{\g_2})^c \subset \{F_{\g_1}^+ = F_{\g_2}^+ = 0\} \cup \{F_{\g_1}^+ = F_{\g_2}^+ = \g_1(\R) \} $ and $T_-({\g_1},{\g_2})^c \subset \{F_{\g_1}^- = F_{\g_2}^- = 0\} \cup \{F_{\g_1}^- = F_{\g_2}^- = {\g_2}(\R)\}$, it follows that $\{F_{\g_1}^+ > F_{\g_2}^+\} \subset T_+({\g_1},{\g_2})$ and $ \{F_{\g_1}^- > F_{\g_2}^-\} \subset T_-({\g_1},{\g_2})$). Thus, the relation $\g_1 \leq_{\F} \g_2$ precisely means that the inclusions $\{F_{\g_1}^+ > F_{\g_2}^+\} \subset T_+(\g_1,\g_2)$ and $\{F_{\g_1}^- > F_{\g_2}^-\} \subset T_-(\g_1,\g_2)$ are actually equalities.
			\item\label{point:infsupsupp} Observe that, for all $\g \in \MM^+(\R)$, $s_{\g} = \inf(\{F_{\g}^- > 0\})$ and $S_\g = \sup(\{F_\g^+ < \g(\R)\}).$
			\item\label{point:leq_F} To state this point, recall notation $\At(\g)$ (see Sub-subsection \ref{subsection:not}, Point \ref{point:atoms_not}). By the previous point of the remark, if $\g_1 \leq_{\st} \g_2$,
			\begin{equation}
				T_+(\g_1,\g_2) = 
				\begin{cases}
					[s_{\g_1},S_{\g_2}[ \text{ if } s_{\g_1} \in \At({\g_1}) \\
					]s_{\g_1},S_{\g_2}[ \text{ otherwise } 
				\end{cases} \text{ and \hspace{0.2cm}}
				T_-(\g_1,\g_2) =
				\begin{cases}
					]s_{\g_1},S_{\g_2}] \text{ if } S_{\g_2} \in \At({\g_2}) \\
					]s_{\g_1},S_{\g_2}[ \text{ otherwise } 
				\end{cases}.
			\end{equation} 
			\item The reader may verify that, if $F_{\g_1}^+ \geq F_{\g_2}^+$, $]s_{\g_1},S_{\g_2}[ \subset E^+(\g_1,\g_2)$, $s_{\g_1} \notin \At({\g_2})$, and $S_{\g_2} \notin \At({\g_1})$, then $\g_1 \leq_\F \g_2.$
		\end{enumerate}
	\end{remq}

	We now examine the properties of the marginal decomposition introduced in Definition \ref{def:component_marginals}. Specifically, we establish that each pair is ordered in the  reinforced stochastic order of Kellerer and that the fixed parts coincide.

	\begin{pro}\label{pro:prop_deco}
		Consider $(\mu,\nu) \in \MM^2_+.$
		\begin{enumerate}
			\item The equalities $\mu = \sum_{k \in \KK^+ } {\mu}_k^+ + \sum_{k \in \KK^-} {\mu}_k^- + {\mu}^=$ and $\nu =\sum_{k \in \KK^+ } {\nu}_k^+ + \sum_{k \in \KK^-} {\nu}_k^- + {\nu}^=$ hold.
			\item\label{point:ordre_compo} For all $k \in \KK^+$ (resp. $\KK^-$), ${\mu}_k^+ \leq_{\F} {\nu}_k^+$ (resp. ${\nu}_k^- \leq_{\F} {\mu}_k^-$).
			\item We have ${\mu}_1^= = {\nu}_1^=$, ${\mu}_2^= = {\nu}_2^=$, and ${\mu}^= = {\nu}^=.$ 
		\end{enumerate}
	\end{pro}
	
	\begin{proof}
		\begin{enumerate}
			\item\label{point:pro_marg_sum_compo} This was established in Point \ref{point:marg_sum_compo} of Remark \ref{remq:on_marginals}.
			\item For all $k \in \KK^+$ and $t \in \R$,  
			\begin{equation}\label{eq:fdr_rest_pos}
				\accol{F_{{\mu}_k^+}^+(t) = \1_{a_k^+ \leq t < b_k^+}(F_\mu^+(t)-F_\nu^+(a_k^+)) + \1_{t \geq b_k^+} (F_\mu^-(b_k^+)-F_\nu^+(a_k^+)) \\ F_{{\nu}_k^+}^+(t) = \1_{a_k^+< t < b_k^+} (F_\nu^+(t)- F_\nu^+(a_k^+)) + \1_{t \geq b_k^+}((F_\mu^-(b_k^+)-F_\nu^+(a_k^+))}
			\end{equation}
			and 
			\begin{equation}\label{eq:fdr_rest_neg}
				\accol{F_{{\mu}_k^+}^-(t) = \1_{a_k^+ < t \leq b_k^+}(F_\mu^-(t)-F_\nu^+(a_k^+)) + \1_{t>b_k^+}(F_\mu^-(b_k^+)-F_\nu^+(a_k^+)) \\ F_{{\nu}_k^+}^-(t) = \1_{a_k^+<t \leq b_k^+}(F_\nu^-(t)- F_\nu^+(a_k^+))+ \1_{t> b_k^+}(F_\mu^-(b_k^+)-F_\nu^+(a_k^+))},
			\end{equation} which implies
			\begin{equation}\label{eq:diff_rep_compo}
				\begin{cases}
					F_{{\mu}^+_k}^+(t) - F_{{\nu}^+_k}^+(t) = \1_{t=a_k^+}(F_\mu^+(a_k^+) - F_\nu^+(a_k^+)) + \1_{a_k^+< t < b_k^+} (F_\mu^+(t)-F_\nu^+(t))\\ F_{{\mu}_k^+}^-(t) - F_{{\nu}_k^+}^-(t) = \1_{a_k^+< t \leq b_k^+}(F_{\mu}^-(t)-F_{\nu}^-(t))
				\end{cases}.
			\end{equation}
			We have $]a_k^+,b_k^+[\subset E^+$ and Equation \eqref{eq:inclusion_border} implies $a_k^+ \in \{F_\mu^+ \geq F_\nu^+\}$. From Equation \eqref{eq:diff_rep_compo}, it follows that $F_{{\mu}_k^+}^+ \geq F_{{\nu}_k^+}^+.$ Morever, for all $t\in ]a_k^+,b_k^+[$, $F_\mu^-(t) > F_\nu^-(t) \geq F_\nu^+(a_k^+)$ and $F_\nu^+(t) < F_\mu^+(t) \leq F_\mu^-(b_k^+)$. By Equation \eqref{eq:fdr_rest_neg} and \eqref{eq:fdr_rest_pos}, we respectively obtain that $\{F_{\mu_k^+}^- > 0\} = ]a_k^+,+\infty[$ and $\{F_{\nu_k^+}^+ < \nu_k^+(\R)\}.$ By Point \ref{point:infsupsupp} of Remark \ref{remq:leq_F}, we get $s_{{\mu}_k^+} = a_k^+$ and $S_{{\nu}_k^+} = b_k^+$ and Equation \eqref{eq:diff_rep_compo} implies $]s_{{\mu}_k^+},S_{{\nu}_k^+}[ \subset E^+({\mu}_k^+,{\nu}_k^+)$. Since ${\mu}_k^+(S_{{\nu}_k^+}) = {\nu}_k^+(s_{{\mu}_k^+}) = 0,$ by Point \ref{point:leq_F} of Remark \ref{remq:leq_F}, we get ${\mu}_k^+ \leq_{\F} {\nu}_k^+$. The proof of the inequality ${\nu}_k^- \leq_{\F} {\mu}_k^-$ is similar.
			\item The Points \ref{point:transport_fitting_marg_eq_first} and \ref{point:transport_fitting_marg_eq_second} of Proposition \ref{pro:transport_fitting_marg} establish that ${\mu}_1^= = {\nu}_1^=$ and  ${\mu}_2^= = {\nu}_2^=$. This shows that ${\mu}^= = {\mu}_1^= + {\mu}_2^= ={\nu}_1^= + {\nu}_2^= = {\nu}^=$.
		\end{enumerate}
	\end{proof}

	We now state and prove the decomposition result for cyclically monotone transport plans associated with the marginal components introduced in Definition \ref{def:component_marginals}.

	\begin{them}[Decomposition of $\CCC(\mu,\nu)$]\label{pro:decomposition}
		Consider $(\mu,\nu) \in \MM^2_+.$ We use the notation for marginal components introduced in Definition \ref{def:component_marginals} and the notation $((A_k^+)_{k \in \KK^+},(A_k^-)_{k \in \KK^-},A^=)$ introduced in Definition \ref{defi:composantes}. Define also ${\mathfrak{P}} = \prod_{k \in \KK^+} \CCC({\mu}_k^+,{\nu}_k^+) \times \prod_{k \in \KK^-} \CCC({\mu}_k^-, {\nu}_k^-) \times \CCC({\mu}^=, {\nu}^=).$
		\begin{enumerate}
			\item\label{point:ordre_dec_part} For all $k \in \KK^+$ (resp. $\KK^-$), ${\mu}_k^+ \leq_{\F} {\nu}_k^+$ (resp. ${\nu}_k^- \leq_{\F} {\mu}_k^-$). Moreover, ${\mu}^= = {\nu}^=.$
			\item\label{point:somme_dec_part} The map $\varphi : {\mathfrak{P}} \to \CCC(\mu,\nu)$ defined by 
			$$\varphi(({\pi}_k^+)_{k \in \KK^+}, ({\pi}_k^-)_{k \in \KK^-},{\pi}^=) = \sum_{k \in \KK^+ } {\pi}_k^+ + \sum_{k \in \KK^- } {\pi}_k^- + {\pi}^=$$ is a bijection. Furthermore, its inverse $\varphi^{-1} : \CCC(\mu ,\nu) \to {\mathfrak{P}} $ is given by 
			\begin{equation}\label{eq:formul_inverse_int}
				\varphi^{-1}(\pi) = \left( \left(\pi_{\res A_k^+} \right)_{k \in \KK^+},  \left(\pi_{\res A_k^-} \right)_{k \in \KK^-}, \pi_{\res A^=} \right).
			\end{equation}
		\end{enumerate}
	\end{them}
	\begin{proof}
		The first point follows from Proposition \ref{pro:prop_deco}. For the second point, we first establish that $\varphi$ is indeed valued in $\CCC(\mu,\nu)$. Consider a family ${\FF} := (({\pi}_k^+)_{k \in \KK^+},({\pi}_k^-)_{k \in \KK^-}, {\pi}^=) \in {\mathfrak{P}}$. By Point \ref{point:pro_marg_sum_compo} of Proposition \ref{pro:prop_deco}, $\pi^* := \sum_{k \in \KK^+} {\pi}_k^+ + \sum_{k \in \KK^-} {\pi}_k^- + {\pi}^=$ is a transport plan from $\mu$ to $\nu$. By Proposition \ref{pro:ordre_stochastique}, $\pi$ is concentrated  on $\Gamma^* = \bigcup_{k \in \KK^+}(\F \cap A_k^+) \cup \bigcup_{k \in \KK^-} (\ti{\F} \cap A_k^-) \cup \D$. As $\F$ and $\ti{\F}$ are cyclically monotone sets, by Remark \ref{remq:cyclicall_monotonicity_triangle}, $\Gamma^*$ is a cyclically monotone set. Thus, $\pi^* \in \CCC(\mu,\nu)$ and $\varphi$ is indeed valued in $\CCC(\mu,\nu).$ We now establish that $\varphi$ is injective. Consider $\pi \in \CCC(\mu,\nu)$ and  $ {\FF} := (({\pi}_k^+)_{k \in \KK^+},({\pi}_k^-)_{k \in \KK^-}, {\pi}^=) \in {\mathfrak{P}}$ such that $\varphi(\FF) = \pi$. Since ${\pi}^= \in \CCC({\mu}^=,{\nu}^=) = \{(\id,\id)_\# {\mu}^=\}$, ${\pi}^= = (\id,\id)_\# {\mu}^=$. In particular ${\pi}^= $ is concentrated on $A^=$. For all $k \in \KK^+$(resp. $k \in \KK^-$), ${\pi}_k^+$ (resp. ${\pi}_k^-$) is concentrated on $A_k^+$ (resp. $A_k^-$). As $\ens{A_k^+}{k \in \KK^+} \cup \ens{A_k^-}{k \in \KK^-}   \cup \{A^=\}$ is  a class of disjoints set, for all $k \in \KK^+$, $$\pi_{\res A_k^+} = \left(\sum_{k \in \KK^+ } {\pi}_j^+ + \sum_{j \in \KK^- } {\pi}_j^- + {\pi}^=\right)_{\res A_k^+} = {\pi}_k^+.$$ Similarly, for all $k \in \KK^-$, ${\pi}_k^- = \pi_{\res A_k^-}$, and $\pi^= = \pi_{\res A^=}$. Therefore, ${\FF}$ is equal to the second term of Equation \eqref{eq:formul_inverse_int} and $\varphi$ is injective. We establish now that $\varphi$ is surjective. Consider $\pi \in \CCC(\mu,\nu)$ and let $\FF$ denote the right term of Equation \eqref{eq:formul_inverse_int}. From Points (1--4) of Proposition \ref{pro:transport_fitting_marg} and Remark \ref{remq:cycl_res}, if follows that $\FF \in \mathfrak{P}.$ By Point $\ref{point:surj_int}$ of Proposition \ref{pro:transport_fitting_marg}, it follows that $\varphi(\FF) = \pi.$ Therefore, $\varphi$ is surjective and Formula \eqref{eq:formul_inverse_int} is satisfied.
	\end{proof}

	Using Notation \ref{nott:somme_directe}, Point \ref{point:somme_dec_part} gives the direct sum decomposition 
	\begin{equation}\label{eq:somme_directe_int}
		\begin{aligned}
			\CCC(\mu,\nu) &= \left(\bigoplus_{k \in \KK^+} \CCC({\mu}_k^+,{\nu}_k^+) \right) \oplus \left( \bigoplus_{k \in \KK^-} \CCC({\mu}_k^-,{\nu}_k^-) \right) \oplus \CCC({\mu^=},{\nu}^=) \\
			&= \left(\bigoplus_{k \in \KK^+} \Marg_{\F}({\mu}_k^+,{\nu}_k^+) \right) \oplus \left( \bigoplus_{k \in \KK^-} \Marg_{\ti{\F}}({\mu}_k^-,{\nu}_k^-) \right) \oplus \{(\id,\id)_\# \eta\}.
		\end{aligned}
	\end{equation}
	
	\begin{remq}\label{remq:finite_cost}
		As previously noted, if $W_1(\mu,\nu) <+ \infty$, then $\CCC(\mu,\nu) = \OO(\mu,\nu).$ However, in case $W_1(\mu,\nu) = + \infty$, one can not replace $\CCC(\mu,\nu)$ with $\OO(\mu,\nu)$ in this subsection. For instance, define 
		\begin{equation*}
			\begin{cases}
				\mu = \sum_{ k \geq 1} \frac{1}{2^k} \d_{-k} + \d_0 + \sum_{k \geq 1} \frac{1}{2^k} \d_k \\
				\nu = \sum_{ k \geq 1} \frac{1}{2^k} \d_{-k-2^k} + \d_0 + \sum_{k \geq 1} \frac{1}{2^k} \d_{k+2^k}
			\end{cases}.
		\end{equation*}
		As  $\pi^* := \sum_{k \geq 1} \frac{1}{2^k} \d_{(-k,-k-2^k)} + \d_{(0,0)} + \sum_{k \geq 1} \frac{1}{2^k} \d_{(k,k+2^k)} \in \Marg(\mu,\nu)$ is concentrated on $\left(\ti{\F} \cap ]-\infty,0]^2\right) \cup \left(\F \cap [0,+\infty[^2 \right)$, by Lemma \ref{lem:cycl_union}, $\pi^*$  belongs to $\CCC(\mu,\nu)$. By Theorem \ref{them:cyclicality_optimality}, $W_1(\mu,\nu) = J(\pi^*) = 2\sum_{k \geq 1} \frac{1}{2^k}2^k + 1 = + \infty.$ In this example, the decomposition yields $E^+(\mu,\nu) = ]0,+ \infty[$, $E^-(\mu,\nu) = ]-\infty, 0[$,
		and $E^=(\mu,\nu) = \{0\}$. Hence $\mu_1^+ = \sum_{k \geq 1} \frac{1}{2^k} \d_k$, $\nu_1^+ = \sum_{k \geq 1} \frac{1}{2^k} \d_{k+2^k}$, $\mu_1^- = \sum_{k \geq 1} \frac{1}{2^k} \d_{-k}$, $\nu_1^- = \sum_{k \geq 1} \frac{1}{2^k} \d_{-k - 2^k}$, and $\mu^= = \nu^= = \d_0$.
		Clearly $\mu \otimes \nu \notin \OO(\mu_1^+, \nu_1^+) \oplus \OO(\mu_1^-, \nu_1^-) \oplus \OO(\mu^=, \nu^=)$, and Theorem \ref{pro:decomposition} does not hold for $\OO(\mu,\nu)$ in place of $\CCC(\mu,\nu)$. 
	\end{remq}
	
	\begin{remq}[The decomposition of Kellerer]
		In \cite[Proposition $1.20$]{kellerer_order_1986}, Kellerer stated that if  $\mu \leq_{\st} \nu$ and $(\mu,\nu)$ satisfy $\mu(\{F_\mu^+=F_\nu^+\}) = \nu(\{F_\mu^-=F_\nu^-\}) = 0$, then $$\varphi : (\pi_k^+)_{k \in \KK^+} \in \prod_{k \in \KK^+} \Marg_{\F}\left(\mu_{\res [a_k^+,b_k^+[}, \nu_{\res ]a_k^+,b_k^+]}\right) \mapsto \sum_{k \in \KK^+} \pi_k^+ \in \Marg_{\F}(\mu,\nu).$$ is a bijection. Observe that, by Proposition \ref{pro:ordre_stochastique} and Remark \ref{remq:finite_cost}, we know that $\Marg_{\F}(\mu,\nu) = \mathfrak{C}(\mu,\nu).$ Hence, this result is a particular case of Theorem \ref{pro:decomposition}. Note that, in this case, $\mu^= = \mu^- = 0$, and the issue of mass allocation at boundary points vanishes.
	\end{remq}

	\begin{remq}
		One could ask about the optimality of our decomposition of $\CCC(\mu,\nu)$. Observe that, if we define  $\mu = \d_0 + 2\d_1$ and $\nu = 2 \d_1 + \d_2$, we have $\CCC(\mu,\nu) = \{\d_{(1,1)}\} \oplus \CCC(\d_0+\d_1, \d_1+\d_2)$ while Equation \eqref{eq:somme_directe_int} rewrites $\CCC(\mu,\nu) = \CCC(\mu,\nu)$. Thus, while the decomposition of $\CCC(\mu,\nu)$ given by Theorem \ref{pro:decomposition} is relevant with respect to the entropic problem with the distance cost, it can still be refined. In a companion paper, we will address with more details the question of the optimality of our decomposition. There we will introduce a refined version of our decomposition. This refinement will serve as the first step for the step for the study of entropic selection problem with the cost $ c^* : (x,y) \in \R^2  \mapsto |y-x| + (+\infty)\1_{x=y} \in [0,+\infty]$. In this paper we will also prove that this refined decomposition and the decomposition associated to Theorem \ref{pro:decomposition} are both optimal in a certain sense.
	\end{remq}

	\subsection{Applications of the decomposition result}\label{subsection:appli_deco}
	
	\subsubsection{Definition of $\KK(\mu,\nu)$ and characterization of $\BB(\mu,\nu)$}

	\begin{defi}[Large strong multiplicativity]\label{defi:strong_multiplicativity_large}
		A measure $\pi \in \MM_+(\R^2)$ is said to be (largely) strongly multiplicative if there exists $\eta_1, \eta_2 \in \MM_+^\s(\R)$ such that $\pi = \left(\eta_1 \otimes \eta_2\right)_{\res \F}.$
	\end{defi}

	For any pair $(\g_1,\g_2)$ defined on the same measurable space, we write $\g_1 \ll \g_2$ if $\g_1$ is absolutely continuous with respect to $\g_2$. In this case, we denote by $\frac{d \g_1}{d\g_2}$ the Radon-Nikodym derivative of $\g_1$ with respect to $\g_2$. If $\g_1 \ll \g_2$ and $\g_2 \ll \g_1$, we say that $\g_1$ and $\g_2$ are equivalent and we write $\g_1 \sim \g_2.$ We can now state the Strassen-type theorem of Kellerer associated to $\leq_\F$ \cite[Theorem $3.6$]{kellerer_order_1986}.
	
	\begin{them}[Strassen-type theorem for $\leq_{\F}$]\label{them:strassen_F}
		Consider $({\g_1},{\g_2}) \in \MM_+(\R)^2.$ The relation ${\g_1} \leq_{\F} {\g_2}$ holds if and only if $\Marg(\g_1,\g_2)$ contains a strongly multiplicative measure, \ie, there exists $\eta_1, \eta_2 \in \MM_+^\s(\R)$ such that 
		\begin{equation}\label{eq:kellerer_appartenance}
			\left(\eta_1 \otimes \eta_2\right)_{\res \F} \in \Marg(\g_1,\g_2).
		\end{equation} 
		In this case, $\Marg(\g_1,\g_2)$ contains exactly one strongly multiplicative measure. Moreover, one can assume $\eta_1 \sim \g_1$ and $\eta_2 \sim \g_2$ in Equation \eqref{eq:kellerer_appartenance}.
	\end{them}
	
	We denote by $\KK_\F(\g_1,\g_2)$ the unique strongly multiplicative measure in $\Marg(\g_1,\g_2)$ when $\g_1 \leq_\F \g_2$ and call it Kellerer transport plan. Note that $\KK_\F(\g_1, \g_2)$ belongs to $\Marg_\F(\g_1,\g_2) = \CCC(\g_1,\g_2).$ If $\g_2 \leq_\F \g_1$, we define $\KK_{\ti{\F}}(\g_1,\g_2) = i_\# \KK_\F(\g_2,\g_1)$, where $i : \R^2 \to \R^2$ is the map defined by $i(x,y) = (y,x)$. In this case, $\KK_{\ti{\F}}(\g_1,g_2)$ belongs to $\Marg_{\ti{\F}}(\g_1,\g_2)$. Our decomposition allows us to extend the definition of Kellerer transport plans to measures that are not in the reinforced large stochastic order.
	
	\begin{defi}[Generalized Kellerer transport plan]\label{defi:gen_kell_transport_plan}
		By Point \ref{point:ordre_dec_part} of Theorem \ref{pro:decomposition} and Theorem \ref{them:strassen_F}, the transport plan $\KK(\mu,\nu) := \sum_{k \in \KK^+} \KK_\F(\mu_k^+,\nu_k^+) + \sum_{k \in \KK^-} \KK_{\ti{\F}}(\mu_k^-,\nu_k^-) + (\id,\id)_\# \mu^=$ is well defined. By Equation \eqref{eq:somme_directe_int}, $\KK(\mu,\nu)$ belongs to $\CCC(\mu,\nu).$
	\end{defi}
	
	The following result is a characterization of barrier points in terms of cumulative distribution functions. Its proof relies on the Strassen-type result of Kellerer for $\leq_{\F}$ and our decomposition. We refer to Definition \ref{defi:composantes} for the definition of $E^=(\mu,\nu)$ and Definition \ref{defi:barrier_points} for the definition of $\BB(\mu,\nu)$. 
	
	\begin{pro}\label{pro:caracterization_barrier}
		Consider $(\mu,\nu) \in \MM^2_+.$ Then $\BB(\mu,\nu) = E^=(\mu,\nu)$.
	\end{pro}

	\begin{proof}
		According to Lemma \ref{remq:preliminary_inclusion}, $E^= \subset \BB$. It remains to prove $\BB \subset E^=$. Suppose, to derive a contradiction, that $\BB \cap E^+$ is non-empty. Then, there exists $x \in \BB$ and $k \in \KK^+$ such that $x \in ]a_k^+,b_k^+[.$ Since $\KK(\mu,\nu) \in \CCC(\mu,\nu)$ and $x \in \BB$, we have $\KK(\mu,\nu)(]-\infty,x[ \times ]x,+\infty[) = 0$. This implies $\KK_\F(\mu_k^+,\nu_k^+)(]-\infty,x[ \times ]x,+\infty[) = 0.$ According to Theorem \ref{them:strassen_F}, there exists two $\sigma$-finite measures $\eta_1, \eta_2 \in \MM_+^\s(\R)$ such that  $\KK_\F(\mu_k^+,\nu_k^+) = (\eta_1 \otimes \eta_2)_{\res \F}$, $\mu_k^+ \sim \eta_1$ and $\nu_k^+ \sim \eta_2.$ Hence, $0 = \KK_\F(\mu_k^+,\nu_k^+)(]-\infty,x[ \times ]x,+\infty[) = (\eta_1 \otimes \eta_2)_{\res \F}(]-\infty,x[ \times ]x,+\infty[) = \eta_1(]-\infty,x[)\eta_2(]x,+\infty[)$, which forces either $\eta_1(]-\infty,x[)=0$ or $\eta_2(]x,+\infty[)=0$. Thus, $\mu_k^+(]-\infty,x[)=0$ or $\nu_k^+(]x,+\infty[)=0$. Moreover, in the proof of Point \ref{point:ordre_compo} of Proposition \ref{pro:prop_deco}, we established  that $]a_k^+,b_k^+[ = ]s_{\mu_k^+},S_{\nu_k^+}[$. By Point \ref{point:infsupsupp} of Remark \ref{remq:leq_F}, this is a contradiction. Hence, $ \BB \cap E^+ = \emptyset$. Similarly, $\BB \cap E^- = \emptyset$. Since $\R = E^+ \uplus E^- \uplus E^=$, we get $\BB \subset (E^+)^c \cap (E^-)^c = E^=.$
	\end{proof}
	
	\begin{remq}\label{remq:computation_E^=}
		As $E^= = (E^+)^c \cap (E^-)^c = \left(\{F_\mu^+ > F_\nu^+\}\cap \{F_\mu^- > F_\nu^-\}\right)^c \cap \left(\{F_\mu^+ < F_\nu^+\}\cap \{F_\mu^- < F_\nu^-\}\right)^c = \{F_\mu^+ = F_\nu^+\} \cup  \{F_\mu^- = F_\nu^-\} \cup \{F_\nu^- < F_\mu^- \leq F_\mu^+ < F_\nu^+\} \cup  \{F_\mu^- < F_\nu^- \leq F_\nu^+ < F_\mu^+\},$ Equation \eqref{eq:E_==_B_modifie} follows from Proposition \ref{pro:caracterization_barrier}.
	\end{remq}

	\subsubsection{Characterization of cyclical monotonicity in terms of crossings}

	\begin{defi}
		We define the set of crossings as $$\CC= \ens{((x_1,y_1),(x_2,y_2)) \in \R^2 \times \R^2}{x_1 \leq x_2 \text{ and } y_2 \leq y_1}.$$ A crossing $((x_1,y_1),(x_2,y_2)) \in \CC$ is said to be free if $y_1 \leq x_1$ or $x_2 \leq y_2.$ Let $\CC_0$ denote the set of free crossing and define $\CC_1 = \CC \setminus \CC_0$. We define the transport plans avoiding non-free crossing as $$\Marg^{\CC_1^c}(\mu,\nu) = \ens{\pi \in \Marg(\mu,\nu)}{\exists \Gamma \in \BB(\R^2), \pi(\Gamma^c)=0 \text{ and } \Gamma^2 \cap \CC_1 = \emptyset}.$$
	\end{defi}

	\begin{remq}\label{remq:crossing}
		\begin{enumerate}
			\item For notational simplicity, we adopt a non-symmetrical definition of crossing, \ie,  a definition that does not include pairs $((x_1,y_1),(x_2,y_2)) \in \R^2 \times \R^2$ such that $x_2 \leq x_1$ and $y_1 \leq y_2$. Had we chosen the symmetrical definition of $\CC$ (and $\CC_0$, $\CC_1$), we would have obtained the same set $\Marg^{\CC_1^c}(\mu,\nu)$.
			\item\label{point:deux_cycl_car} Consider $(x_1,y_1),(x_2,y_2) \in \R^2$ with $x_1 \leq x_2.$ Then, $((x_1,y_1),(x_2,y_2)) \in \CC_1^c$ if and only if $|y_1 - x_1| + |y_2 -x_2| \leq |y_2-x_1| + |y_1-x_2|$. Therefore, $\Marg^{\CC_1^c}(\mu,\nu)$ is the set of transport plan $\pi$ concentrated on a set $\Gamma$ satisfying the following condition:
			\begin{equation}\label{eq:deux_cycl_mon}
				\forall (x_1,y_1),(x_2,y_2) \in \Gamma, |y_1 - x_1| + |y_2 -x_2| \leq |y_2-x_1| + |y_1-x_2|.
			\end{equation}
			As this condition is implied by cyclical monotonicity, we have $\mathfrak{C}(\mu,\nu) \subset \Marg^{\CC_1^c}(\mu,\nu).$
		\end{enumerate}
	\end{remq}
	
	\begin{figure}[h]
		\begin{center}
			\def\svgwidth{11cm}
			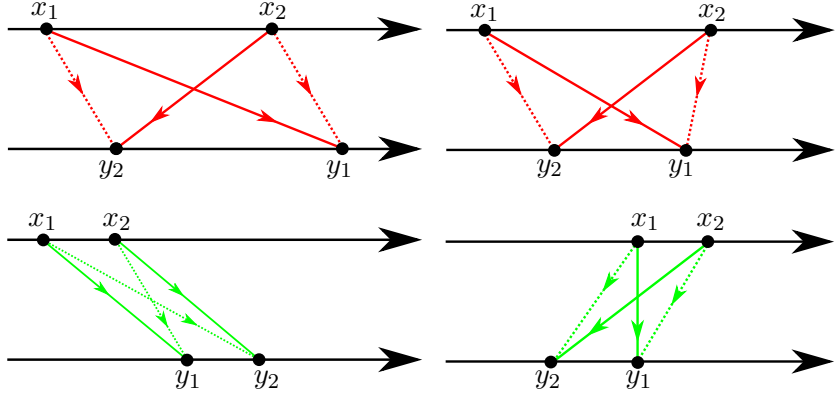
			\caption{Free crossings (in green) and non-free crossings (in red).}
			\label{preuve_dico}
		\end{center}
	\end{figure}
	
	\begin{pro}\label{pro:swapping_lemma}
		Consider $(\mu,\nu) \in \MM_+^2$ 
		and define $\Gamma = \left(\bigcup_{k \in \KK^+} \F \cap A_k^+\right) \cup \left(\bigcup_{k \in \KK^-}  \ti{\F} \cap A_k^- \right) \cup (A^= \cap \D).$ Then
		\begin{equation}
			\mathfrak{C}(\mu,\nu) = \Marg^{\CC_1^c}(\mu,\nu) =  \Marg_{\Gamma}(\mu,\nu).
		\end{equation}
	\end{pro}

	\begin{proof}
		We prove the inclusions $\mathfrak{C}(\mu,\nu) \subset \Marg^{\CC_1^c}(\mu,\nu) \subset \Marg_{\Gamma}(\mu,\nu) \subset \CCC(\mu,\nu).$  The first inclusion corresponds to Point \ref{point:deux_cycl_car} of Remark \ref{remq:crossing}. For the second inclusion, by Remark \ref{remq:just_two}, the proof of Proposition \ref{pro:barrier_points} is only based on the fact that $\pi$ is concentrated on a set satisfying Equation \eqref{eq:deux_cycl_mon}. By Point \ref{point:deux_cycl_car} of Remark \ref{remq:crossing}, we may thus replace $\pi \in \CCC(\mu,\nu)$ with $\pi \in \Marg^{\CC_+^c}(\mu,\nu)$ in the statement of Proposition \ref{pro:barrier_points}. Since the proof of Theorem \ref{pro:decomposition} and its intermediate steps rely on Proposition \ref{pro:barrier_points}, we may also replace \( \CCC(\mu,\nu) \) with the larger set \( \Marg^{\CC_1^c}(\mu,\nu) \) in the statement of Theorem \ref{pro:decomposition}. Thus, every element of $\Marg^{\CC_1^c}(\mu,\nu)$ belongs to $\sum_{k \in \KK^+} \Marg_\F(\mu_k^+,\nu_k^+) + \sum_{k \in \KK^-} \Marg_{\ti{\F}}(\mu_k^-,\nu_k^-) + \{(\id,\id)_\# \mu^=\}$. Therefore, elements of $\Marg^{\CC_1^c}(\mu,\nu)$ are concentrated on $\Gamma$, which proves the second inclusion. The third inclusion is an immediate consequence of Remark \ref{remq:cyclicall_monotonicity_triangle}.
	\end{proof}

	\begin{remq}
		In the case of $L^p$ transport with $p>1$, it is well known that a transport plan is optimal if and only if there exists a set $\Gamma$ such that, for every $(x,y), (x',y') \in \Gamma$, $|y-x|^p+ |y'-x'|^p \leq |y-x'|^p+ |y'-x|^p.$ The first equality of Proposition \ref{pro:swapping_lemma} implies that this condition, sometimes referred as $2$-cyclical monotonicity \cite{pascale_60_2024}, is also equivalent to cyclical monotonicity when $p=1$. In case $p<1$, the equivalence seems to be less clear: we refer the reader to \cite[Section $2.1$]{juillet_solution_2020} for a discussion on the geometric consequences of the previous condition and to \cite[Part 2]{gangbo_geometry_1996} for general informations about optimal transportation for concave costs. 
	\end{remq}

	\section{Convergence of the solutions to the entropically regularized problem}\label{sec:convergence_entro}
	
	Throughout this section, we fix a pair $(\mu,\nu) \in \MM_+^2$ of measures. Given a measurable space $S$ and $(\g_1,\g_2) \in \MM_+(S)^2$, we denote by $\Ent(\g_1 | \g_2)$ the entropy of $\g_1$ relatively to $\g_2$ defined by: 
	\begin{equation*}
		\Ent(\g_1|\g_2) =
		\begin{cases}
			\int_{S} \log\left(\frac{d\g_1}{d\g_2}\right) \frac{\ddd\g_1}{\ddd\g_2} \dd \g_2 &\text{if } \g_1 \ll \g_2 \\+\infty &\text{otherwise }
		\end{cases},
	\end{equation*}	
	where we recall that $\g_1 \ll \g_2$ means that $\g_1$ is absolutely continuous with respect to $\g_2$ and $\frac{d\g_1}{d\g_2}$ stands for the Radon-Nykodim derivative of $\g_1$ with respect to $\g_2$. We recall the $\Ent(\cdot|\g_2)$ is valued in $]-\infty,+ \infty]$, strictly convex on $\MM_+(S)$ and lower semi-continuous. For all $\ee > 0$, let $$J_\ee : \pi \in \Marg(\mu,\nu) \mapsto \int_{\R^2} |y-x| \dd \pi(x,y) + \ee \Ent(\pi|\mu \otimes \nu)$$ denote the transport cost with  regularization parameter $\ee$, and define $$W^\ee_1(\mu,\nu) = \min_{\pi \in \Marg(\mu,\nu)} J_\ee(\pi).$$ For all $\ee >0$, there exists a notion analogue to cyclical monotonicity in the original optimal transport problem. This notion was introduced by Bernton, Ghosal and Nutz in \cite{bernton_entropic_2022} for more general costs and characterizes the solution of the minimization problem associated to $J_\ee$ when $W_1^\ee(\mu,\nu) < + \infty.$
	
	\begin{defi}
		Consider $\ee > 0$ and $\pi \in \Marg(\mu,\nu)$. The transport plan $\pi \in \Marg(\mu,\nu)$ is said to be $\ee$-cyclically invariant if there exists $f : \R^2 \to [0, + \infty[$ such that $\pi = f \cdot \mu \otimes \nu $, and, for all $((x_i,y_i))_{i \in \ent{1}{n}} \in  (\R^2)^n$, $$ \prod_{i=1}^n \exp\left( -\frac{1}{\ee} |y_i-x_i|\right)f(x_i,y_i) = \prod_{i=1}^n \exp\left( -\frac{1}{\ee} |y_{i+1}-x_i|\right)f(x_i,y_{i+1}),$$ where $y_{n+1}$ stands for $y_1.$ 
	\end{defi}
	As shown by the same authors in \cite[Theorem $1.3$]{ghosal_stability_2022}, there exists an unique $\ee$-cyclically invariant transport plan, even when $W_1^\ee(\mu,\nu) = +\infty$. Analogously to Theorem \ref{them:cyclicality_optimality} for the classical transport problem, they proved that it coincides with the unique minimizer of $J_\ee$ when $W_1^\ee(\mu,\nu)$ is finite.
	
	\begin{them}\label{them:well_posedness} 
		Consider $(\mu, \nu) \in \MM_+^2$ and $\ee > 0.$
		\begin{enumerate}
			\item There exists a unique $\ee$-cyclically invariant transport plan from $\mu$ to $\nu.$ In the following, we denote this measure by $\pi_\ee \in \Marg(\mu,\nu).$
			\item In case $W_1^\ee(\mu,\nu)$ is finite, then $J_\ee$ admits a unique minimizer, and this minimizer is equal to $\pi_\ee.$
		\end{enumerate}
	\end{them}
	We now study the behaviour of $(\pi_\ee)_{\ee > 0}$ when $\ee \to 0^+.$ This problem is related to the minimization of $\Ent(\cdot|\mu\otimes \nu)$ among the elements $\CCC(\mu,\nu)$. Indeed, in the case of measures with finite support, every optimal transport plan has finite entropy, and by strict convexity of $\Ent(\cdot|\mu \otimes \nu)$ on $\CCC(\mu,\nu)$, there exists a unique minimizer of $\Ent(\cdot|\mu \otimes \nu)$ among the elements of $\CCC(\mu,\nu).$ In this case, it is well known that 
	$\lim_{\ee \to 0^+} \pi_\ee = \argmin_{\pi \in \CCC(\mu,\nu)} \Ent(\pi|\mu \otimes \nu).$ In the following subsection, we verify that $\KK(\mu,\nu)$ is a minimizer of $\Ent(\cdot|\mu \otimes \nu)$ among $\CCC(\mu,\nu)$ for every pair $(\mu,\nu) \in \MM_+^2.$ After, we establish that, if $W_1(\mu,\nu) < + \infty$ and there exists a optimal transport plan with finite entropy, then $\KK(\mu,\nu) = \lim_{\ee \to 0^+} \pi_\ee.$ Finally, we briefly discuss the convergence result of Di Marino and Louet.
	
	\subsection{$\KK(\mu,\nu)$ minimizes $\Ent(\cdot|\mu \otimes \nu)$ among $\CCC(\mu,\nu)$.}\label{subsection:ent_minimizes}
	
	We begin by we expressing the relative entropy of a cyclically monotone transport plan as the sum of the relative entropies of its components. 	
	\begin{pro}\label{pro:decomposition_entropie}
		Consider $\pi \in \CCC(\mu,\nu)$ and define the family $\ens{\pi_k^+}{k \in \KK^+} \cup \ens{\pi_k^-}{k \in \KK^-} \cup \{\pi^=\}$ as in Definition \ref{defi:composantes}.	Then the entropy decomposes as	\begin{equation}\label{eq:decomposition_entropie}
			\Ent(\pi | \mu \otimes \nu) = \sum_{k \in \KK^+} \Ent({\pi}_k^+ | \mu \otimes \nu) + \sum_{k \in \KK^-} \Ent({\pi}_k^- | \mu \otimes \nu) + \Ent(\pi^=|\mu \otimes \nu).
		\end{equation}
	\end{pro}
	
	\begin{proof}
		According to Theorem \ref{pro:decomposition}, we have $\pi = \sum_{k \in \KK^+} \pi_k^+ + \sum_{k \in \KK^-} \pi_k^- + \pi^=$. Thus, $\pi \ll \mu \otimes \nu$ if and only if 
		\begin{equation}\label{eq:ab_cont_compo}
			\begin{cases*}
				\forall k \in \KK^+, \pi_k^+ \ll \mu \otimes \nu \\
				\forall k \in \KK^+, \pi_k^+ \ll \mu \otimes \nu \\
				\pi^= \ll \mu \otimes \nu
			\end{cases*}.
		\end{equation}
		Thus, if $\pi \ll \mu \otimes \nu$ is not satisfied, then both sides of Equation \eqref{eq:decomposition_entropie} are equal to $+\infty$. Assume now $\pi \ll \mu \otimes \nu$ is satisfied, so that Equation \eqref{eq:ab_cont_compo} holds. Recall that for all $k \in \KK^+$ (resp. $k \in \KK^-$), $\pi_k^+$ (resp. $\pi_k^-$) is concentrated on $A_k^+$ (resp. $A_k^-$), that $\pi^=$ is concentrated on $A^=$, and that $\{A_k^+\}_{k \in \KK^+} \cup \{A_k^-\}_{k \in \KK^-} \cup \{A^=\}$ is a family of disjoint sets. In particular, for all $k \in \KK^+$, $\pi- \pi_k^+ =\sum_{j \in \KK^+ \setminus \{k\}} \pi_j^+ + \sum_{j \in \KK^-} \pi_j^- + \pi^=$ is concentrated on $\left[\biguplus_{j \in \KK^+ \setminus \{k\}} A_j^+\right] \uplus \left[\biguplus_{j \in \KK^-} A_j^-\right] \uplus A^= \subset \left(A_k^+\right)^c.$
		Therefore, one can assume that $\frac{d \pi- \pi_k^+}{\ddd\mu \otimes \nu}$ vanishes on $A_k^+$. Thus, we have
		\begin{equation}
			\int_{\R^2} \log\left( \frac{\ddd \pi }{\ddd \mu \otimes \nu}\right)\dd {\pi}_k^+ = \int_{A_k^+} \log\left( \frac{\ddd {\pi}_k^+ }{\ddd \mu \otimes \nu} + \frac{\ddd \pi - {\pi}_k^+ }{\ddd \mu \otimes \nu}\right)\dd {\pi}_k^+ = \int_{\R^2} \log\left( \frac{\ddd {\pi}_k^+ }{\ddd \mu \otimes \nu} \right)\ddd {\pi}_k^+.
		\end{equation}
		Similarly, for all $k \in \KK^-$, $\int_{\R^2} \log\left( \frac{\ddd \pi }{d \mu \otimes \nu}\right)\dd {\pi}_k^- =  \int_{\R^2} \log\left( \frac{\ddd {\pi}_k^- }{\ddd \mu \otimes \nu} \right)\dd {\pi}_k^-$, and $\int_{\R^2} \log\left( \frac{\ddd \pi }{\ddd \mu \otimes \nu}\right)\dd {\pi}^= =  \int_{\R^2} \log\left( \frac{\ddd {\pi}^= }{\ddd \mu \otimes \nu} \right)\dd {\pi}^=$. Finally, 
		\begin{align*}
			\Ent(\pi|\mu \otimes \nu) &= \sum_{k \in \KK^+} \int_{\R^2} \log\left( \frac{\ddd \pi }{\ddd \mu \otimes \nu}\right)\dd {\pi}_k^+ + \sum_{k \in \KK^-} \int_{\R^2} \log\left( \frac{\ddd \pi }{\ddd \mu \otimes \nu}\right)\dd {\pi}_k^- +  \int_{\R^2} \log\left( \frac{\ddd \pi }{\ddd \mu \otimes \nu}\right) \dd {\pi}^=\\
			&= \sum_{k \in \KK^+} \int_{\R^2} \log\left( \frac{\ddd \pi_k^+ }{\ddd \mu \otimes \nu}\right)\dd {\pi}_k^+ + \sum_{k \in \KK^-} \int_{\R^2} \log\left( \frac{\ddd \pi_k^- }{\ddd \mu \otimes \nu}\right)\dd {\pi}_k^- +  \int_{\R^2} \log\left( \frac{d \pi^= }{d \mu \otimes \nu}\right) \dd {\pi}^=\\
			&= \sum_{k \in \KK^+} \Ent(\pi_k^+|\mu \otimes \nu) + \sum_{k \in \KK^-} \Ent({\pi}_k^- | \mu \otimes \nu) + \Ent(\pi^=|\mu \otimes \nu). \qedhere
		\end{align*}
	\end{proof}

	Now, to prove that $\KK(\mu,\nu)$ minimizes $\Ent(\cdot|\mu \otimes \nu)$ over $\CCC(\mu_k^+,\nu_k^+)$, it suffices to show that, for every $k \in \KK^+$, $\KK(\mu,\nu)_{\res A_k^+} = \KK(\mu_k^+,\nu_k^+)$ is a minimizer of $\Ent(\cdot|\mu\otimes \nu)$ over $\CCC(\mu,\nu)$ (and similarly if $k \in \KK^-$). Before establishing this result in Proposition \ref{lem:minimum_kell_ord}, we state the following approximation result. For a proof, we refer to \cite[Lemma $5.1$]{di_marino_entropic_2018}
	
	\begin{lem}[\cite{di_marino_entropic_2018}, Lemma 5.1]\label{lem:approx_troncated}
		Consider a measure $\pi \in \MM_+(\R^2)$ and two Borel functions $a,b : \R^2 \to \R$. For all $n \geq 1$, define $\varphi_n : t \in \R \mapsto \max(-n,\min(n,t))$. If $(a+b)_- \in L^1(\pi)$, then $\int_{\R^2} \varphi_n(a) + \varphi_n(b) \dd\pi \dcv{n \to + \infty} \int_{\R^2} a + b \dd\pi.$
	\end{lem}
	
	The following result shows that for measures $\theta_1,\theta_2 \in \MM_+(\R)$ such that $\theta_1 \leq_\F \theta_2$ (see Definition \ref{def:Kellerer_order_large}), the Kellerer transport plan $\KK_\F(\theta_1,\theta_2)$ (defined after Theorem \ref{them:strassen_F}) minimizes  $\Ent(\cdot|\theta_1 \otimes \theta_2)$ among the cyclically monotone transport plans. In fact the statement is more general: $\KK_\F(\theta_1,\theta_2)$  minimizes $\Ent(\cdot|\g_1 \otimes \g_2)$ for every pair $(\g_1,\g_2) \in \MM_2(\R)^2$ satisfying $\theta_1 \ll \g_1$ and $\theta_2 \ll \g_2$. This has already been proven by Di Marino and Louet in a more restricted context\footnote{Precisely, they ask for $\g_1$ and $\g_2$ to be compactly supported and atomless. Their optimal transport minimizing entropy is not defined using the Strassen Theorem from Kellerer: in fact, they explicitly construct a strongly multiplicative transport plan, which by the uniqueness part of Theorem \ref{them:strassen_F}, is the Kellerer transport plan. Their proof of the lemma does not use their specific construction, but only relies on the product form of the Kellerer transport plan.}, but their proof holds in our more general setting. For ease of reading and precaution, we reproduce their proof.
	
	\begin{pro}\label{lem:minimum_kell_ord}
		Consider $(\g_1, \g_2) \in \MM_+(\R)^2$ and $\theta_1, \theta_2 \in \MM_+(\R)$ such that $\theta_1 \ll \g_1, \theta_2 \ll \g_2$ and $\theta_1 \leq_{\F} \theta_2$. Then $$\KK_\F(\theta_1,\theta_2) \in \argmin_{\pi \in \CCC(\theta_1,\theta_2)} \Ent(\pi|\g_1 \otimes \g_2).$$
	\end{pro}
	
	\begin{proof}
		Consider $\pi \in \CCC(\theta_1,\theta_2)$. If $\pi \ll \g_1 \otimes \g_2$ is not satisfied, then $\Ent(\pi|\g_1 \otimes \g_2) = + \infty \geq \Ent(\KK_\F(\theta_1,\theta_2)|\g_1 \otimes \g_2)$. We now assume $\pi \ll \g_1 \otimes \g_2$, and define $\pi_* = \frac{\ddd \pi}{d\g_1 \otimes \g_2}$. Let $\nu_1, \nu_2 \in \MM_+^\s(\R)$ be such that $\nu_1 \ll \theta_1$, $\nu_2 \ll \theta_2$ and $\KK(\g_1,\g_2) = (\nu_1 \otimes \nu_2)_{\res \F}$. We have
		$$\frac{d\KK_\F(\theta_1,\theta_2)}{d\g_1 \otimes \g_2} = \frac{d\KK_\F(\theta_1,\theta_2)}{d\theta_1 \otimes \theta_2}\frac{d\theta_1 \otimes \theta_2}{d\g_1 \otimes \g_2} = \left(\1_\F \frac{\ddd\nu_1}{d\theta_1}\otimes\frac{\ddd\nu_2}{d\theta_2} \right)\left( \frac{d\theta_1}{d\g_1}\otimes\frac{d\theta_2}{d\g_2}\right) = \1_\F f \otimes g,$$
		where $f$ and $g$ are defined by $f = \frac{\ddd\nu_1}{d{\g_1}}$ and $g = \frac{\ddd\nu_2}{d\g_2}$. Now, define $a = \log(f)$ and $b = \log(g).$ For all $n \geq 1$, we define $a_n := \varphi_n(a)$ and $b_n := \varphi_n(b)$. The reader may verify that, for all $(c,t) \in \R \times \R_+$, $t \log(t) - tc \geq -e^{c-1}  $. For all $(x,y,n) \in \R^2\times \N^*$, by applying this inequality with $c = a_n(x) + b_n(y) +1$ and $t = \pi_*(x,y)$, we get $\pi_*(x,y) \log(\pi_*(x,y)) - \pi_*(x,y)(a_n(x) + b_n(y) +1) \geq -e^{a_n(x) + b_n(y)}.$ Thus, as $\CCC(\theta_1,\theta_2) = \Marg_\F(\theta_1,\theta_2)$, for all $n \geq 1$:
		\begin{align*}
			\Ent(\pi|{\g_1} \otimes {\g_2}) &= \int_{\F} \pi_*(x,y) \log(\pi_*(x,y)) \dd {\g_1} \otimes {\g_2}(x,y) \\
			&=\int_{\F} \pi_*(x,y) \log(\pi_*(x,y)) - \pi_*(x,y)(a_n(x) + b_n(y) +1)\dd {\g_1} \otimes {\g_2}(x,y) \\  &\hspace{2cm}+\int_{\F}\pi_*(x,y)(a_n(x) + b_n(y) +1)\dd {\g_1} \otimes {\g_2}(x,y) \\
			&\geq - \int_{\F} e^{a_n(x)+b_n(y)} \dd{\g_1} \otimes {\g_2}(x,y) + \int_{\R^2}(a_n(x) + b_n(y) +1) \dd \pi(x,y) \\
			&= - \int_{\R^2} e^{a_n(x)+b_n(y)} \dd {\g_1} \otimes {\g_2}(x,y) + \int_{\R^2}(a_n(x) + b_n(y) +1)\dd \KK_\F({\theta_1},{\theta_2})(x,y),
		\end{align*}
		where we used that $\KK_\F({\theta_1},{\theta_2})$ and $\pi$ both belong to $\Marg({\theta_1},{\theta_2})$ to establish the last equality. Since for all $(z_1,z_2) \in \R^2$, $(\varphi_n(z_1) + \varphi_n(z_2))_+ \leq (z_1 + z_2)_+$, for all $(x,y) \in \R^2$, $0 \leq e^{a_n(x) + b_n(y)} \leq e^{(a(x) + b(y))_+} \leq \max(1,f(x)g(y)) \leq 1 + f(x)g(y).$ As $\int_{\F} 1 + f(x)g(y) \dd {\g_1} \otimes {\g_2}(x,y) \leq ({\g_1} \otimes {\g_2})(\R^2) + \KK_\F({\theta_1},{\theta_2})(\R^2) < + \infty$ and the pointwise limits $a = \lim_{n \to + \infty} a_n $ and $b = \lim_{n \to + \infty} b_n$ hold, by dominated convergence $$\int_{\F}  e^{a_n(x)+b_n(y)} \dd{\g_1} \otimes {\g_2}(x,y) \dcv{n \to + \infty} \int_{\F} e^{a(x)+b(y)} \dd{\g_1} \otimes {\g_2}(x,y) = \KK_\F({\theta_1},{\theta_2})(\R^2).$$ As $$\int_{\R^2} (a(x) + b(y))_- \dd \KK_\F({\theta_1},{\theta_2}) = \int_{\F} (f(x)g(y) \log(f(x)g(y)))_- \dd({\g_1} \otimes {\g_2})(x,y) \leq \int_{\R^2} 1 \dd{\g_1}\otimes {\g_2} = {\g_1} \otimes {\g_2} (\R^2),$$ Lemma \ref{lem:approx_troncated} applies. Thus, $$\int_{\R^2} a_n(x) + b_n(y) \dd\KK_\F({\theta_1},{\theta_2})(x,y) \dcv{n \to + \infty} \int_{\R^2} a(x) + b(y) \dd \KK_\F({\theta_1},{\theta_2})(x,y).$$ Since $\int_{\R^2} a(x) + b(y) \dd \KK_\F({\theta_1},{\theta_2})(x,y)= \int_{\F} \log(f(x)g(y)) \dd\KK_\F({\theta_1},{\theta_2})(x,y) = \Ent(\KK_\F({\theta_1},{\theta_2})| {\g_1} \otimes {\g_2})$, by gathering all our estimates together, we conclude $\Ent(\pi | {\g_1} \otimes {\g_2}) \geq -\KK_\F({\theta_1},{\theta_2})(\R^2) + \Ent(\KK_\F({\theta_1},{\theta_2}) | {\g_1} \otimes {\g_2}) + \KK_\F({\theta_1},{\theta_2})(\R^2) = \Ent(\KK_\F({\theta_1},{\theta_2}) | {\g_1} \otimes {\g_2})$, which proves the proposition.
	\end{proof}

	We now prove that $\KK(\mu,\nu)$ minimizes $\Ent(\cdot|\mu \otimes \nu)$ among $\CCC(\mu,\nu).$
	
	\begin{them}\label{them:kellerer_minimizes}
		Consider $(\mu,\nu) \in \MM_+^2$ and define $\KK(\mu,\nu)$ as in Definition \ref{defi:gen_kell_transport_plan}. Then 
		\begin{equation}\label{eq:Kell_min}
			\Ent(\KK(\mu,\nu)|\mu \otimes \nu) = \min_{\pi \in \CCC(\mu,\nu)}\Ent(\pi|\mu \otimes \nu).
		\end{equation}
	\end{them}
	
	\begin{proof}
		Consider $\pi \in \CCC(\mu,\nu)$, define $\FF = ((\pi_k^+)_{k \in \KK^+},(\pi_k^-)_{k \in \KK^-}, \pi^=)$ using the notation in Definition \ref{defi:composantes}, and $\PPP$ as in Theorem \ref{pro:decomposition}. As $\FF \in \PPP$, by Proposition \ref{pro:decomposition_entropie} and Proposition \ref{lem:minimum_kell_ord}, we obtain:
		\begin{align*}
			\Ent(\pi|\mu \otimes \nu) &= \sum_{k \in \KK^+} \Ent(\pi_k^+|\mu \otimes \nu) + \sum_{k \in \KK^-} \Ent(\pi_k^-|\mu \otimes \nu) + \Ent(\pi^=|\mu \otimes \nu) \\
			&\geq  \sum_{k \in \KK^+} \Ent(\KK_\F(\mu_k^+,\nu_k^+)|\mu \otimes \nu) + \sum_{k \in \KK^-} \Ent(\KK_{\ti{\F}}(\mu_k^-,\nu_k^-)|\mu \otimes \nu) + \Ent(\pi^=|\mu \otimes \nu)\\
			&= \Ent(\KK(\mu,\nu)|\mu \otimes \nu).\qedhere
		\end{align*} 
	\end{proof}

	\subsection{Convergence of $(\pi_\ee)_{\ee > 0}$ in two different settings.}\label{subsection:conv_connue}
	
	In this part, we first prove that $\lim_{\ee \to 0^+} \pi_\ee = \KK(\mu,\nu)$ when $W_1(\mu,\nu) < +\infty$ and $\Ent(\pi|\mu \otimes \nu) < + \infty$. Then, we briefly recall the result of Di Marino and Louet in the case of atomless marginals, when there exists $\pi \in \CCC(\mu,\nu)$ such that $\Ent(\pi-(\id,\id)_\# \mu^=|\mu \otimes \nu)$ is finite. In this case, under some additional technical assumptions, Di Marino and Louet proved (using involved Gamma-convergence methods) that $(\pi_\ee)_{\ee >0}$ converges to $\KK(\mu,\nu).$ In another part (Theorem \eqref{them:convergence_semi_discret}), we will prove a new convergence result. This result is by no mean generalizing the two precedent result, but is a new case of convergence. Under the hypotheses that non-equal components form pairs of mutually singular measures, we prove that $\KK(\mu,\nu) = \lim_{\ee \to 0^+} \pi_\ee.$ Observe that we do not require the existence of an optimal transport plan with finite entropy, nor --- as in the case of Di Marino and Louet --- the existence of an optimal transport plan whose non-fixed part has finite entropy.
	
	\subsubsection{Convergence under the existence of an optimal transport plan with finite entropy}

	In case of measures with finite support, it is well known that $\Ent(\cdot|\mu\otimes \nu)$ admits a unique minimizer $\pi^*$ among $\OO(\mu,\nu)$, and that $\pi^* = \lim_{\ee \to 0^+} \pi_\ee.$ According to Theorem \ref{them:kellerer_minimizes}, we immediately get $\KK(\mu,\nu) = \lim_{\ee \to 0^+} \pi_\ee.$ We now extend this convergence result to a more general setting, namely when:
	\begin{equation}\label{eq:finiteness_entropy}
		W_1(\mu,\nu) < + \infty \text{\hspace{0.2cm} and \hspace{0.2cm}}\inf_{\pi \in \CCC(\mu,\nu)} \Ent(\pi|\mu \otimes \nu) < + \infty.
	\end{equation}
	
	Note that under Condition \eqref{eq:finiteness_entropy}, Theorem \ref{them:cyclicality_optimality} ensures that $\CCC(\mu,\nu) = \OO(\mu,\nu).$ The proof follows a similar path as the one in the finite support case (which can be found in the monograph \cite[Proposition $4.1$]{peyre_computational_2019}), but an additional result is required to guarantee that cluster points of $(\pi_\ee)_{\ee > 0}$ are elements of $\CCC(\mu,\nu).$ This result has been proven in \cite[Proposition $3.2$]{bernton_entropic_2022} for any continuous cost function.
	
	\begin{pro}\label{them:valeur_adherence_solutions_penalisees_generales}
		Cluster points of $(\pi_\ee)_{\ee>0}$ at point $0^+$ are  cyclically monotone transport plans.
	\end{pro} 
	\begin{them}\label{them:min_kell}
		Consider $(\mu,\nu) \in \MM_+^2$ and assume Condition \eqref{eq:finiteness_entropy} holds. Then ${\KK}(\mu,\nu) = \lim_{\ee \to 0^+} \pi_\ee.$
	\end{them}

	\begin{proof}
		Let $\pi^*$ be a cluster point of $(\pi_\ee)_{\ee>0}$. By Proposition \ref{them:valeur_adherence_solutions_penalisees_generales}, $\pi^* \in \CCC(\mu,\nu).$  Fix $\ee > 0$ and $\pi \in \CCC(\mu,\nu).$ As $J(\pi_\ee) + \ee\Ent(\pi_\ee|\mu \otimes \nu) =J_\ee(\pi_\ee)  \leq J_\ee(\pi) = J(\pi) + \ee\Ent(\pi|\mu \otimes \nu)$ and $J(\pi) \leq J(\pi_\ee)$, we get
		\begin{equation*}\label{eq:convergence_in_support_fini}
			\ee\Ent(\pi_\ee | \mu \otimes \nu) \leq J(\pi_\ee) +   \ee \Ent(\pi_\ee | \mu \otimes \nu)  - J(\pi) \leq  J(\pi) + \ee \Ent(\pi | \mu \otimes \nu) - J(\pi) = \ee\Ent(\pi | \mu \otimes \nu).
		\end{equation*} 
		By lower semi-continuity of $\Ent(\cdot|\mu \otimes \nu)$, we get $\Ent(\pi^*|\mu \otimes \nu) \leq \Ent(\pi|\mu \otimes \nu).$ Hence, $\pi^*$ minimizes  $\Ent(\cdot|\mu \otimes \nu)$ over $\CCC(\mu,\nu)$. By the strict convexity of $\Ent(\cdot | \mu \otimes \nu)$ on $\CCC(\mu,\nu)$ and Condition \eqref{eq:finiteness_entropy}, we know that $\Ent(\cdot | \mu \otimes \nu)$ has a unique minimizer among elements of $\CCC(\mu,\nu)$. By Theorem \ref{them:kellerer_minimizes} this minimizer is $\KK(\mu,\nu)$. Therefore $\pi^* = \KK(\mu,\nu)$. This establishes the convergence result.
	\end{proof}
	
	Observe that, when $\mu$ and $\nu$ have finite support, Condition \eqref{eq:finiteness_entropy} is satisfied, so that $\KK(\mu,\nu) = \lim_{\ee \to 0^+}\pi_\ee.$ In case of atomic measures (not necessarily with finite support), Condition \eqref{eq:finiteness_entropy} may be satisfied or not. For instance, if  for some atomic measure $\theta$, we have $\mu =\nu = \theta$, then $\CCC(\mu,\nu) = \OO(\mu,\nu) = \{(\id,\id)_\#\theta \}$, and $\Ent((\id,\id)_\#\theta | \mu \otimes \nu) = -\sum_{x \in \spt(\theta)} \theta(x) \log(\theta(x)).$ Thus, for $\theta = \sum_{n \geq 2} n^{-2} \d_{n}$ Condition \eqref{eq:finiteness_entropy} is satisfied, whereas for $\theta = \sum_{n \geq 2} (n\log(n)^2)^{-1}\d_n$ Condition \eqref{eq:finiteness_entropy} is not satisfied.

	\subsubsection{The Di Marino--Louet case}\label{subsubsection:Dl_results}
	
	In \cite[Theorem $4.1$]{di_marino_entropic_2018}, Di Marino and Louet proved the following convergence result of $(\pi_\ee)_{\ee>0}$. 
	\begin{them}\label{them:conv_Di_Marino_Louet}
		Assume $\mu, \nu \in \PP(\R)$ satisfy the following conditions:
		\begin{enumerate}
			\item The measures $\mu, \nu$ have compact support and satisfy $\Ent(\mu|\LL^1)< + \infty$ and $\Ent(\nu|\LL^1) < +\infty$;
			\item The set  $\spt(\mu) \cap \{F_\mu^+ = F_\nu^+\}$ has a Lebesgue-negligible boundary and the family  $(]a_i,b_i[)_{i \in \II}$ of connected components of its interior satisfies: $$-\sum_{i \in \II} \int_{a_i}^{b_i} \log(\min(x-a_i,b_i-x)) \dd\mu(x) < + \infty;$$
			\item  $\min_{\pi \in \OO(\mu,\nu)} \Ent\left(\pi_{\res \{F_\mu^+ = F_\nu^+\}^c \times  \{F_\mu^+ = F_\nu^+\}^c}|\mu \otimes \nu\right)<+\infty$.
		\end{enumerate}
		Then $\lim_{\ee \to + \infty} \pi_\ee = \argmin_{\OO(\mu,\nu)} \Ent\left(\pi_{\res \{F_\mu^+ = F_\nu^+\}^c \times  \{F_\mu^+ = F_\nu^+\}^c}|\mu \otimes \nu\right).$
	\end{them}
	
	Observe that the first assumption implies that $\mu$ and $\nu$ are atomless. We briefly give the idea of the proof of this result to highlight the difference with our own convergence result (under different assumptions on $\mu$ and $\nu$, see Theorem \ref{them:convergence_semi_discret}). A classical method to establish the convergence of $(\pi_\ee)_{\ee>0}$ is to prove that $(J_\ee)_{\ee>0}$ converges to $J$ in the sense of Gamma-convergence, which requires a technique called block approximation. \footnote{For more information on Gamma-convergence, we e.g. refer to \cite{braides_gamma-convergence_2002}. For more details on block approximation, we refer the reader to \cite{carlier_convergence_2017}.} Using a classical result on Gamma-convergence, we obtain that cluster points of $(\pi_\ee)_{\ee >0}$ converge to minimizers of $J$, that is toward optimal transport plans. When $\OO(\mu,\nu)$ is not a singleton, this is not sufficient to prove that $\pi_\ee$ converges. In this case, one can prove that the functional $(H_\ee)_{\ee>0}$ defined by $H_\ee : \pi \in \Marg(\mu,\nu) \to \frac{1}{\ee}(J(\pi)-W_1(\mu,\nu)) + \Ent(\pi|\mu \otimes \nu)$ $\Gamma$-converges to $$H: \pi \in \Marg(\mu,\nu) \mapsto 
	\begin{cases}
		\Ent(\pi|\mu \otimes \nu) & \text{if } \pi \in \OO(\mu,\nu)\\
		+\infty & \text{otherwise}
	\end{cases}.$$
	If there exists $\pi \in \OO(\mu,\nu)$ such that $\Ent(\pi|\mu \otimes \nu)$ if finite, this $\Gamma$-convergence results proves that $(\pi_\ee)_{\ee > 0}$ converges toward the unique minimizer of $\Ent(\cdot|\mu \otimes \nu)$ among $\OO(\mu,\nu)$. In the $L^1$ case, as $H \equiv + \infty$ in general, this $\Gamma$-convergence result does not provide any information on the convergence of $\pi_\ee.$ The idea of Di Marino and Louet is to consider the refined family of functional $(F_\ee)_{\ee > 0}$ defined as $F_\ee = H_\ee + \log(2\ee) \mu(\{F_\mu^+ = F_\nu^+\})$, and prove that this family $\Gamma$-converges toward the functional $F$ defined as 
	\begin{equation*}
		F: \pi \in \Marg(\mu,\nu) \mapsto 
		\begin{cases}
			\Ent\left(\pi_{\res \{F_\mu^+ = F_\nu^+\}^c \times  \{F_\mu^+ = F_\nu^+\}^c}|\mu \otimes \nu\right) + \Ent(\mu_{\res \{F_\mu^+ = F_\nu^+\}}|\LL^1) & \text{if } \pi \in \OO(\mu,\nu) \\
			+\infty & \text{otherwise}
		\end{cases}.
	\end{equation*}
	The last hypothesis of Theorem \ref{them:conv_Di_Marino_Louet} is crucial, as it allows to go from the convergence of $(F_\ee)_{\ee >0}$ toward $F$ to the convergence of $(\pi_\ee)_{\ee>0}$. Next, they establish the existence of a strongly multiplicative transport plan on each component \cite[Proposition 5.1]{di_marino_entropic_2018}, which corresponds to Theorem \ref{them:strassen_F} for atomless measure with compact support. They finally prove that the minimizer of $F$ is a strongly multiplicative measure on each component. With our own notation, their result rewrites $\KK(\mu, \nu) = \argmin_{\OO(\mu,\nu)} \Ent\left(\pi_{\res \{F_\mu^+ = F_\nu^+\}^c \times  \{F_\mu^+ = F_\nu^+\}^c}|\mu \otimes \nu\right)$: in particular $\pi_\ee \dcv{\ee \to 0^+} \KK(\mu,\nu)$. In their article Di Marino and Louet provide a necessary and sufficient condition to the third assumption \cite[Theorem $1.1$]{di_marino_entropic_2018}, which is satisfied if and only if $-\int_{\{F_\mu^+ \neq F_\nu^+\}}\log(|F_\mu^+(x)-F_\nu^+(x)|)\dd\mu(x)$ is finite. This condition is not always satisfied, for instance when $\mu = \1_{[0,1]} \cdot \LL^1$ and $\nu$ is defined by the relation 
	\begin{equation*}
		F_\nu^+(x) = 
		\begin{cases}
			0 & \text{ if } x \leq 0 \\
			(1- 2e^{-2})x &\text{ if } x \in [0,1/2] \\
			x-e^{-\frac{1}{1-x}} & \text{ if } x \in [1/2,1]
		\end{cases},
	\end{equation*}
	we have $-\int_{\{F_\mu^+ \neq F_\nu^+\}}\log(|F_\mu^+(x)-F_\nu^+(x)|)\dd\mu(x) = - \int_0^1 \log(F_\mu^+(x)-F_\nu^+(x))\dd x \geq \int_{1/2}^1 \frac{1}{1-x}\dd x = +\infty.$
	
	\subsection{Properties of cluster points of $(\pi_\ee)_{\ee >0}$}\label{subsection:general_property}

	In this section, we show that restriction of any cluster point of $(\pi_\ee)_{\ee > 0}$ to a product set included in $\F$ or $\ti{\F}$ is a product measure. This property, referred to as (large) weak multiplicativity, is a symmetrized version of the notion of weak $\F$-multiplicativity introduced by Kellerer \cite[Definition $5.1$]{kellerer_order_1986}. First, let $\CC$ be the class $\ens{B_1 \times B_2}{(B_1,B_2) \in \BB(\R)^2,  B_1 \times B_2 \subset \F \text{ or }  B_1 \times B_2 \subset \ti{\F}}$ of product sets included in $\F$ or $\ti{\F}.$ 
	
	\begin{defi}[Large weak multiplicativity]\label{defi:weak_multiplicativity}
		A measure $\pi \in \MM_+(\R^2)$ is said to be (largely) weakly multiplicative if, for all $A \in \CC$, $\pi_{\res A}$ is a product measure.
	\end{defi}
	
	Clearly, strongly multiplicative measure (see Definition \ref{defi:strong_multiplicativity_large}) are weakly multiplicative. The converse is not true, as $\pi = \d_{(0,0)} + \d_{(1,1)}$ is weakly multiplicative without being strongly multiplicative. We begin by showing that weak multiplicativity is preserved under weak convergence. Then, we use a structure result on the measures $\pi_\ee$ to prove that cluster points of $\pi_{\ee}$ at $0^+$ are weakly multiplicative.
	
	\subsubsection{Stability of weak multiplicativity}
	
	The following characterization is straightforward and left to the reader.
	
	\begin{lem}\label{pro:caracterisation_produit}
		A measure $\pi \in \MM_+(\R^2)$ is a product measure if and only if $\pi(\R^2) \cdot \pi = ({p_1}_{\#}\pi) \otimes ({p_2}_{\#}\pi).$
	\end{lem}

	The following lemma provides a useful characterization of weak multiplicativity. For every $(a,b) \in \R^2$ such that $b<a$, we shall use the convention $[a,b] = \emptyset.$
	
	\begin{pro}\label{pro:caracterisation_SWFM}
		Consider $\pi \in \Marg(\mu,\nu)$ and let $\CC'$ be the class of subsets of $\R^2$ defined by $$\CC' = \ens{]-\infty,t] \times [t,+\infty[}{t \in \R} \cup \ens{ [t,+\infty[ \times ]-\infty,t]}{t \in \R}.$$ The following statements are equivalent:
		\begin{enumerate}
			\item $\pi$ is weakly multiplicative; 
			\item The restriction of $\pi$ to any element of $\CC'$ is a product measure;
			\item For all $(x,y,t) \in \R^3$: 
			\begin{equation}\label{eq:cara_weak_F}
				\pi(]-\infty,t] \times [t,+\infty[) \pi(]-\infty,t \wedge x] \times [t,y]) = \pi(]-\infty,t \wedge x] \times [t,+\infty[)\pi(]-\infty,t] \times [t,y])
			\end{equation}
			and
			\begin{equation}\label{eq:cara_weak_tiF}
				\pi([t,+\infty[\times ]-\infty,t])\pi([t,x] \times ]-\infty, t \wedge y]) = \pi([t,x] \times ]-\infty,t]) \pi([t,+\infty[ \times ]-\infty, t\wedge y]).
			\end{equation}
		\end{enumerate}
	\end{pro}

	\begin{proof}
		The equivalence between the first two statements was established by Kellerer in \cite[Theorem 5.2]{kellerer_order_1986}. We now prove the equivalence between the second and the third statements. For every $x \in \R$, define $I_x = ]-\infty,x]$, and let us fix $t \in \R$. According to Proposition \ref{pro:caracterisation_produit}, $\pi_{]-\infty,t]\times [t,+ \infty[}$ is a product measure if and only if $$\pi_{]-\infty,t]\times [t,+ \infty[}(\R^2) \pi_{]-\infty,t]\times [t,+ \infty[} = ({p_1}_{\#}\pi_{]-\infty,t]\times [t,+ \infty[}) \otimes ({p_2}_{\#}\pi_{]-\infty,t]\times [t,+ \infty[}),$$ that is, for all $(x,y) \in \R^2$,
		$$\pi_{]-\infty,t]\times [t,+ \infty[}(\R^2) \pi_{]-\infty,t]\times [t,+ \infty[}(I_x \times I_y) = ({p_1}_{\#}\pi_{]-\infty,t]\times [t,+ \infty[}) \otimes ({p_2}_{\#}\pi_{]-\infty,t]\times [t,+ \infty[})(I_x \times I_y),$$
		which means Equation \eqref{eq:cara_weak_F} holds. Similarly, $\pi_{[t,+\infty[ \times ]-\infty,t]}$ is a product measure if and only if, for all $(x,y) \in \R^2,$ Equation \eqref{eq:cara_weak_tiF} is satisfied. The equivalence between the second and the third point is now straightforward.
	\end{proof}
	To prove that weak multiplicativity is stable with respect to weak convergence, we shall prove that the third formulation is stable with respect to weak convergence. This is achieved by applying a variant of the Portmanteau theorem adapted to transport plans. For a proof, we refer to \cite[Proposition/Notation $3.7$]{boubel_markov-quantile_2022}.
	\begin{lem}\label{lem:weak_con_marg}
		Consider $(\pi_n)_{n \geq 1} \in \Marg(\mu,\nu)^{\N^*}$ and $\pi \in \Marg(\mu,\nu).$ The following conditions are equivalent:
		\begin{enumerate}
			\item $(\pi_n)_{n \geq 1}$ weakly converges to $\pi$;
			\item For every pair $(I_1,I_2)$ of intervals, $\pi_n(I_1 \times I_2) \dcv{n \to + \infty} \pi(I_1 \times I_2).$
		\end{enumerate}
	\end{lem}
	
	We can now prove that the set of  weakly multiplicative transport plans from $\mu$ to $\nu$ is closed under weak convergence.
	
	\begin{pro}\label{cor:stabilite_SWFM}
		Consider $(\pi_n)_{n \geq 1} \in \Marg(\mu,\nu)^{\N^*}$ and $\pi \in \Marg(\mu,\nu).$ If the measures $(\pi_n)_{n \geq 1}$ are  weakly multiplicative and $(\pi_n)_{n \geq 1}$ converges weakly to $\pi$, then $\pi$ is  weakly multiplicative.
	\end{pro}
	
	\begin{proof}
		According to Proposition \ref{pro:caracterisation_SWFM}, for all $(x,y,t,n) \in \R^3 \times \N^*$,  $$\pi_n(]-\infty,t] \times [t,+\infty[)\pi_n(]-\infty,t \wedge x] \times [t,y]) = \pi_n(]-\infty,t \wedge x] \times [t,+\infty[)\pi_n(]-\infty,t] \times [t,y])$$ and $$ \pi_n([t,+\infty[\times ]-\infty,t])\pi_n([t,x] \times ]-\infty, t \wedge y]) = \pi_n([t,x] \times ]-\infty,t]) \pi_n([t,+\infty[ \times ]-\infty, t\wedge y]).$$ By Lemma \ref{lem:weak_con_marg}, we get  $$\pi(]-\infty,t \wedge x] \times [t,y])\pi(]-\infty,t] \times [t,+\infty[) = \pi(]-\infty,t \wedge x] \times [t,+\infty[)\pi(]-\infty,t] \times [t,y])$$ and $$\pi([t,+\infty[\times ]-\infty,t])\pi([t,x] \times ]-\infty, t \wedge y]) = \pi([t,x] \times ]-\infty,t]) \pi([t,+\infty[ \times ]-\infty, t\wedge y]).$$ From Proposition \ref{pro:caracterisation_SWFM}, it follows that $\pi$ is  weakly multiplicative.
	\end{proof}
	
	\subsubsection{Weak multiplicativity of cluster points of $(\pi_\ee)_{\ee>0}$}
	
	The following lemma provides us with information on the global form of $\ee$-cyclically invariant transport plans. Its proof is relatively straightforward and can be found in \cite[Lemma $2.7$]{nutz_introduction_nodate} by taking $R(dx,dy) = \exp(-|y-x|/\ee)\mu(dx)\nu(dy).$
	
	\begin{lem}\label{them:structure_entropique}
		Consider $\pi \in \Marg(\mu,\nu)$ and $\ee > 0.$ The following conditions are equivalent:
		\begin{enumerate}
			\item $\pi$ is $\ee$-cyclically monotone;
			\item There exists two measurable maps $\phi_\ee : \R \to \R$ and $\psi_\ee : \R \to \R$ such that $$ \frac{\ddd \pi}{\ddd\mu \otimes \nu}(x,y) = \exp \left( \frac{\phi_\ee(x) + \psi_\ee(y)-|y-x|}{\ee} \right) \hspace{1cm} (\mu \otimes \nu)(dx,dy)-a.s. \ .$$
		\end{enumerate}
	\end{lem}

	\begin{them}\label{them:valeur_adherence_solutions_penalisees_generales_SWFM}
		Consider $(\mu,\nu) \in \MM_+^2$. If $\pi^*$ is a cluster point of $(\pi_\ee)_{\ee>0}$, then $\pi^*$ is a weakly multiplicative measure. From Proposition \ref{them:valeur_adherence_solutions_penalisees_generales}, it follows that cluster points of $(\pi_\ee)_{\ee> 0}$ are weakly multiplicative elements of $\CCC(\mu,\nu)$.
	\end{them}
	
	\begin{proof}
		Consider $t \in \R$ and $\ee > 0$. By Lemma \ref{them:structure_entropique}, there exists $\phi_\ee : \R \to \R$ and $\psi_\ee : \R \to \R$ measurable such that $$ \pi_\ee(dx,dy) = \exp\left(\frac{\phi_\ee(x) + \psi_\ee(y) - |y - x|}{\ee} \right) \mu(dx) \nu(dy).$$ Thus, for every $t \in \R$, $$\1_{]-\infty,t]\times [t,+\infty[} \pi_\ee = (f_1^t \mu) \otimes (f_2^t \nu)$$ and $$\1_{[t,+\infty[ \times ]-\infty,t]} \pi_\ee = (f_3^t \mu) \otimes (f_4^t \nu),$$ where the maps  $f_1^t,f_2^t,f_3^t,f_4^t : \R \to \R$ are defined by
		\begin{align*}
			f_1^t(x) = \1_{]-\infty,t]}(x)\exp\left(\frac{\phi_\ee(x)+x}{\ee}\right)\\
			f_2^t(y) = \1_{[t,+\infty[}(y)\exp\left(\frac{\psi_\ee(y)-y}{\ee}\right)\\
			f_3^t(x) = \1_{[t,+\infty[}(x)\exp\left(\frac{\phi_\ee(x)-x}{\ee}\right)\\
			f_4^t(y) = \1_{]-\infty,t]}(y)\exp\left(\frac{\psi_\ee(y)+y}{\ee}\right).
		\end{align*}
		
		From Lemma \ref{pro:caracterisation_SWFM}, it follows that $\pi_\ee$ is  weakly multiplicative. By Proposition \ref{cor:stabilite_SWFM}, every cluster point of $(\pi_\ee)_{\ee>0}$ is weakly multiplicative.
	\end{proof}

	\subsection{Convergence of entropic minimizer when the marginal components are mutually singular}\label{subsection:new_convergence}
	
	To handle convergence in presence of mutually singular measures, we introduce the notions of strict reinforced stochastic order, strict strong multiplicativity, and strict weak multiplicativity. The definition of this notions and their relations are analogues, for the strict half-planes $\G = {(x,y) \in \R^2 : y > x}$ and $\ti{\G} = {(x,y) \in \R^2 : y < x}$, to (large) reinforced stochastic order (Definition \ref{def:Kellerer_order_large}), (large) strong multiplicativity (Definition \ref{defi:strong_multiplicativity_large}), and (large) weak multiplicativity (Definition \ref{defi:weak_multiplicativity}).
	
	\begin{defi}[Strict strong multiplicativity]\label{defi:strong_multiplicativity_strict}
		A measure $\pi \in \MM_+(\R^2)$ is called strictly strongly multiplicative if there exists $\eta_1, \eta_2 \in \MM_+^\s(\R)$ such that $\pi = \left(\eta_1 \otimes \eta_2\right)_{\res \G}.$
	\end{defi}
	
	We now introduce the order $\leq_\G$, that will be related to strict strong multiplicativity by its associated Strassen theorem in the following.
	
	\begin{defi}[Strict reinforced stochastic order]\label{def:Kellerer_order_strict}
		Consider $(\g_1,\g_2) \in \MM_+(\R)^2$ and define 
		\begin{equation*}
			T_*(\g_1,\g_2) = \ens{t \in \R}{F_{\g_1}^-(t) >0 \text{ and } F_{\g_2}^+(t) < \g_2(\R)}.
		\end{equation*}
		We say that $\g_1$ is smaller than $\g_2$ in the reinforced strict stochastic order if $F_{\g_1}^- \geq F_{\g_2}^+$ and $T_*(\g_1,\g_2) \subset \{F_{\g_1}^- > F_{\g_2}^+\}$. In this case, we write $\g_1 \leq_\G \g_2$.\footnote{This notation is motivated by the associated Strassen-type theorem.}
	\end{defi}
	
	Similarly to Theorem \ref{them:strassen_F} for the order $\leq_\F$, the following theorem from Kellerer \cite[Theorem $2.4$]{kellerer_order_1986} provides a Strassen-type result for  $\leq_{\G}$.
	
	\begin{them}[Strassen-type theorem for $\leq_{\G}$]\label{them:strassen_G}
		Consider $({\g_1},{\g_2}) \in \MM_+(\R)^2.$ The relation ${\g_1} \leq_{\G} {\g_2}$ is satisfied if and only if $\Marg(\g_1,\g_2)$ contains a stricly strongly multiplicative measure, \ie, there exists $\eta_1, \eta_2 \in \MM_+^\s(\R)$ such that 
		\begin{equation}\label{eq:kellerer_appartenance}
			\left(\eta_1 \otimes \eta_2\right)_{\res \G} \in \Marg(\g_1,\g_2).
		\end{equation} 
		In this case, $\Marg(\g_1,\g_2)$ contains exactly one strictly strongly multiplicative measure.
	\end{them}
	
	We denote by $\KK_\G(\g_1,\g_2)$ the unique strongly multiplicative measure in $\Marg(\g_1,\g_2)$ when $\g_1 \leq_\G \g_2$. Similarly as for the definition of $\CC$ (Definition \ref{defi:weak_multiplicativity}), let $\DD$ denote the class  $$\ens{B_1 \times B_2}{(B_1,B_2) \in \BB(\R)^2,  B_1 \times B_2 \subset \G \text{ or }  B_1 \times B_2 \subset \ti{\G}}$$ of product sets included in $\G$ or $\ti{\G}.$
	
	\begin{defi}[Strict weak multiplicativity]
		A measure $\pi \in \MM_+(\R^2)$ is said to be strictly weakly multiplicative if, for all $A \in \DD$, $\pi_{\res A}$ is a product measure.
	\end{defi}
	
	The following result states that for measures $\g_1$, $\g_2$ in the strict reinforced stochastic order, $\KK_\G(\g_1,\g_2)$ is the only weakly multiplicative measure in $\Marg_\G(\g_1,\g_2).$ There is no similar result for the analogue large properties: when $\g_1 \leq_F \g_2$, in general, $\KK_\F(\g_1,\g_2)$ is not the only weakly multiplicative element of $\Marg_\F(\g_1,\g_2)$. For the proof, we refer to \cite[Proposition $4.3$]{kellerer_order_1986}
	
	\begin{lem}\label{lem:WGM_implique_SGM}
		Consider $(\g_1,\g_2) \in \MM_+(\R)^2$ such that $\g_1 \leq_\G \g_2.$ If $\pi \in \Marg_\G(\g_1,\g_2)$ and $\pi$ is weakly multiplicative, then $\pi = \KK_\G(\g_1,\g_2).$
	\end{lem}

	\begin{pro}\label{pro:pas_atomes_communs}
		Assume $(\g_1,\g_2) \in \MM_+(\R)^2$ have no atoms in common and $\g_1 \leq_{\F} \g_2$. Then, $\g_1 \leq_\G \g_2$ and $\KK_\F (\mu,\nu) = \KK_\G (\mu,\nu).$
	\end{pro}
	
	\begin{proof}
		For all $t \in \R$, as $\g_1$ and $\g_2$ have no atom in common, $\g_1(t) = 0$ or $\g_2(t) = 0.$ Therefore, $F_{\g_1}^-(t) - F_{\g_2}^+(t) \in \{F_{\g_1}^+(t) - F_{\g_2}^+(t),F_{\g_1}^-(t) - F_{\g_2}^-(t)\}$, and the inequality  $F_{\g_1}^-(t) - F_{\g_2}^+(t) \leq \min\big(F_{\g_1}^+(t) - F_{\g_2}^+(t), F_{\g_1}^-(t) - F_{\g_2}^-(t)\big)$ is an equality.  As $\g_1 \leq_{\st} \g_2$, we get $F_{\g_1}^-(t) - F_{\g_2}^+(t) = \min\left(F_{\g_1}^+(t) - F_{\g_2}^+(t), F_{\g_1}^-(t) - F_{\g_2}^-(t)\right) \geq 0.$  Furthermore, the sets $T^*(\g_1,\g_2)$ (Definition \ref{def:Kellerer_order_strict}), $T_+(\g_1,\g_2)$ and $T_-(\g_1,\g_2)$ (Definition \ref{def:Kellerer_order_large}) satisfy: $T_*(\g_1, \g_2) \subset T_+(\g_1, \g_2) \cap \ T_-(\g_1, \g_2) \subset \{F_{\g_1}^+ > F_{\g_2}^+\} \cap \{F_{\g_1}^- > F_{\g_2}^-\}.$ As $ \{F_{\g_1}^+ > F_{\g_2}^+\} \cap \{F_{\g_1}^- > F_{\g_2}^-\} = \{\min\left(F_{\g_1}^+ - F_{\g_2}^+, F_{\g_1}^- - F_{\g_2}^-\right) >0\}$, and, as we just saw, $F_{\g_1}^- - F_{\g_2}^+ = \min\left(F_{\g_1}^+ - F_{\g_2}^+, F_{\g_1}^- - F_{\g_2}^-\right)$, we get $ T_*(\g_1, \g_2) \subset \{ F_{\g_1}^- > F_{\g_2}^+(t)\}.$ Therefore, $\g_1 \leq_\G \g_2.$ Now, as $\g_1 \leq_\F \g_2$, by Theorem \ref{them:strassen_F}, there exists two  $\sigma$-finite measures $\nu_1 \ll \g_1$ and $\nu_2 \ll \g_2$ such that $\KK_\F(\mu,\nu) = (\nu_1 \otimes \nu_2)_{\res \F}$. As $\At(\nu_1)\cap \At(\nu_2) \subset \At(\g_1) \cap \At(\g_2) = \emptyset$ and $\At(\nu_1)$ is countable, we have $\nu_2(\At(\nu_1)) = 0$. Since $x \mapsto \nu_1(\{x\})$ is null outside $\At(\nu_1)$, we get $$(\nu_1 \otimes \nu_2)(\D) = \int_{\R} \left( \int_{\R} \1_{x=y}  \dd\nu_1(x) \right) \dd\nu_2(y) = \int_{\At(\nu_1)^c} \nu_1(y)  \dd\nu_2(y) = 0.$$ Thus, $\KK_\F(\mu,\nu) = (\nu_1 \otimes \nu_2)_{\res \G} + (\nu_1 \otimes \nu_2)_{\res \D} = (\nu_1 \otimes \nu_2)_{\res \G}$ belongs to $\Marg_\G(\g_1, \g_2)$ and is the restriction of a product measure to $\G$. This establishes $\KK_\F(\g_1,\g_2) = \KK_\G(\g_1,\g_2).$ 
	\end{proof}

	Before stating that $\KK(\mu,\nu) = \lim_{\ee \to 0^+} \pi_\ee$ when the components of non-equal pairs of $\DD_\KK$ are singular, we prove that two measures $(\g_1,\g_2) \in \MM_+^2$ forms a pair of singular measures if and only if there exists no transport plan from $\g_1$ to $\g_2$ that fixes mass. To prove this equivalence, we shall use that two measures are singular if and only if their common part is null. Let us recall the definition of common part of two measure. For every $(x,y) \in \R^2$, let $x \wedge y$ denote the minimum between $x$ and $y$. Now, if we define $h_1 = \frac{d\g_1}{d(\g_1+\g_2)}$ and $h_2 = \frac{d\g_2}{d (\g_1 + \g_2)}$, the common part $\g_1 \wedge \g_2$  between $\g_1$ and $\g_2$ is defined by $\g_1 \wedge \g_2 = (h_1 \wedge h_2) \cdot (\g_1 + \g_2)$.
	
	\begin{pro}\label{pro:equi_reduction}
		For all $\g_1, \g_2 \in \MM_+^2$, the following conditions are equivalent:
		\begin{enumerate}
			\item $\g_1 \wedge \g_2 = 0$;
			\item $(\g_1,\g_2)$ is a pair of singular measures;
			\item For all $\pi \in \CCC(\g_1,\g_2)$, $\pi(\D)=0.$
		\end{enumerate}
	\end{pro}
	
	\begin{proof}
		The equivalence between the two first points is classical and omitted. We just prove the implications $(2)\implies (3)$ and $(3) \implies (1)$, beginning  with $(2) \implies (3)$. Consider $E$ such that $\g_1(E) = \g_2(E^c) = 0$  and $\pi \in \CCC(\g_1,\g_2).$  As ${p_1}_\# \pi_{\res \D} \leq \g_1$, ${p_1}_\# \pi_{\res \D} = {p_2}_\# \pi_{\res \D} \leq \g_2$ and $\D = [\D \cap (E \times \R)] \uplus [\D \cap (E^c \cap \R)]$, we have $0 \leq \pi(\D) = \left({p_1}_\# \pi_{\res \D}\right)(E) + \left({p_1}_\# \pi_{\res \D}\right)(E^c) \leq \g_1(E) + \g_2(E^c) = 0 + 0 = 0$. We now prove $(3) \implies (1).$ To derive a contradiction, assume $\g_1 \wedge \g_2 \neq 0$. Then, consider $\pi_0 \in \CCC(\g_1 - \g_1 \wedge \g_2, \g_2 - \g_1 \wedge \g_2)$, and define $\pi = \pi_0 + (\id,\id)_\# \g_1 \wedge \g_2.$ By construction $\pi \in \Marg(\g_1,\g_2)$ and, from Point \ref{point:cycl_D_cup_g} of Lemma \ref{lem:cycl_union}, $\pi$ is cyclically monotone. Hence, $\pi \in \CCC(\g_1,\g_2)$, and $\pi(\D) \geq [(\id,\id)_\# \g_1 \wedge \g_2] (\D) =  (\g_1 \wedge \g_2) (\R) > 0$. This is a contradiction, hence $3 \implies 1.$
	\end{proof}
	
	Recall the notation for $\{(\mu_k^+,\nu_k^+)\}_{k \in \KK^+} \cup \{(\mu_k^-,\nu_k^-)\}_{k \in \KK^-} \cup \{(\mu^=,\mu^=)\}$ introduced in Definition \ref{def:component_marginals}. We now show that under the hypothesis,
	\begin{equation}\label{hyp1}
		\begin{cases*}
			\forall k \in \KK^+, \mu_k^+ \wedge \nu_k^+ = 0\\
			\forall k \in \KK^-, \mu_k^- \wedge \nu_k^- = 0
		\end{cases*},
	\end{equation}
	we have $\KK(\mu,\nu) = \lim_{\ee \to 0^+} \pi_\ee$. Note that Condition \eqref{hyp1} implies the following condition,	
	\begin{equation}\label{hyp2}
		\begin{cases*}
			\forall k \in \KK^+, \At(\mu_k^+) \cap \At(\nu_k^+) = 0\\
			\forall k \in \KK^-, \At(\mu_k^-) \cap \At(\nu_k^-) = 0
		\end{cases*},
	\end{equation}
	which will allow us to apply Proposition \ref{pro:pas_atomes_communs}.
	The assumption  $\mu \wedge \nu = 0$ is sufficient condition for Assumption \eqref{hyp2}, but is not necessary. For instance, $\mu = \d_1 + \d_3 + \d_4$ and $\nu = \d_2 + \d_3 + \d_5$ satisfy Assumption \eqref{hyp2}, but we do not have $\mu \wedge \nu = 0.$

	\begin{them}\label{them:convergence_semi_discret}
		Consider $(\mu,\nu) \in \MM_+^2$. If Assumption \eqref{hyp1} is satisfied, then $$\pi_\ee \dcv{\ee \to 0^+} \KK(\mu,\nu).$$ 
	\end{them}

	\begin{proof}
		Consider a cluster point $\pi$ of $(\pi_\ee)_{\ee>0}.$  According to Theorem \ref{them:valeur_adherence_solutions_penalisees_generales_SWFM}, $\pi$ is weakly multiplicative and belongs to $\CCC(\mu,\nu)$. We now define $\FF = ((\pi_k^+)_{k \in \KK^+}, (\pi_k^-)_{k \in \KK^-}, \pi^=)$ using the notation in Definition \ref{defi:composantes} and $\PPP$ as in Theorem \ref{pro:decomposition}. From this theorem, it follows that $\FF \in \PPP$ and
		$\pi = \sum_{k \in \KK^+} \pi_k^+ +  \sum_{k \in \KK^-} \pi_k^- + (\id,\id)_\# \mu^=.$
		Fix $k \in \KK^+$. As Equation \eqref{hyp2} is satisfied and $\mu_k^+ \leq_\F \nu_k^+$, from Proposition \ref{pro:pas_atomes_communs}, it follows that $\mu_k^+ \leq_\G \nu_k^+$ and $\KK_\F(\mu_k^+,\nu_k^+) = \KK_\G(\mu_k^+,\nu_k^+)$. Moreover, as Equation \eqref{hyp1} is satisfied and $\pi \in \CCC(\mu_k^+,\nu_k^+)$, by Proposition \ref{pro:equi_reduction}, $\pi_k^+(\D) = 0$. As $\CCC(\mu_k^+,\nu_k^+) = \Marg_{\F}(\mu_k^+,\nu_k^+)$, $\pi_k^+ \in \Marg_\G(\mu_k^+,\nu_k^+)$. Now, consider $B_1 \times B_2 \in \CC$. As $\pi$ is weakly multiplicative, there exists $\nu_1, \nu_2 \in \MM^+(\R)$ such that ${\pi}_{\res B_1 \times B_2} = (\nu_1 \otimes \nu_2)$. Thus, ${\pi_k^+}_{\res B_1 \times B_2} =   (\1_{[a_k^+,b_k^+[} \cdot \nu_1)\otimes (\1_{]a_k^+,b_k^+]}\cdot\nu_2).$ Therefore $\pi_k^+$ is weakly multiplicative. As $\mu_k^+ \leq_\G \nu_k^+$ and $\pi_k^+ \in \Marg_\G(\mu_k^+,\nu_k^+)$, $\pi_k^+ = \KK_\G(\mu_k^+,\nu_k^+) = \KK_\F(\mu_k^+,\nu_k^+).$  Similarly, for all $k \in \KK^-$, $\pi_k^- = i_\#\KK_\F(\nu_k^-,\mu_k^-).$ Hence $\pi = \sum_{k \in \KK^+} \KK_\F(\mu_k^+,\nu_k^+) + \sum_{k \in \KK^-} i_\#\KK_\F(\nu_k^-,\mu_k^-) + (\id,\id)_\# \mu^= = \KK(\mu,\nu).$ Therefore $(\pi_\ee)_{\ee>0}$ admits a unique cluster point $\pi$. As $\Marg(\mu,\nu)$ is compact, this establishes our result.
	\end{proof}

	\noindent\textbf{\large{Acknowledgment.}} I would  like to express my deep gratitude to Nicolas Juillet for introducing me to this research problem and for his constant support and valuable suggestions throughout the writing process.
	
	\noindent\textbf{\large{Copyright notice}:}
	\begin{minipage}{1.8cm}
		\includegraphics[width=1.8cm]{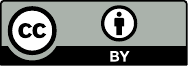}
	\end{minipage} 
	\vspace{0.1cm} This work was supported by the Austrian Science Fund (FWF), grant 10.55776/PIN1806324, and by the
	Agence nationale de la recherche (ANR), grant ANR-23-CE40-0017. A CC-BY public copyright license
	applies to this document and to all subsequent versions up to the Author Accepted Manuscript,
	in accordance with open-access requirements.

	\vspace{0.5cm}
	\begin{minipage}{18cm}
		\emph{Armand Ley} -- Universität Wien\\
		Oskar-Morgenstern-Platz 1, 1090 Vienna, Austria\\
		Email : \texttt{armand.ley@univie.ac.at}
	\end{minipage}

\end{document}

%% file: flee.pdf_tex
%% Creator: Inkscape 1.2.1 (9c6d41e410, 2022-07-14), www.inkscape.org
%% Accompanies image file 'flee.pdf' (pdf, eps, ps)
\begingroup%
  \makeatletter%
  \providecommand\color[2][]{%
    \errmessage{(Inkscape) Color is used for the text in Inkscape, but the package 'color.sty' is not loaded}%
    \renewcommand\color[2][]{}%
  }%
  \providecommand\transparent[1]{%
    \errmessage{(Inkscape) Transparency is used (non-zero) for the text in Inkscape, but the package 'transparent.sty' is not loaded}%
    \renewcommand\transparent[1]{}%
  }%
  \providecommand\rotatebox[2]{#2}%
  \newcommand*\fsize{\dimexpr\f@size pt\relax}%
  \newcommand*\lineheight[1]{\fontsize{\fsize}{#1\fsize}\selectfont}%
  \ifx\svgwidth\undefined%
    \setlength{\unitlength}{92.079755bp}%
    \ifx\svgscale\undefined%
      \relax%
    \else%
      \setlength{\unitlength}{\unitlength * \real{\svgscale}}%
    \fi%
  \else%
    \setlength{\unitlength}{\svgwidth}%
  \fi%
  \global\let\svgwidth\undefined%
  \global\let\svgscale\undefined%
  \makeatother%
  \begin{picture}(1,0.81257465)%
    \lineheight{1}%
    \setlength\tabcolsep{0pt}%
    \put(0,0){\includegraphics[width=\unitlength,page=1]{flee.pdf}}%
    \put(0.66210747,0.52326568){\color[rgb]{0,0,0}\makebox(0,0)[lt]{\lineheight{1.25}\smash{\begin{tabular}[t]{l}$1$\end{tabular}}}}%
    \put(0,0){\includegraphics[width=\unitlength,page=2]{flee.pdf}}%
    \put(0.83387993,0.30951448){\color[rgb]{0,0,0}\makebox(0,0)[lt]{\lineheight{1.25}\smash{\begin{tabular}[t]{l}$\mu$\end{tabular}}}}%
    \put(0.8306627,0.15000263){\color[rgb]{0,0,0}\makebox(0,0)[lt]{\lineheight{1.25}\smash{\begin{tabular}[t]{l}$\nu$\\\end{tabular}}}}%
  \end{picture}%
\endgroup%

%% file: barriere_plan.pdf_tex
%% Creator: Inkscape 1.2.1 (9c6d41e410, 2022-07-14), www.inkscape.org
%% Accompanies image file 'barriere_plan.pdf' (pdf, eps, ps)
\begingroup%
  \makeatletter%
  \providecommand\color[2][]{%
    \errmessage{(Inkscape) Color is used for the text in Inkscape, but the package 'color.sty' is not loaded}%
    \renewcommand\color[2][]{}%
  }%
  \providecommand\transparent[1]{%
    \errmessage{(Inkscape) Transparency is used (non-zero) for the text in Inkscape, but the package 'transparent.sty' is not loaded}%
    \renewcommand\transparent[1]{}%
  }%
  \providecommand\rotatebox[2]{#2}%
  \newcommand*\fsize{\dimexpr\f@size pt\relax}%
  \newcommand*\lineheight[1]{\fontsize{\fsize}{#1\fsize}\selectfont}%
  \ifx\svgwidth\undefined%
    \setlength{\unitlength}{354.22359045bp}%
    \ifx\svgscale\undefined%
      \relax%
    \else%
      \setlength{\unitlength}{\unitlength * \real{\svgscale}}%
    \fi%
  \else%
    \setlength{\unitlength}{\svgwidth}%
  \fi%
  \global\let\svgwidth\undefined%
  \global\let\svgscale\undefined%
  \makeatother%
  \begin{picture}(1,0.71466609)%
    \lineheight{1}%
    \setlength\tabcolsep{0pt}%
    \put(0,0){\includegraphics[width=\unitlength,page=1]{barriere_plan.pdf}}%
    \put(0.45549604,0.28799542){\color[rgb]{0,0,0}\makebox(0,0)[lt]{\lineheight{1.25}\smash{\begin{tabular}[t]{l}\tiny{$(x,x)$}\end{tabular}}}}%
  \end{picture}%
\endgroup%

%% file: aire_conc.pdf_tex
%% Creator: Inkscape 1.2.1 (9c6d41e410, 2022-07-14), www.inkscape.org
%% Accompanies image file 'aire_conc.pdf' (pdf, eps, ps)
\begingroup%
  \makeatletter%
  \providecommand\color[2][]{%
    \errmessage{(Inkscape) Color is used for the text in Inkscape, but the package 'color.sty' is not loaded}%
    \renewcommand\color[2][]{}%
  }%
  \providecommand\transparent[1]{%
    \errmessage{(Inkscape) Transparency is used (non-zero) for the text in Inkscape, but the package 'transparent.sty' is not loaded}%
    \renewcommand\transparent[1]{}%
  }%
  \providecommand\rotatebox[2]{#2}%
  \newcommand*\fsize{\dimexpr\f@size pt\relax}%
  \newcommand*\lineheight[1]{\fontsize{\fsize}{#1\fsize}\selectfont}%
  \ifx\svgwidth\undefined%
    \setlength{\unitlength}{78.10747531bp}%
    \ifx\svgscale\undefined%
      \relax%
    \else%
      \setlength{\unitlength}{\unitlength * \real{\svgscale}}%
    \fi%
  \else%
    \setlength{\unitlength}{\svgwidth}%
  \fi%
  \global\let\svgwidth\undefined%
  \global\let\svgscale\undefined%
  \makeatother%
  \begin{picture}(1,1.00570372)%
    \lineheight{1}%
    \setlength\tabcolsep{0pt}%
    \put(0,0){\includegraphics[width=\unitlength,page=1]{aire_conc.pdf}}%
    \put(0.77189454,0.00550037){\color[rgb]{0,0,0}\makebox(0,0)[lt]{\lineheight{1.25}\smash{\begin{tabular}[t]{l}$1$\end{tabular}}}}%
    \put(-0.00587971,0.77367287){\color[rgb]{0,0,0}\makebox(0,0)[lt]{\lineheight{1.25}\smash{\begin{tabular}[t]{l}$1$\end{tabular}}}}%
  \end{picture}%
\endgroup%

%% file: split.pdf_tex
%% Creator: Inkscape 1.2.1 (9c6d41e410, 2022-07-14), www.inkscape.org
%% Accompanies image file 'split.pdf' (pdf, eps, ps)
\begingroup%
  \makeatletter%
  \providecommand\color[2][]{%
    \errmessage{(Inkscape) Color is used for the text in Inkscape, but the package 'color.sty' is not loaded}%
    \renewcommand\color[2][]{}%
  }%
  \providecommand\transparent[1]{%
    \errmessage{(Inkscape) Transparency is used (non-zero) for the text in Inkscape, but the package 'transparent.sty' is not loaded}%
    \renewcommand\transparent[1]{}%
  }%
  \providecommand\rotatebox[2]{#2}%
  \newcommand*\fsize{\dimexpr\f@size pt\relax}%
  \newcommand*\lineheight[1]{\fontsize{\fsize}{#1\fsize}\selectfont}%
  \ifx\svgwidth\undefined%
    \setlength{\unitlength}{83.52787181bp}%
    \ifx\svgscale\undefined%
      \relax%
    \else%
      \setlength{\unitlength}{\unitlength * \real{\svgscale}}%
    \fi%
  \else%
    \setlength{\unitlength}{\svgwidth}%
  \fi%
  \global\let\svgwidth\undefined%
  \global\let\svgscale\undefined%
  \makeatother%
  \begin{picture}(1,0.95253241)%
    \lineheight{1}%
    \setlength\tabcolsep{0pt}%
    \put(0,0){\includegraphics[width=\unitlength,page=1]{split.pdf}}%
    \put(0.42346181,0.29002334){\color[rgb]{0,0,0}\makebox(0,0)[lt]{\lineheight{1.25}\smash{\begin{tabular}[t]{l}\tiny{$\nu_1^+(2)$}\end{tabular}}}}%
    \put(0.23490194,0.35287628){\color[rgb]{0,0,0}\makebox(0,0)[lt]{\lineheight{1.25}\smash{\begin{tabular}[t]{l}\tiny{$\nu^=(2)=\mu^=(2)$}\end{tabular}}}}%
    \put(0.34265051,0.41573){\color[rgb]{0,0,0}\makebox(0,0)[lt]{\lineheight{1.25}\smash{\begin{tabular}[t]{l}\tiny {$\nu_1^-(2)$}\end{tabular}}}}%
    \put(0,0){\includegraphics[width=\unitlength,page=2]{split.pdf}}%
  \end{picture}%
\endgroup%

%% file: Configuration_interdite.pdf_tex
%% Creator: Inkscape 1.2.1 (9c6d41e410, 2022-07-14), www.inkscape.org
%% Accompanies image file 'Configuration interdite.pdf' (pdf, eps, ps)
\begingroup%
  \makeatletter%
  \providecommand\color[2][]{%
    \errmessage{(Inkscape) Color is used for the text in Inkscape, but the package 'color.sty' is not loaded}%
    \renewcommand\color[2][]{}%
  }%
  \providecommand\transparent[1]{%
    \errmessage{(Inkscape) Transparency is used (non-zero) for the text in Inkscape, but the package 'transparent.sty' is not loaded}%
    \renewcommand\transparent[1]{}%
  }%
  \providecommand\rotatebox[2]{#2}%
  \newcommand*\fsize{\dimexpr\f@size pt\relax}%
  \newcommand*\lineheight[1]{\fontsize{\fsize}{#1\fsize}\selectfont}%
  \ifx\svgwidth\undefined%
    \setlength{\unitlength}{258.70832905bp}%
    \ifx\svgscale\undefined%
      \relax%
    \else%
      \setlength{\unitlength}{\unitlength * \real{\svgscale}}%
    \fi%
  \else%
    \setlength{\unitlength}{\svgwidth}%
  \fi%
  \global\let\svgwidth\undefined%
  \global\let\svgscale\undefined%
  \makeatother%
  \begin{picture}(1,0.47198133)%
    \lineheight{1}%
    \setlength\tabcolsep{0pt}%
    \put(0,0){\includegraphics[width=\unitlength,page=1]{Configuration_interdite.pdf}}%
    \put(0.02852325,0.45260319){\color[rgb]{0,0,0}\makebox(0,0)[lt]{\lineheight{1.25}\smash{\begin{tabular}[t]{l}$\tiny{x_1}$\end{tabular}}}}%
    \put(0.30392989,0.45260319){\color[rgb]{0,0,0}\makebox(0,0)[lt]{\lineheight{1.25}\smash{\begin{tabular}[t]{l}$\tiny{x_2}$\end{tabular}}}}%
    \put(0.5595104,0.44972452){\color[rgb]{0,0,0}\makebox(0,0)[lt]{\lineheight{1.25}\smash{\begin{tabular}[t]{l}$\tiny{x_1}$\end{tabular}}}}%
    \put(0.83491713,0.44972452){\color[rgb]{0,0,0}\makebox(0,0)[lt]{\lineheight{1.25}\smash{\begin{tabular}[t]{l}$\tiny{x_2}$\end{tabular}}}}%
    \put(0.1101627,0.26128849){\color[rgb]{0,0,0}\makebox(0,0)[lt]{\lineheight{1.25}\smash{\begin{tabular}[t]{l}$\tiny{y_2}$\end{tabular}}}}%
    \put(0.38556934,0.26128849){\color[rgb]{0,0,0}\makebox(0,0)[lt]{\lineheight{1.25}\smash{\begin{tabular}[t]{l}$\tiny{y_1}$\end{tabular}}}}%
    \put(0.6394548,0.26220335){\color[rgb]{0,0,0}\makebox(0,0)[lt]{\lineheight{1.25}\smash{\begin{tabular}[t]{l}$\tiny{y_2}$\end{tabular}}}}%
    \put(0.79890088,0.26220335){\color[rgb]{0,0,0}\makebox(0,0)[lt]{\lineheight{1.25}\smash{\begin{tabular}[t]{l}$\tiny{y_1}$\end{tabular}}}}%
    \put(0,0){\includegraphics[width=\unitlength,page=2]{Configuration_interdite.pdf}}%
    \put(0.02533011,0.19677188){\color[rgb]{0,0,0}\makebox(0,0)[lt]{\lineheight{1.25}\smash{\begin{tabular}[t]{l}$\tiny{x_1}$\end{tabular}}}}%
    \put(0.20264989,0.00635181){\color[rgb]{0,0,0}\makebox(0,0)[lt]{\lineheight{1.25}\smash{\begin{tabular}[t]{l}$\tiny{y_1}$\end{tabular}}}}%
    \put(0.1151997,0.19677188){\color[rgb]{0,0,0}\makebox(0,0)[lt]{\lineheight{1.25}\smash{\begin{tabular}[t]{l}$\tiny{x_2}$\end{tabular}}}}%
    \put(0.29783777,0.00635181){\color[rgb]{0,0,0}\makebox(0,0)[lt]{\lineheight{1.25}\smash{\begin{tabular}[t]{l}$\tiny{y_2}$\end{tabular}}}}%
    \put(0.63750253,0.00635181){\color[rgb]{0,0,0}\makebox(0,0)[lt]{\lineheight{1.25}\smash{\begin{tabular}[t]{l}$\tiny{y_2}$\end{tabular}}}}%
    \put(0.74718556,0.00635181){\color[rgb]{0,0,0}\makebox(0,0)[lt]{\lineheight{1.25}\smash{\begin{tabular}[t]{l}$\tiny{y_1}$\end{tabular}}}}%
    \put(0.75346318,0.19677188){\color[rgb]{0,0,0}\makebox(0,0)[lt]{\lineheight{1.25}\smash{\begin{tabular}[t]{l}$\tiny{x_1}$\end{tabular}}}}%
    \put(0.83292791,0.19677188){\color[rgb]{0,0,0}\makebox(0,0)[lt]{\lineheight{1.25}\smash{\begin{tabular}[t]{l}$\tiny{x_2}$\end{tabular}}}}%
  \end{picture}%
\endgroup%